\documentclass[11pt,a4paper,leqno]{article}
\usepackage{a4wide}
\usepackage{latexsym}
\usepackage{amsmath}
\usepackage{amssymb}

\usepackage{graphicx}
\usepackage{color}
\newtheorem{defin}{Definition}
\newtheorem{lemma}{Lemma}
\newtheorem{prop}{Proposition}
\newtheorem{theo}{Theorem}
\newtheorem{corol}{Corollary}
\newenvironment{proof}{\medskip\par\noindent{\bf Proof}}{\hfill $\Box$
\medskip\par}

\newcommand{\C}{\mathbb{C}}
\newcommand{\R}{\mathbb{R}}
\newcommand{\Z}{\mathbb{Z}}

\begin{document}
\title{On parametric multilevel $q-$Gevrey asymptotics for some linear Cauchy problem}
\author{{\bf A. Lastra\footnote{The author is partially supported by the project MTM2012-31439 of Ministerio de Ciencia e
Innovacion, Spain}, S. Malek}\\
University of Alcal\'{a}, Departamento de F\'{i}sica y Matem\'{a}ticas,\\
Ap. de Correos 20, E-28871 Alcal\'{a} de Henares (Madrid), Spain,\\
University of Lille 1, Laboratoire Paul Painlev\'e,\\
59655 Villeneuve d'Ascq cedex, France,\\
{\tt alberto.lastra@uah.es}\\
{\tt Stephane.Malek@math.univ-lille1.fr }}
\date{}
\maketitle
\thispagestyle{empty}
{ \small \begin{center}
{\bf Abstract}
\end{center}

We study a linear $q-$difference-differential Cauchy problem, under the action of a perturbation parameter $\epsilon$. This work deals with a $q-$analog of the research made in~\cite{lama15} giving rise to a generalization of the work~\cite{ma15}. This generalization is related to the nature of the forcing term which suggests the use of a $q-$analog of an acceleration procedure. 

The proof leans on a $q-$analog of the so-called Ramis-Sibuya theorem which entails two distinct $q-$Gevrey orders. The work concludes with an application of the main result when the forcing term solves a related problem.

\medskip

}
\bigskip

\section{Introduction}

The present work deals with the study of the solution $u(t,z,\epsilon)$ of a family of inhomogeneous linear $q-$difference-differential Cauchy problems of the form
\begin{align}
&Q(\partial_z)\sigma_qu(t,z,\epsilon)=(\epsilon t)^{d_{D}}\sigma_{q}^{\frac{d_{D}}{k_2}+1}R_{D}(\partial_{z})u(t,z,\epsilon)\nonumber\\
&\hspace{3cm}+\sum_{\ell=1}^{D-1}\left(\sum_{\lambda\in I_{\ell}}t^{d_{\lambda,\ell}}\epsilon^{\Delta_{\lambda,\ell}}\sigma_{q}^{\delta_\ell}c_{\lambda,\ell}(z,\epsilon)R_{\ell}(\partial_z)u(t,z,\epsilon)\right)+\sigma_{q}f(t,z,\epsilon).\label{e1}
\end{align}
Here, $D, k_2, d_{D}$ are positive integers with $D\ge 3$, $q$ is a real number with $q>1$ and for every $1\le \ell\le D-1$, $I_\ell$ is a finite nonempty subset of nonnegative integers whilst $\delta_{\ell}$ is a positive integer. For each $1\le \ell\le D-1$ and $\lambda\in I_{\ell}$, we take $d_{\lambda,\ell}\ge1$ and $\Delta_{\lambda,\ell}\ge0$. 

The elements $Q$ and $R_{\ell}$, for $1\le \ell\le D$,  are polynomials with $\deg(Q)\ge \deg(R_{D})\ge\deg(R_{\ell})$ for all $1\le\ell\le D-1$. The details on the properties satisfied by the previous constants and polynomials involved in the equation under study are carefully described at the beginning of Section~\ref{seccion5}. We also give an example of a problem under study in the present work at the end of Section~\ref{seccion5}.

The variable $\epsilon$ acts as a perturbation parameter in the problem. We describe an asymptotic meaning of the solutions and provide the existence of a formal solution to the main problem with respect to this parameter (see Theorem~\ref{teo1215}).

For every $\gamma\in\R$, the operator $\sigma_{q}^{\gamma}$ appearing in (\ref{e1}) stands for the generalization of the dilation operator on $t$ variable, with $\gamma=1$. More precisely, for any function $g$ given in a set $H$, $\sigma_{q}^{\gamma}$ is defined by
$$\sigma_{q}^{\gamma}(g(t)):=g(q^{\gamma}t),$$ 
whenever the right-hand side makes sense, i.e. if $q^{\gamma}t\in H$ for all $t\in H$. We will also consider the natural extension of this definition to the formal framework in the following way: given a formal power series $\hat{f}(z)=\sum_{\ell\ge0}f_{\ell}z^{\ell}$ with coefficients in a set which is closed under multiplication by real numbers (in our concerns, this set would turn out to be a complex Banach space), the formal power series $\sigma_q^{\gamma}\hat{f}$ is given by $\sum_{\ell\ge0}q^{\gamma\ell}f_{\ell}z^{\ell}$.

For every $1\le \ell\le D-1$ and $\lambda\in I_{\ell}$, the function $c_{\lambda,\ell}(z,\epsilon)$ is constructed as the inverse Fourier transform with respect to $z$ of a continuous function $(m,\epsilon)\mapsto C_{\lambda,\ell}(m,\epsilon)$ defined in $\R\times B$, where $B$ is a neighborhood of the origin. As a matter of fact, $c_{\lambda,\ell}$ is a bounded holomorphic function defined in a horizontal strip in the variable $z$, say $H_{\beta'}$ (see (\ref{e268bb})), times $B$.

The forcing term $f(t,z,\epsilon)$ turns out to be a holomorphic function defined in $\mathcal{T}\times H_{\beta'}\times \mathcal{E}$, where $\mathcal{T}$ and $\mathcal{E}$ stand for finite sector with vertex at the origin. In the sequel, we provide more details on this function which is crucial in order to understand the interest of this work.

We choose $1\le k_1<k_2$, and put 
$$\frac{1}{\kappa}=\frac{1}{k_1}-\frac{1}{k_2}.$$

The construction of $f(t,z,\epsilon)$ regards as follows. Let $m\mapsto F_{n}(m,\epsilon)$ be a continuous function for $m\in\R$ and holomorphic with respect to $\epsilon\in B$, for every $n\ge0$. We assume the formal power series $F(T,m,\epsilon)=\sum_{n\ge0}F_{n}(m,\epsilon)T^{n}$ is such that its formal $q-$Borel transform of order $k_1$ (see Definition~\ref{defi155})
$$\psi_{k_1}(\tau,m,\epsilon):=\mathcal{B}_{q;1/k_1}(F(T,z,\epsilon))(\tau)=\sum_{n\ge0}\frac{F_n}{(q^{1/k_1)^{\frac{n(n-1)}{2}}}}\tau^n$$
is convergent in a neighborhood of the origin, $D_1$, with respect to $\tau$ variable. Moreover, we assume there exists a finite family of directions $(\mathfrak{d}_{p})_{0\le p\le \varsigma-1}$ such that $\psi_{k_1}$ extends holomorphically to an infinite sector $U_{\mathfrak{d}_{p}}$ with vertex at 0 and bisecting direction $\mathfrak{d}_{p}$, with $q-$exponential growth of order $k_1$ at infinity, uniformly with respect to $\epsilon\in B$. We write $\psi_{k_1}^{\mathfrak{d}_{p}}$ for this extension. This last assertion states there exists an appropriate $\varsigma_{\psi_{k_1}}(m)>0$ such that
\begin{equation}\label{e3}
\sup_{\epsilon\in B}|\psi_{k_1}^{\mathfrak{d}_{p}}(\tau,m,\epsilon)|\le \varsigma_{\psi_{k_1}}(m)\exp\left(\frac{k_1\log^2|\tau|}{2\log(q)}\right),
\end{equation}
for every $\tau\in U_{\mathfrak{d}_{p}}$ with $\tau\notin\overline{D}_1$ (see~(\ref{e859})). 

One may apply $q-$Laplace transform of order $k_1$ on $\psi_{k_1}^{\mathfrak{d}_{p}}$ (see Lemma~\ref{lema403}). Also, the dependence on $m$ lying in $\varsigma_{\psi_{k_1}}(m)$ allows us to take inverse Fourier transform on this variable and define $\tilde{f}^{\mathfrak{d}_{p}}$ as the result of both transformations. Finally, regarding assumption 3) in Definition~\ref{def818}, one may apply the change of variable $\tau\mapsto \epsilon t$ to define $f^{\mathfrak{d}_{p}}(t,z,\epsilon):=\tilde{f}^{\mathfrak{d}_{p}}(t \epsilon,z,\epsilon)$, as a holomorphic and bounded function defined in $\mathcal{T}\times H_{\beta'}\times \mathcal{E}_{p}$. Here, $\mathcal{E}_{p}$ is a finite sector with vertex at the origin in the perturbation parameter, where the family $(\mathcal{E}_{p})_{0\le p\le\varsigma-1}$ is chosen to determine a good covering in $\C^{\star}$ (see Definition~\ref{def817}). For a more detailed construction of these elements, we refer to Section~\ref{seccion5}.

More precisely, we aim to study the solution $u^{\mathfrak{d}_{p}}(t,z,\epsilon)$ for $0\le p\le \varsigma-1$ of a family of  related problems regarding each direction of extendability associated to the forcing term, rather than (\ref{e1}). We write
\begin{align}
&Q(\partial_z)\sigma_qu^{\mathfrak{d}_{p}}(t,z,\epsilon)=(\epsilon t)^{d_{D}}\sigma_{q}^{\frac{d_{D}}{k_2}+1}R_{D}(\partial_{z})u^{\mathfrak{d}_{p}}(t,z,\epsilon)\nonumber\\
&\hspace{3cm}+\sum_{\ell=1}^{D-1}\left(\sum_{\lambda\in I_{\ell}}t^{d_{\lambda,\ell}}\epsilon^{\Delta_{\lambda,\ell}}\sigma_{q}^{\delta_\ell}c_{\lambda,\ell}(z,\epsilon)R_{\ell}(\partial_z)u^{\mathfrak{d}_{p}}(t,z,\epsilon)\right)+\sigma_{q}f^{\mathfrak{d}_{p}}(t,z,\epsilon).\label{e2}
\end{align}
for each $0\le p\le \varsigma -1$ for the different equations under study.

Let us take a brief look at equation (\ref{e1}) (or equation (\ref{e2})) and describe some concerns which are important to understand the nature of the problem studied. Regarding variable $z$ in equation (\ref{e2}), we have decided to split the right-hand side in two terms: a first term related to $R_{D}(\partial_{z})$ in which the degree of the operator exceeds those of the remaining terms, associated to $R_{\ell}(\partial_{z})$, $0\le \ell\le D-1$. It is at this point where one of the $q-$Gevrey growth phenomena regulating the equation arises. Indeed, the dilation operator of this term causes a $q-$Gevrey phenomena of type $k_2$ to appear. 

As a first attempt one is tempted to study an auxiliary problem in the Borel plane directly, following the classical method of summability of formal solutions of different types of equations. In this direction, regarding Proposition~\ref{prop259}, one might try to apply $q-$Borel transform of order $k_2$ at both sides of equation (\ref{e1}), study the resulting $q-$difference-convolution problem (\ref{e479}), obtain a solution to this problem having an adequate growth in $\tau$ variable in order to provide a solution to (\ref{e1}) via the analytic inverse operator, the $q-$Laplace transform of order $k_2$. However, this procedure, followed in the recent work~\cite{ma15}, is not fruitful because of the growth nature of the forcing term. Indeed, the application of $q-$Borel transform of order $k_2$ on the forcing term gives rise to a formal power series which might have null radius of convergence. 

The alternative procedure followed in this work is to split the summation procedure in two steps. Firstly, we proceed with a $q-$analog of Borel-Laplace summation method of a lower type, $\kappa$ and attain the solution by means of an acceleration-like action. It is worth mentioning that this idea is an adaptation of that in~\cite{lama15}, to the $q-$Gevrey case. Also, the idea of concatenating formal and analytic $q-$analogs of Borel and Laplace operators in order to solve $q-$difference equations appears in~\cite{dr}.    

The present work continues a series of works dedicated to the asymptotic behavior of holomorphic solutions to different kind of $q-$difference-differential problems involving irregular singularities investigated in~\cite{lama12},~\cite{lama13},~\cite{lamasa12},~\cite{ma11}. These works can be classified in the branch of studies devoted to study from an analytic point of view of $q-$difference
equations and their formal/analytic classiffication in ~\cite{virasazh},~\cite{mazh},~\cite{rasazh},~\cite{rasazh2},~\cite{razh}. It is worth pointing out another approach in the construction of a $q-$analog of summability for formal solutions to inhomogeneous linear $q-$difference-differential equations based on Newton polygon methods, see~\cite{taya}, and also the contribution in the framework of nonlinear $q-$analogs of Briot-Bouquet type partial differential equations, see~\cite{ya}.

Let us exhibit the plan of the work.

We first state the definitions and some properties of the Banach spaces of functions involved in the construction of the solution of equation (\ref{e2}). The elements of this spaces consist of holomorphic functions defined in an infinite sector with vertex at 0 (resp. an infinite sector with vertex at infinity and a disc at 0) subjected to a $q-$exponencial growth at infinity, with respect to the first variable. Also an exponential decreasement at $\pm\infty$ is assumed in the real-valued variable $m$ (see Section 2). In Section 3, we recall some formal and analytic transformations such as the formal $q-$Borel transform of a positive order, and the analytic $q-$Laplace transform of a positive order. This transformations were introduced in the work~\cite{dr} to construct meromorphic solutions to linear $q-$difference equations from formal ones. In this section we also give a review on the properties satisfied by inverse Fourier transform, $\mathcal{F}^{-1}$.

In Section~\ref{seccion41}, we consider the auxiliary equation
\begin{align}&Q(im)\sigma_{q}U(T,m,\epsilon)=T^{d_{D}}\sigma_{q}^{\frac{d_{D}}{k_2}+1}R_{D}(im)U(T,m,\epsilon)\nonumber\\
+&\sum_{\ell=1}^{D-1}\left(\sum_{\lambda\in I_{\ell}}T^{d_{\lambda,\ell}}\epsilon^{\Delta_{\lambda,\ell}-d_{\lambda,\ell}}\frac{1}{(2\pi)^{1/2}}\int_{-\infty}^{\infty}C_{\lambda,\ell}(m-m_1,\epsilon)R_{\ell}(im_1)U(q^{\delta_{\ell}}T,m_1,\epsilon)dm_1\right)+\sigma_{q}F(T,m,\epsilon).\label{e4}
\end{align}
and study the resulting equation after the action of formal $q-$Borel transformation of order $k_1$ on an equation coming from (\ref{e4}):
\begin{align}&Q(im)\frac{\tau^{k_1}}{(q^{1/k_1})^{k_1(k_1-1)/2}}w_{k_1}(\tau,m,\epsilon)=\frac{\tau^{d_D+k_1}}{(q^{1/k_1})^{(d_{D}+k_1)(d_{D}+k_1-1)/2}}\sigma_{q}^{-d_{D}/\kappa}R_{D}(im)w_{k_1}(\tau,m,\epsilon)\nonumber\\
+&\sum_{\ell=1}^{D-1}\left(\sum_{\lambda\in I_{\ell}}\frac{\epsilon^{\Delta_{\lambda,\ell}-d_{\lambda,\ell}}\tau^{d_{\lambda,\ell}+k_1}}{(q^{1/k_1})^{(d_{\lambda,\ell}+k_1)(d_{\lambda,\ell}+k_1-1)/2}}\sigma_q^{\delta_{\ell}-\frac{d_{\lambda,\ell}}{k_1}-1}\frac{1}{(2\pi)^{1/2}}(C_{\lambda,\ell}(m,\epsilon)\ast^{R_{\ell}}w_{k_1}(\tau,m,\epsilon))\right)\nonumber\\
&\hspace{10cm}+\frac{\tau^{k_1}}{(q^{1/k_1})^{k_{1}(k_1-1)/2}}\psi_{k_1}(\tau,m,\epsilon).\label{e5}
\end{align}
For every $0\le p\le \varsigma-1$, we come up to a novel auxiliary problem fixing $\psi_{k_1}:=\psi_{k_1}^{\mathfrak{d}_{p}}$, and by means of Proposition~\ref{prop321} we get the existence of a solution $w_{k_1}^{\mathfrak{d}_{p}}(\tau,m,\epsilon)$ of (\ref{e5}), which satisfies that
$$\sup_{\stackrel{\tau\in U_{\mathfrak{d}_{p}}\cup \overline{D}_1}{m\in\R}}|w_{k_1}^{\mathfrak{d}_{p}}(\tau,m,\epsilon)|\le C_{w_{k_1}^{\mathfrak{d}_{p}}}\frac{1}{(1+|m|)^{\mu}}e^{-\beta|m|}\exp\left(\frac{\kappa\log^2|\tau+\delta|}{2\log(q)}+\alpha\log|\tau+\delta|\right),$$
uniformly with respect to $\epsilon\in B$, for some $C_{w_{k_1}^{\mathfrak{d}_{p}}},\tilde{\delta},\alpha,\mu,\beta>0$.

In Section~\ref{seccion42}, we study the action of formal $q-$Borel transform of order $k_2$ on an equation obtained from (\ref{e4}). As it was pointed out above, one can not guarantee convergence of the resulting element in a neighborhood of the origin, arriving at the formal problem
\begin{align}&Q(im)\frac{\tau^{k_2}}{(q^{1/k_2})^{k_2(k_2-1)/2}}\hat{w}_{k_2}(\tau,m,\epsilon)=R_{D}(im)\frac{\tau^{d_D+k_2}}{(q^{1/k_2})^{(d_{D}+k_2)(d_{D}+k_2-1)/2}}\hat{w}_{k_2}(\tau,m,\epsilon)\nonumber\\
+&\sum_{\ell=1}^{D-1}\left(\sum_{\lambda\in I_{\ell}}\frac{\epsilon^{\Delta_{\lambda,\ell}-d_{\lambda,\ell}}\tau^{d_{\lambda,\ell}+k_2}}{(q^{1/k_2})^{(d_{\lambda,\ell}+k_2)(d_{\lambda,\ell}+k_2-1)/2}}\sigma_q^{\delta_{\ell}-\frac{d_{\lambda,\ell}}{k_2}-1}\frac{1}{(2\pi)^{1/2}}(C_{\lambda,\ell}(m,\epsilon)\ast^{R_{\ell}}\hat{w}_{k_2}(\tau,m,\epsilon))\right)\nonumber\\
&\hspace{8cm}+\frac{\tau^{k_2}}{(q^{1/k_2})^{k_{2}(k_2-1)/2}}\hat{\mathcal{B}}_{q;1/k_2}(F(T,m,\epsilon)).\label{e6}
\end{align}
We substitute the formal power series $\hat{\mathcal{B}}_{q;1/k_2}(F(T,m,\epsilon))$ by the acceleration of $\psi_{k_1}^{\mathfrak{d}_{p}}$ for each $0\le p\le \varsigma-1$ and study the resulting equation for each $p$. The role of the acceleration operator is being played by $q-$Laplace transform of an adecquate order. Indeed, we substitute $\hat{\mathcal{B}}_{q;1/k_2}(F(T,m,\epsilon))$ by $\mathcal{L}_{q;1/\kappa}(h\mapsto\psi_{k_1}^{\mathfrak{d}_{p}}(h,m,\epsilon))(\tau)$, constructed in Lemma~\ref{lema403}, for every $0\le p\le \varsigma-1$. By Proposition~\ref{prop571}, we obtain a solution $w_{k_{2}}^{\mathfrak{d}_{p}}$ of the previous problem by means of a fixed point result in appropriate Banach spaces. This solution is defined in $(\mathcal{R}_{\mathfrak{d}_{p}}^{b}\cup S_{\mathfrak{d}_{p}})\times \R\times B$, where $\mathcal{R}_{\mathfrak{d}_{p}}^{b}$ (resp. $S_{\mathfrak{d}_{p}}$) stands for a finite wide sector (resp. an infinite sector) of bisecting direction $\mathfrak{d}_{p}$.
Indeed, one has
\begin{equation}\label{e7}
\sup_{\stackrel{\tau\in (\mathcal{R}_{\mathfrak{d}_{p}}^{b}\cup S_{\mathfrak{d}_{p}})}{m\in\R}}|w_{k_2}^{\mathfrak{d}_{p}}(\tau,m,\epsilon)|\le C_{w_{k_2}^{\mathfrak{d}_{p}}}\frac{1}{(1+|m|)^{\mu}}e^{-\beta|m|}\exp\left(\frac{k_2\log^2|\tau|}{2\log(q)}+\nu\log|\tau|\right),
\end{equation}
for some $C_{w_{k_2}^{\mathfrak{d}_{p}}},\nu>0$.

We prove (Proposition~\ref{prop638}) that $w_{k_{2}}^{\mathfrak{d}_{p}}$ and $\mathcal{L}^{\mathfrak{d}_{p}}_{q_{1/\kappa}}(\tau\mapsto w_{k_{1}}^{\mathfrak{d}_{p}}(\tau,m,\epsilon))$ coincide in the intersection of their domains of definition. Consequently, $\mathcal{L}^{\mathfrak{d}_{p}}_{q_{1/\kappa}}(\tau\mapsto w_{k_{1}}^{\mathfrak{d}_{p}}(\tau,m,\epsilon))$ can be extended to  $\mathcal{R}_{\mathfrak{d}_{p}}^{b}\cup S_{\mathfrak{d}_{p}}$. In view of (\ref{e7}) and the choice of the domains associated to the good covering $(\mathcal{E}_{p})_{0\le p\le \varsigma-1}$ (see Definition~\ref{def818}), one can define the function
$$ u^{\mathfrak{d}_{p}}(t,z,\epsilon):=\mathcal{F}^{-1}\left(m\mapsto  \frac{1}{\pi_{q^{1/k_2}}}\int_{L_{\gamma_p}}\frac{w^{\mathfrak{d}_{p}}_{k_2}(u,m,\epsilon)}{\Theta_{q^{1/k_2}}\left(\frac{u}{\epsilon t}\right)}\frac{du}{u}\right)(z),$$
for every $(t,z,\epsilon)\in \mathcal{T}\times H_{\beta'}\times \mathcal{E}_{p}$, for all $0\le p\le \varsigma-1$. Here $L_{\gamma_{p}}$ stands for a ray from 0 to infinity contained in $S_{\mathfrak{d}_{p}}$ see Theorem~\ref{teo872}. The function $u^{\mathfrak{d}_{p}}(t,z,\epsilon)$ solves (\ref{e2}) (see Theorem~\ref{teo872}).

In the spirit of~\cite{dr}, the procedure we have followed can be summarized in some sense by the composition of these operators:
$$\mathcal{L}_{q;1/k_2}\circ\mathcal{L}_{q;1/\kappa}\circ\hat{\mathcal{B}}_{1/\kappa}\circ\hat{\mathcal{B}}_{1/k_2}=\mathcal{L}_{q;1/k_2}\circ\mathcal{L}_{q;1/\kappa}\circ\hat{\mathcal{B}}_{1/k_1}.$$
In the second part of Section~\ref{seccion5} we study the difference of two solutions and obtain two different results, depending on the geometry of the problem (see Proposition~\ref{prop925} and Proposition~\ref{prop1047}). The previous results are applied in Section~\ref{seccion6} to attain the main result of the work (Theorem~\ref{teo1281}), by means of a novel two-level version of a $q-$analog of Ramis-Sibuya theorem (see Theorem~\ref{teo1215}), namely, the existence of a formal power series 
$$\hat{u}(t,z,\epsilon)=\sum_{m\ge0}h_m(t,z)\frac{\epsilon^m}{m!}\in\mathbb{F}[[\epsilon]],$$
with coefficients in the Banach space $\mathbb{F}$ of bounded holomorphic functions defined on $\mathcal{T}\times H_{\beta'}$ with the supremum norm. This formal power series is a formal solution of 
\begin{align*}
&Q(\partial_z)\sigma_q\hat{u}(t,z,\epsilon)=(\epsilon t)^{d_{D}}\sigma_{q}^{\frac{d_{D}}{k_2}+1}R_{D}(\partial_{z})\hat{u}(t,z,\epsilon)\\
&\hspace{3cm}+\sum_{\ell=1}^{D-1}\left(\sum_{\lambda\in I_{\ell}}t^{d_{\lambda,\ell}}\epsilon^{\Delta_{\lambda,\ell}}\sigma_{q}^{\delta_\ell}c_{\lambda,\ell}(z,\epsilon)R_{\ell}(\partial_z)\hat{u}(t,z,\epsilon)\right)+\sigma_{q}\hat{f}(t,z,\epsilon),
\end{align*}
where $\hat{f}$ is the common asymptotic representation of $f^{\mathfrak{d}_{p}}$ with respect to the perturbation parameter $\epsilon$, as described in Lemma~\ref{lema1271}. The sense in which $\hat{u}(t,z,\epsilon)$ represents $u^{\mathfrak{d}_{p}}(t,z,\epsilon)$ for all $0\le p\le \varsigma-1$ is detailed in Theorem~\ref{teo1281}.

The work concludes in Section~\ref{seccion7} with an application of the main result when the formal power series $F(T,m,\epsilon)$ is a solution of another related problem (see Theorem~\ref{teofin}).

\section{Banach spaces of functions and related results}

Throughout the whole section we fix real numbers $\beta,\mu>0$, $q>1$ and $\alpha$. Some conditions on this elements may be described when needed in the following constructions and results.

Through this section we assume $U_d\subseteq \C^{\star}:=\C\setminus\{0\}$ is a sector with vertex at the origin and bisecting direction $d\in\R$. We also choose $\rho>0$ and consider the disc centered at $0\in\C$ with radius $\rho$, notated by $D(0,\rho):=\{\tau\in\C:|\tau|<\rho\}$. Let $\delta>0$ and assume that the distance from $U_d\cup D(0,\rho)$ to the real number $-\delta$ is positive. We also take $k>0$.

We denote $\overline{D}(0,\rho)$ the closure of $D(0,\rho)$.  

\begin{defin}
We denote $\hbox{Exp}^{d}_{(k,\beta,\mu,\alpha,\rho)}$ the vector space of continuous complex valued functions $(\tau,m)\mapsto h(\tau,m)$ on $U_d\cup \overline{D}(0,\rho)\times \R$, holomorphic with respect to $\tau$ on $U_d\cup D(0,\rho)$ such that
$$\left\|h(\tau,m)\right\|_{(k,\beta,\mu,\alpha,\rho)}:=\sup_{\stackrel{\tau\in U_{d}\cup \overline{D}(0,\rho)}{m\in\R}}(1+|m|)^{\mu}e^{\beta|m|}\exp\left(-\frac{k\log^2|\tau+\delta|}{2\log(q)}-\alpha\log|\tau+\delta|\right)|h(\tau,m)|<\infty. $$
The set $\hbox{Exp}^{d}_{(k,\beta,\mu,\alpha,\rho)}$ turns out to be a Banach space when endowed with the norm $\left\|\cdot\right\|_{(k,\beta,\mu,\alpha,\rho)}$.
\end{defin}

The previous norm is a modified version of that used in the previous work~\cite{lama12}, and by the first author in~\cite{ma15}. Here, a shift on the variable $\tau$ is needed in such a way that the elements belonging to $\hbox{Exp}^{d}_{(k,\beta,\mu,\alpha,\rho)}$ remain holomorphic and bounded in a neighborhood of the origin, whilst $q-$exponential behavior at infinity is preserved.

\begin{lemma}
Let $(\tau,m)\mapsto a(\tau,m)$ be a bounded continuous function defined in $(U_{d}\cup \overline{D}(0,\rho))\times\R$, holomorphic with respect to $\tau$ on $U_{d}\cup D(0,\rho)$. For every $h(\tau,m)\in\hbox{Exp}_{(k,\beta,\mu,\alpha,\rho)}^{d}$, the function $a(\tau,m)h(\tau,m)\in \hbox{Exp}_{(k,\beta,\mu,\alpha,\rho)}^{d}$ and 
$$\left\|a(\tau)h(\tau,m)\right\|_{(k,\beta,\mu,\alpha)}\le C_{1}\left\|h(\tau,m)\right\|_{(k,\beta,\mu,\alpha,\rho)},$$
where $C_{1}=\sup_{\tau\in (U_{d}\cup \overline{D}(0,\rho)),m\in\R}|a(\tau,m)|$.
\end{lemma}

\begin{prop}\label{prop82}
Let $\gamma_1,\gamma_2\ge0$ such that 
\begin{equation}\label{e74}
\gamma_1\le k\gamma_2.
\end{equation}
Then, there exists $C_2>0$ (depending on $k,q,\alpha,\gamma_1,\gamma_2,\delta$) with
$$\left\|\tau^{\gamma_1}\sigma_{q}^{-\gamma_2}f(\tau,m)\right\|_{(k,\beta,\mu,\alpha,\rho)}\le C_2\left\|f(\tau,m)\right\|_{(k,\beta,\mu,\alpha,\rho)},$$
for every $f(\tau,m)\in\hbox{Exp}^{d}_{(k,\beta,\mu,\alpha,\rho)}$.
\end{prop}
\begin{proof}
Let $f\in\hbox{Exp}^{d}_{(k,\beta,\mu,\alpha,\rho)}$. One can write
\begin{align*}
&\left\|\tau^{\gamma_1}\sigma_{q}^{-\gamma_2}f(\tau,m)\right\|_{(k,\beta,\mu,\alpha,\rho)}=\sup_{\stackrel{\tau\in U_{d}\cup \overline{D}(0,\rho)}{m\in\R}}(1+|m|)^{\mu}e^{\beta|m|}\exp\left(-\frac{k\log^2|\tau+\delta|}{2\log(q)}-\alpha\log|\tau+\delta|\right)|\tau|^{\gamma_1}\\
&\times  |f(\tau/q^{\gamma_2},m)|\exp\left(-\frac{k\log^2\left|\frac{\tau}{q^{\gamma_2}}+\delta\right|}{2\log(q)}-\alpha\log\left|\frac{\tau}{q^{\gamma_2}}+\delta\right|\right)\exp\left(\frac{k\log^2\left|\frac{\tau}{q^{\gamma_2}}+\delta\right|}{2\log(q)}+\alpha\log\left|\frac{\tau}{q^{\gamma_2}}+\delta\right|\right)
\end{align*}

We observe that $\sigma_{q}^{-\gamma_{2}}(U_{d}\cup \overline{D}(0,\rho))\subseteq U_{d}\cup \overline{D}(0,\rho)$. As a consequence, the previous expression is upper bounded by
\begin{equation}\label{e93}
\left\|f(\tau,m)\right\|_{(k,\beta,\mu,\alpha,\rho)}\sup_{\tau\in U_{d}\cup \overline{D}(0,\rho)}\exp\left(\frac{k(\log^2\left|\frac{\tau}{q^{\gamma_2}}+\delta\right|-\log^2|\tau+\delta|)}{2\log(q)}\right)|\tau|^{\gamma_1}\left(\frac{\left|\tau/q^{\gamma_2}+\delta\right|}{|\tau+\delta|}\right)^\alpha.
\end{equation}

For every $\tau\in U_{d}\cup \overline{D}(0,\rho)$ one has
\begin{equation}\label{e102}
\log^2\left|\frac{\tau}{q^{\gamma_2}}+\delta\right|-\log^2|\tau+\delta|=\log^2|\tau+\delta q^{\gamma_2}|+\log^2(q^{\gamma_2})-2\gamma_2\log(q)\log|\tau+\delta q^{\gamma_2}|-\log^2|\tau+\delta|.
\end{equation}
Let $\rho_1=2\rho$. For every $\tau\in S_{d}$ with $|\tau|>\rho_1$, standard calculations yield 
\begin{align}
\log^2|\tau+\delta q^{\gamma_2}|-\log^2|\tau+\delta|&=\log|(\tau+\delta q^{\gamma_2})(\tau+\delta)|\log\left|\frac{\tau+\delta q^{\gamma_2}}{\tau+\delta}\right|\nonumber\\
&\le C_{21}\log|\tau|\log\left(1+\frac{\delta(q^{\gamma_2}-1)}{|\tau+\delta|}\right)\le C_{22}\log|\tau|\frac{\delta(q^{\gamma_2}-1)}{|\tau+\delta|}\le C_{23},\label{e102b}
\end{align}
for some $C_{21},C_{22},C_{23}>0$ (depending on $q,\gamma_2,\delta$). On the other hand, the function 
$$\tau\mapsto \log^2|\tau+\delta q^{\gamma_2}|-\log^2|\tau+\delta|$$
is continuous in the compact set $\overline{D}(0,\rho)\cup \{\tau\in \overline{U_{d}}: |\tau|\le\rho_1\}$. This and (\ref{e102b}) provide the existence of $C_{24}>0$ (depending on $k, q,\gamma_2,\delta$) such that (\ref{e102}) is estimated from above by $C_{24}-2\gamma_2\log(q)\log|\tau+\delta q^{\gamma_2}|$. Taking into account these estimates, one derives the existence of $C_{25}>0$ (depending on $k,\delta,q,\gamma_2$) such that
$$\left\|\tau^{\gamma_1}\sigma_{q}^{-\gamma_2}f(\tau,m)\right\|_{(k,\beta,\mu,\alpha,\rho)}\le C_{25}\left\|f(\tau,m)\right\|_{(k,\beta,\mu,\alpha,\rho)}\sup_{\tau\in U_{d}\cup \overline{D}(0,\rho)}|\tau|^{\gamma_1}\left(\frac{\left|\tau+\delta q^{\gamma_2}\right|}{|\tau+\delta|}\right)^\alpha|\tau+\delta q^{\gamma_{2}}|^{-k\gamma_2}$$
It is straight to check that
$$\sup_{\tau\in U_{d}\cup \overline{D}(0,\rho)}|\tau|^{\gamma_1}\left(\frac{\left|\tau+\delta q^{\gamma_2}\right|}{|\tau+\delta|}\right)^\alpha|\tau+\delta q^{\gamma_{2}}|^{-k\gamma_2}\le  \sup_{x\ge0}C_{26}x^{\gamma_1}(x+\delta q^{\gamma_2})^{-k\gamma_2},$$
for some $C_{26}>0$. The result follows from here, in view of (\ref{e74}).
\end{proof}

\begin{defin}
We denote $E_{(\beta,\mu)}$ the vector space of continuous functions $h:\R\to\C$ such that
$$\left\|h(m)\right\|_{(\beta,\mu)}:=\sup_{m\in\R}(1+|m|)^{\mu}\exp(\beta|m|)|h(m)|<\infty.$$
The pair $(E_{(\beta,\mu)},\left\|\cdot\right\|_{(\beta,\mu)})$ is a Banach space. 
\end{defin}

\begin{lemma}\label{lema123}
Let $h_j:\R\to\C$ be a continuous function for $j=1,2$. Assume that $sup_{m\in\R}|h_1(m)|$ is finite, and $h_2\in E_{(\beta,\mu)}$. Then the product $h_1h_2\in E_{(\beta,\mu)}$ and 
$$\left\|h_1(m)h_2(m)\right\|_{(\beta,\mu)}\le \left(\sup_{m\in\R}|h_1(m)|\right)\left\|h_2(m)\right\|_{(\beta,\mu)}.$$
\end{lemma}
\begin{proof}
It follows from the definition of the norm $\left\|\cdot\right\|_{(\beta,\mu)}$.
\end{proof}

Let $h_j:\R\to\C$ be a continuous function for $j=1,2$. Let $Q\in\C[X]$. One can define the convolution product 
$$h_1(m)\ast^{Q}h_2(m):=\int_{-\infty}^{\infty}h_1(m-m_1)Q(im_1)h_2(m_1)dm_1,\quad m\in\R,$$
whenever the integral converges, extending the classical convolution product for $Q\equiv 1$. 

\begin{prop}\label{prop128}
Let $Q,R\in\C[x]$ such that $\deg(R)\ge \deg(Q)$, $R(im)\neq0$ for every $m\in\R$. Assume moreover that $\mu>\deg(Q)+1$. Given a continuous function $b:\R\to\C$ with $\sup_{m\in\R}|b(m)R(im)|\le1$, the space $(E_{(\beta,\mu)},\left\|\cdot\right\|_{(\beta,\mu)})$ turns out to be a Banach algebra when endowed with the product $\ast^{b,Q}$ defined by $h_1(m)\ast^{b,Q}h_2(m):=b(m)h_1(m)\ast^{Q}h_2(m)$ for every $m\in\R$.
\end{prop}

We refer to~\cite{lama14} for a proof of the previous result.

\begin{prop}\label{prop143}
Let $Q,R,b$ be as in Proposition~\ref{prop128}. We assume $c(m)\in E_{(\beta,\mu)}$. Then, for every $f\in\hbox{Exp}^{d}_{(k,\beta,\mu,\alpha,\rho)}$, the function $c\ast^{b,Q}f\in \hbox{Exp}^{d}_{(k,\beta,\mu,\alpha,\rho)}$ and 
\begin{equation}\label{e145} \left\|c(m)\ast^{b,Q}f(\tau,m)\right\|_{(k,\beta,\mu,\alpha,\rho)}\le C_3\left\|c(m)\right\|_{(\beta,\mu)}\left\|f(\tau,m)\right\|_{(k,\beta,\mu,\alpha,\rho)},
\end{equation}
for some $C_{3}>0$ (depending on $\mu, q,\alpha,k,Q(X),R(X)$).
\end{prop}
\begin{proof}
Let $f\in\hbox{Exp}^{d}_{(k,\beta,\mu,\alpha,\rho)}$. Regarding Proposition~\ref{prop128}, it is direct to check that $c(m)\ast^{b,Q}f(\tau,m)$ defines a continuous function in $U_{d}\cup \overline{D}(0,\rho)\times \R$ and holomorphic in $U_{d}\cup D(0,\rho)$ with respect to the variable $\tau$. From the definition of the space $\hbox{Exp}^{d}_{(k,\beta,\mu,\alpha,\rho)}$ one has
$$\left\|c(m)\ast^{b,Q}f(\tau,m)\right\|_{(k,\beta,\mu,\alpha,\rho)}$$
\begin{align*}
&\le\sup_{\stackrel{\tau\in U_{d}\cup \overline{D}(0,\rho)}{m\in\R}}(1+|m|)^{\mu}e^{\beta |m|}\exp\left(-\frac{k\log^2|\tau+\delta|}{2\log(q)}-\alpha\log|\tau+\delta|\right)\\
&\times\frac{|R(im)b(m)|}{|R(im)|}\int_{-\infty}^{\infty}\left((1+|m-m_1|)^{\mu}e^{\beta|m-m_1|}|c(m-m_1)||Q(im_1)|\right.\\
&\times\left.|f(\tau,m_1)|(1+|m_1|)^\mu e^{\beta|m_1|}\right)\frac{e^{-\beta|m-m_1|}e^{-\beta|m_1|}}{(1+|m-m_1|)^{\mu}(1+|m_1|)^{\mu}}  dm_1
\end{align*}
From the hypotheses made on $Q$ and $R$, there exist $C_{31},C_{32}>0$ such that
\begin{equation}\label{e157}
|Q(im_1)|\le C_{31}(1+|m_1|)^{\deg(Q)},\qquad |R(im)|\ge C_{32}(1+|m|)^{\deg(R)},
\end{equation}
for every $m\in\R$. The triangular inequality, the estimates in (\ref{e157}), and Lemma 4 in~\cite{ma2} (or Lemma~2.2 in~\cite{cota2}) yield the existence of $C_{3}>0$ such that 
\begin{align}
&\frac{(1+|m|)^{\mu}}{|R(im)|}\int_{-\infty}^{\infty}\frac{|Q(im_1)|e^{\beta(|m|-|m-m_1|-|m_1|)}}{(1+|m-m_1|)^{\mu}(1+|m_1|)^{\mu}}dm_1 \nonumber \\
&\le \sup_{m\in\R}\frac{C_{31}}{C_{32}}(1+|m|)^{\mu-\deg(Q)}\int_{-\infty}^{\infty} \frac{1}{(1+|m-m_1|)^{\mu}(1+|m_1|)^{\mu-\deg(Q)}}dm_1\le C_{3},
\label{e159}
\end{align}
for every $m\in\R$, provided that $\mu>\deg(Q)+1$. This proves (\ref{e145}).
\end{proof}

Let $S_{d}$ be an infinite sector of bisecting direction $d$ and $\mathcal{R}_{d}^{b}$ be a finite sector of bisecting direction $d$. We take $\nu\in\R$.

We define another space of functions which will be useful in the sequel. It corresponds to that of Definition 1 in~\cite{ma15}. 
\begin{defin}
$\hbox{Exp}^{d}_{(k,\beta,\mu,\nu)}$ stands for the vector space of continuous complex valued functions $(\tau,m)\mapsto h(\tau,m)$ on $( S_{d} \cup\overline{\mathcal{R}_{d}^{b}})\times \R$, holomorphic with respect to $\tau$ on $S_d \cup \mathcal{R}_{d}^{b}$ such that
$$\left\|h(\tau,m)\right\|_{(k,\beta,\mu,\nu)}:=\sup_{\tau\in (S_{d} \cup\overline{\mathcal{R}_{d}^{b}}),m\in\R}(1+|m|)^{\mu}e^{\beta|m|}\exp\left(-\frac{k\log^2|\tau|}{2\log(q)}-\nu\log|\tau|\right)|h(\tau,m)|<\infty. $$
The space $\hbox{Exp}^{d}_{(k,\beta,\mu,\nu)}$ turns out to be a Banach space when endowed with the norm $\left\|\cdot\right\|_{(k,\beta,\mu,\nu)}$.
\end{defin}

The growth behavior of the elements in the space $\hbox{Exp}^{d}_{(k,\beta,\mu,\nu)}$ differs at 0, when compared to the growth rate of the elements in $\hbox{Exp}^{d}_{(k,\beta,\mu,\nu,\rho)}$ at the origin with respect to $\tau$ variable. However, both spaces share functions with the same growth at infinity.

We state some auxiliary lemmas in the shape as those enunciated for the space $\hbox{Exp}_{(k,\beta,\mu,\alpha,\rho)}^{d}$. The proofs for these results are omitted, and they can be found in~\cite{ma15}.

\begin{lemma}\label{lema182}
Let $(\tau,m)\mapsto a(\tau,m)$ be a bounded continuous function in $S_{d} \cup\overline{\mathcal{R}_{d}^{b}}\times\R$, holomorphic with respect to $\tau$ on $S_{d} \cup\mathcal{R}_{d}^{b}$. For every $h(\tau,m)\in\hbox{Exp}_{(k,\beta,\mu,\nu)}^{d}$, the function $a(\tau,m)h(\tau,m)\in \hbox{Exp}_{(k,\beta,\mu,\nu)}^{d}$ and 
$$\left\|a(\tau,m)h(\tau,m)\right\|_{(k,\beta,\mu,\nu)}\le \tilde{C}_{1}\left\|h(\tau,m)\right\|_{(k,\beta,\mu,\nu)},$$
where $\tilde{C}_{1}=\sup_{\tau\in (S_{d} \cup\overline{\mathcal{R}_{d}^{b}}),m\in\R}|a(\tau,m)|$.
\end{lemma}

\begin{prop}\label{prop188}
Let $\gamma_{1},\gamma_{2},\gamma_{3}\ge0$ such that
\begin{equation}\label{e191}
\gamma_2\ge k\gamma_3,\quad \gamma_1+k\gamma_3\ge\gamma_2.
\end{equation}
Let $a_{\gamma_{1}}(\tau)$ be a holomorphic function on $S_{d} \cup\mathcal{R}_{d}^{b}$, continuous on $S_{d} \cup\overline{\mathcal{R}_{d}^{b}}$ with $(1+|\tau|)^{\gamma_{1}}|a_{\gamma_{1}}(\tau)|\le 1$ for every $\tau\in(S_{d} \cup\overline{\mathcal{R}_{d}^{b}})$. Then, for every $f\in\hbox{Exp}_{(k,\beta,\mu,\nu)}^{d}$, the function $a_{\gamma_{1}}(\tau)\tau^{\gamma_2}\sigma_{q}^{-\gamma_3}f(\tau,m)$ belongs to $\hbox{Exp}_{(k,\beta,\mu,\nu)}^{d}$, and there exists $\tilde{C}_{2}>0$, depending on $k,q,\nu,\gamma_1,\gamma_2,\gamma_3$, such that
$$ \left\|a_{\gamma_{1}}(\tau)\tau^{\gamma_2}\sigma_{q}^{-\gamma_3}f(\tau,m)\right\|_{(k,\beta,\mu,\nu)}\le\tilde{C}_{2}\left\|f(\tau,m)\right\|_{(k,\beta,\mu,\nu)}.$$ 
\end{prop}

\begin{prop}\label{prop197}
Let $Q(X),R(X)$ and $b$ be as in Proposition~\ref{prop128}. We assume $c(m)\in E_{(\beta,m)}$. Then, for every $f(\tau,m)\in\hbox{Exp}^{d}_{(k,\beta,\mu,\nu)}$ the function $c(m)\ast^{b,Q}f(\tau,m)\in\hbox{Exp}_{(k,\beta,\mu,\nu)}^{d}$ and there exists $\tilde{C}_{3}>0$, depending on $Q(X), R(X),\mu, q, \nu, k$, such that
$$\left\|c(m)\ast^{b,Q}f(\tau,m)\right\|_{(k,\beta,\mu,\nu)}\le\tilde{C}_{3}\left\|c(m)\right\|_{(\beta,\mu)}\left\|f(\tau,m)\right\|_{(k,\beta,\mu,\nu)}.$$
\end{prop}

\section{Review of some formal and analytic transforms}

In the present section, we recall the definitions and main properties of some formal and analytic transforms. More precisely, we will be concerned with $q-$Borel, $q-$Laplace and Fourier transforms. Throughout this section, $\mathbb{E}$ stands for a complex Banach space.

Let $q>1$ be a real number and $k\ge1$ be an integer. The next definition and result can both be found in~\cite{dr}, and also in the previous work~\cite{ramis}. 

\begin{defin}\label{defi155}
Let $\hat{a}(T)=\sum_{n\ge0}a_nT^n\in\mathbb{E}[[T]]$. We define the formal $q-$Borel transform of order $k$ of $\hat{a}(T)$ as the formal power series
$$\hat{\mathcal{B}}_{q;1/k}(\hat{a}(T))(\tau)=\sum_{n\ge0}a_n\frac{\tau^n}{(q^{1/k})^{n(n-1)/2}}\in\mathbb{E}[[\tau]].$$
\end{defin}

We recall that for every $\gamma\in\R$, the operator $\sigma_{q}^{\gamma}$, acting on $\mathbb{E}[[T]]$, stands for the generalized dilation operator on $T$ variable, $\sigma_{q}^{\gamma}(\hat{a}(T))=\hat{a}(q^{\gamma}T)$ for every $\hat{a}(T)\in\mathbb{E}[[T]]$.

The proof of this result can be found in~\cite{ma15}, Proposition 5.

\begin{prop}\label{prop157}
Let $\sigma\in\mathbb{N}$ and $j\in\mathbb{Q}$. Then, the following formal identity holds
$$\hat{\mathcal{B}}_{q;1/k}(T^{\sigma}\sigma_{q}^{j}\hat{a}(T))(\tau)=\frac{\tau^\sigma}{(q^{1/k})^{\sigma(\sigma-1)/2}}\sigma_q^{j-\frac{\sigma}{k}}\left(\hat{\mathcal{B}}_{q;1/k}(\hat{a}(T))(\tau)\right),$$
for every $\hat{a}(T)\in\mathbb{E}[[T]]$.
\end{prop}

At this point, we can recall the definition of a $q-$Laplace transform of order $k$, extending that used in~\cite{lama12} for $k=1$, and introduced in the work~\cite{dizh}. It provides a continuous $q-$analog for the formal inverse of $\hat{\mathcal{B}}_{q;1/k}$ developed in~\cite{dr}. The associated kernel of the $q-$Laplace operator is Jacobi theta function of order $k$ defined by
$$\Theta_{q^{1/k}}(x)=\sum_{n\in\Z}q^{-\frac{n(n-1)}{2k}}x^n,$$
for $x\in\C^{\star}$, $m\in\Z$. This function solves the $q-$difference equation
\begin{equation}\label{e196}
\Theta_{q^{1/k}}(q^{\frac{m}{k}}x)=q^{\frac{m(m+1)}{2k}}x^m\Theta_{q^{1/k}}(x),
\end{equation}
for every $x\in\C^\star$. As a direct consequence of Lemma~4.1 in~\cite{lama12}, extended for any value of $k$, Jacobi theta function of order $k$ satisfies that for every $\tilde{\delta}>0$ there exists a positive constant $C_{q,k}$ not depending on $\tilde{\delta}$, such that
\begin{equation}\label{e200}
\left|\Theta_{q^{1/k}}(x)\right|\ge C_{q,k}\tilde{\delta}\exp\left(\frac{k}{2}\frac{\log^2|x|}{\log(q)}\right)|x|^{1/2},
\end{equation}
for every $x\in\C^{\star}$ verifying $|1+xq^{\frac{m}{k}}|>\tilde{\delta}$, for all $m\in\Z$. This last property is crucial in order for the $q-$Laplace transform of order $k$ to be well-defined.

\begin{defin}\label{def203}
Let $\rho>0$ and $U_{d}$ be an unbounded sector with vertex at 0 and bisecting direction $d\in\R$. Let $f:D(0,\rho)\cup U_{d}\to\mathbb{E}$ be a holomorphic function, continuous on $\overline{D}(0,\rho)$ such that there exist constants $K>0$ and $\alpha\in\R$ with
\begin{equation}\label{e205}
\left\|f(x)\right\|_{\mathbb{E}}\le K\exp\left(\frac{k}{2}\frac{\log^2|x|}{\log(q)}+\alpha\log|x|\right)
\end{equation}
for every $x\in U_{d}$, $|x|\ge \rho$ and 
\begin{equation}\label{e209}
\left\|f(x)\right\|_{\mathbb{E}}\le K
\end{equation}
for all $x\in \overline{D}(0,\rho)$. Take $\gamma\in\R$ such that $e^{i\gamma}\in U_{d}$. We put $\pi_{q^{1/k}}=\frac{\log(q)}{k}\prod_{n\ge0}(1-\frac{1}{q^{\frac{n+1}{k}}})^{-1}$, and define the $q-$Laplace transform of order $k$ of $f$ in direction $\gamma$ as
$$\mathcal{L}^{\gamma}_{q;1/k}(f(x))(T)=\frac{1}{\pi_{q^{1/k}}}\int_{L_{\gamma}}\frac{f(u)}{\Theta_{q^{1/k}}\left(\frac{u}{T}\right)}\frac{du}{u},$$
where $L_{\gamma}$ stands for the set $\R_{+}e^{i\gamma}:=\{te^{i\gamma}:t\in (0,\infty)\}$.
\end{defin}

The incoming lemmas are stated without proof which can be found in~\cite{ma15}, Lemma 4 and Proposition 6. The first result studies the domain of definition of the $q-$Laplace transform of order $k$ whilst the second states a commutation formula of the $q-$Laplace operator of order $k$ with respect to some other operators.

\begin{lemma}\label{lema219}
Let $\tilde{\delta}>0$. Under the hypotheses of Definition~\ref{def203}, $\mathcal{L}^{\gamma}_{q;1/k}(f(x))(T)$ defines a bounded and holomorphic function on the domain $\mathcal{R}_{\gamma,\tilde{\delta}}\cap D(0,r_1)$ for any $0<r_1\le q^{\left(\frac{1}{2}-\alpha\right)/k}/2$, where
$$\mathcal{R}_{\gamma,\tilde{\delta}}=\left\{T\in\C^{\star}:|1+\frac{e^{i\gamma}r}{T}|>\tilde{\delta},\hbox{ for all }r\ge0\right\}.$$
The value of $\mathcal{L}^{\gamma}_{q;1/k}(f(x))(T)$ does not depend on the choice of $\gamma$ under the condition $e^{i\gamma}\in S_{d}$ due to Cauchy formula.
\end{lemma}

\begin{prop}\label{prop259}
Let $f$ be a function satisfying the properties in Definition~\ref{def203}, and $\tilde{\delta}>0$. Then, for every $\sigma\ge0$ one has
$$T^{\sigma}\sigma_{q}^{j}(\mathcal{L}^{\gamma}_{q;1/k}f(x))(T)=\mathcal{L}_{q;1/k}^{\gamma}\left(\frac{x^{\sigma}}{(q^{1/k})^{\sigma(\sigma-1)/2}}\sigma_{q}^{j-\frac{\sigma}{k}}f(x)\right)(T),$$
for every $T\in\mathcal{R}_{\gamma,\tilde{\delta}}\cap D(0,r_1)$, where $0<r_1\le q^{(\frac{1}{2}-\alpha)/k}/2$.
\end{prop}

We are also making use of Fourier transform and some of its properties, in the spirit of~\cite{lama14,ma15}.

\begin{prop}\label{prop267} Take $\mu>1,\beta>0$ and let $f\in E_{(\beta,\mu)}$. The inverse Fourier transform is defined by
$$\mathcal{F}^{-1}(f)(x)=\frac{1}{(2\pi)^{1/2}}\int_{-\infty}^{\infty}f(m)\exp\left(ixm\right) dm,$$
for $x\in\R$, which can be extended to an analytic function on the strip
\begin{equation}\label{e268bb}
H_{\beta}=\{z\in\C:|\Im(z)|<\beta\}.
\end{equation}
Let $\phi(m)=imf(m)\in E_{(\beta,\mu-1)}$. Then, we have
\begin{equation}\label{e272}
\partial_z\mathcal{F}^{-1}(f)(z)=\mathcal{F}^{-1}(\phi)(z),
\end{equation}
for every $z\in H_{\beta}$.

Let $g\in E_{(\beta,\mu)}$ and let $\psi(m)=\frac{1}{(2\pi)^{1/2}}(f\ast g)(m)$, the convolution product of $f$ and $g$, for all $m\in\R$. A direct application of Proposition~\ref{prop128} when choosing $b\equiv R\equiv Q\equiv 1$ allow us to affirm that the function $\psi$ is an element of $E_{(\beta,\mu)}$. Moreover, we have
\begin{equation}\label{e278}
\mathcal{F}^{-1}(f)(z)\mathcal{F}^{-1}(g)(z)=\mathcal{F}^{-1}(\psi)(z),
\end{equation}
for every $z\in H_{\beta}$.

\end{prop}

\section{Formal and analytic solutions to an auxiliary convolution problem}\label{seccion4}

Let $1\le k_1<k_2$ and $D\ge3$ be integer numbers.
Let $\kappa>0$ be defined by
\begin{equation}\label{e236}
\frac{1}{\kappa}=\frac{1}{k_1}-\frac{1}{k_2}.
\end{equation}

We also take $q\in\R$, $q>1$ and assume that for every $1\le \ell\le D-1$, $I_{\ell}$ stands for a finite nonempty set of nonnegative integers.

Let $d_{D}\ge1$ be an integer. For every $1\le \ell\le D-1$, we consider an integer $\delta_{\ell}\ge1$ and for each $\lambda\in I_{\ell}$, we choose integers $d_{\lambda,\ell}\ge 1$ , $\Delta_{\lambda,\ell}\ge0$. In addition to that, we make the assumption that

$$\quad\delta_1=1,\quad \delta_{\ell}<\delta_{\ell+1}\quad ,$$
for every $1\le\ell\le D-1$.

We make the hypotheses 
\begin{equation}\label{e180}
\Delta_{\lambda,\ell}\ge d_{\lambda,\ell},\quad\frac{d_{\lambda,\ell}}{k_2}+1\ge \delta_\ell,\qquad \frac{d_{D}-1}{k_2}+1\ge \delta_{\ell}
\end{equation}
for every $1\le \ell\le D-1$ and all $\lambda\in I_{\ell}$. Let $Q,R_{\ell}\in\C[X]$, $1\le \ell\le D$, with
\begin{equation}\label{e92}
\deg(Q)\ge\deg(R_{D})\ge \deg(R_{\ell}), \quad Q(im)\neq0,\quad R_{D}(im)\neq 0,
\end{equation}
for all $1\le \ell\le D-1$ and $m\in\R$. For every $1\le \ell\le D-1$ and all $\lambda\in I_{\ell}$, we choose the function $m\mapsto C_{\lambda,\ell}(m,\epsilon)$ in the space $E_{(\beta,\mu)}$ for some $\beta>0$ and $\mu>\deg(R_{D})+1$. In addition to that, we assume all these functions depend holomorphically on $\epsilon\in D(0,\epsilon_0)$ for some $\epsilon_0>0$, and also the existence of a positive constant $\tilde{C}_{\lambda,\ell}$ such that 
\begin{equation}\label{e223}
\left\|C_{\lambda,\ell}(m,\epsilon)\right\|_{(\beta,\mu)}\le \tilde{C}_{\lambda,\ell},
\end{equation}
for every $\epsilon\in D(0,\epsilon_0)$.

Let $F(T,m,\epsilon)=\sum_{n\ge0}F_{n}(m,\epsilon)T^n$ be a formal power series in $T$ with coefficients in $E_{(\beta,\mu)}$ which depend holomorphically on $\epsilon\in D(0,\epsilon_0)$. 

We assume the formal power series $\hat{\mathcal{B}}_{q,1/k_1}(F(T,m,\epsilon))(\tau)\in E_{(\beta,\mu)}[[\tau]]$, which depends holomorphically on $\epsilon\in D(0,\epsilon_0)$. Moreover, we assume uniform bounds on the perturbation parameter in the domain of convergence. More precisely, we assume there exists $C_{F}>0$, which does not depend on $\epsilon\in D(0,\epsilon_0)$, such that 
\begin{equation}\label{e270}
\left\|F_{n}(m,\epsilon)\right\|_{(\beta,\mu)}\le C_{F}\rho^{-n}q^{\frac{n(n-1)}{2k_1}}.
\end{equation}
The nature of $\hat{\mathcal{B}}_{q,1/k_1}(F(T,m,\epsilon))(\tau)$ allow us to affirm that this function can be extended to an unbounded sector of bisecting direction $d$, say $U_{d}$, under $q-$exponential bounds of $k_1$ type. Again, we assume uniformity on the bounds for the perturbation parameter in the sense that, if we denote this extension by $\psi_{k_1}$, then one has that
\begin{equation}\label{e268}
|\psi_{k_1}(\tau,m,\epsilon)|\le C_{\psi_{k_1}}\frac{e^{-\beta|m|}}{(1+|m|)^{\mu}}\exp\left(\frac{k_1\log^{2}|\tau+\delta|}{2\log(q)}+\alpha\log|\tau+\delta|\right),
\end{equation}
for some $C_{\psi_{k_1}}>0$, $\tau\in U_{d}\cup \overline{D}(0,\rho)$ and $m\in\R$, valid for all $\epsilon\in D(0,\epsilon_0)$.

We consider the equation
\begin{align}&Q(im)\sigma_{q}U(T,m,\epsilon)=T^{d_{D}}\sigma_{q}^{\frac{d_{D}}{k_2}+1}R_{D}(im)U(T,m,\epsilon)\nonumber\\
+&\sum_{\ell=1}^{D-1}\left(\sum_{\lambda\in I_{\ell}}T^{d_{\lambda,\ell}}\epsilon^{\Delta_{\lambda,\ell}-d_{\lambda,\ell}}\frac{1}{(2\pi)^{1/2}}\int_{-\infty}^{\infty}C_{\lambda,\ell}(m-m_1,\epsilon)R_{\ell}(im_1)U(q^{\delta_{\ell}}T,m_1,\epsilon)dm_1\right)+\sigma_{q}F(T,m,\epsilon).\label{e149}
\end{align}

\begin{prop}
Let $U_{h}\in E_{(\beta,\mu)}$ for 
$$h\in\left\{0,1,...,\max\{\max_{1\le \ell\le D-1,\lambda\in I_{\ell}}d_{\lambda,\ell},d_{D}\}\right\},$$
depending holomorphically on $\epsilon\in D(0,\epsilon_0)$.
Then, there exists a unique formal power series $\hat{U}(T,m,\epsilon)=\sum_{n\ge0}U_{n}(m,\epsilon)T^n\in E_{(\beta,\mu)}[[T]]$ which solves (\ref{e149}). The elements $U_{n}$ depend holomorphically on $\epsilon\in D(0,\epsilon_0)$.
\end{prop}
\begin{proof}
A formal power series $\hat{U}(T,m,\epsilon)=\sum_{n\ge0}U_{n}(m,\epsilon)T^n$ provides a formal solution of (\ref{e149}) if its coefficients satisfy the recursion formula
\begin{align}
&Q(im)U_n(m,\epsilon)q^n=R_{D}(im)U_{n-d_{D}}(m,\epsilon)q^{\left(\frac{d_{D}}{k_2}+1\right)(n-d_{D})}\label{e164}\\
&+\sum_{\ell=1}^{D-1}\sum_{\lambda\in I_{\ell}}\epsilon^{\Delta_{\lambda,\ell}-d_{\lambda,\ell}}q^{(n-d_{\lambda,\ell})\delta_{\ell}}\frac{1}{(2\pi)^{1/2}}\int_{-\infty}^{\infty}C_{\lambda,\ell}(m-m_1,\epsilon)R_{\ell}(im_{1})U_{n-d_{\lambda,\ell}}(m_1,\epsilon)dm_1\nonumber \\
&\hspace{9cm}+F_{n}(m,\epsilon)q^n,\nonumber
\end{align}
for every $\epsilon\in D(0,\epsilon_0)$, and $m\in\R$. Holomophicity of $U_{n}(m,\epsilon)$ for every $n\ge0$ comes from the previous recursion formula and assumption (\ref{e180}). Moreover, regarding Proposition~\ref{prop128} together with assumption (\ref{e92}), and Lemma~\ref{lema123} one derives $U_{n}(m,\epsilon)\in E_{(\beta,\mu)}$ for every $n\ge0$.  
\end{proof}

As it was explained in the introduction, a procedure of Borel-Laplace summation on the perturbation parameter is not valid in this framework from the growth nature of the source and the nature of the singularities associated to the equation in $\epsilon$. This phenomena is treated in in two steps: a first step in which Borel-Laplece summation procedure provides a holomorphic solution with too large growth at infinity so that Laplace transform is not available; and a second step, solving this difficulty, by means of an acceleration operator. 

\subsection{Analytic solutions for an auxiliary problem arising from the action of the formal $q-$Borel transform of order $k_1$}\label{seccion41}

We commence with the first step in the procedure to follow. For that purpose, we multiply both sides of equation (\ref{e149}) by $T^{k_1}$ to get 

\begin{align}&Q(im)T^{k_1}\sigma_{q}U(T,m,\epsilon)=T^{d_{D}+k_1}\sigma_{q}^{\frac{d_{D}}{k_2}+1}R_{D}(im)U(T,m,\epsilon)\nonumber\\
+&\sum_{\ell=1}^{D-1}\left(\sum_{\lambda\in I_{\ell}}T^{d_{\lambda,\ell}+k_1}\epsilon^{\Delta_{\lambda,\ell}-d_{\lambda,\ell}}\frac{1}{(2\pi)^{1/2}}\int_{-\infty}^{\infty}C_{\lambda,\ell}(m-m_1,\epsilon)R_{\ell}(im_1)U(q^{\delta_{\ell}}T,m_1,\epsilon)dm_1\right)\nonumber\\
&\hspace{10cm}+T^{k_1}\sigma_{q}F(T,m,\epsilon).\label{e216}
\end{align}

By applying the formal $q-$Borel transform of order $k_1$ at both sides of the quation (\ref{e216}) and bearing in mind the identity described in Proposition~\ref{prop157}, we arrive at

\begin{align}&Q(im)\frac{\tau^{k_1}}{(q^{1/k_1})^{k_1(k_1-1)/2}}w_{k_1}(\tau,m,\epsilon)=\frac{\tau^{d_D+k_1}}{(q^{1/k_1})^{(d_{D}+k_1)(d_{D}+k_1-1)/2}}\sigma_{q}^{-d_{D}/\kappa}R_{D}(im)w_{k_1}(\tau,m,\epsilon)\nonumber\\
+&\sum_{\ell=1}^{D-1}\left(\sum_{\lambda\in I_{\ell}}\frac{\epsilon^{\Delta_{\lambda,\ell}-d_{\lambda,\ell}}\tau^{d_{\lambda,\ell}+k_1}}{(q^{1/k_1})^{(d_{\lambda,\ell}+k_1)(d_{\lambda,\ell}+k_1-1)/2}}\sigma_q^{\delta_{\ell}-\frac{d_{\lambda,\ell}}{k_1}-1}\frac{1}{(2\pi)^{1/2}}(C_{\lambda,\ell}(m,\epsilon)\ast^{R_{\ell}}w_{k_1}(\tau,m,\epsilon))\right)\nonumber\\
&\hspace{10cm}+\frac{\tau^{k_1}}{(q^{1/k_1})^{k_{1}(k_1-1)/2}}\psi_{k_1}(\tau,m,\epsilon).\label{e224}
\end{align}

Here, we have put
\begin{equation}\label{e313}
w_{k_1}(\tau,m,\epsilon):=\hat{\mathcal{B}}_{q;1/k_1}(U(T,m,\epsilon))(\tau),\qquad \psi_{k_1}(\tau,m,\epsilon):=\hat{\mathcal{B}}_{q;1/k_1}(F(T,m,\epsilon))(\tau).
\end{equation}

The main aim of this section is to prove that $w_{k_1}$ is indeed a holomorphic function in a neighborhood of the origin in the variable $\tau$, with values in $E_{(\beta,\mu)}$, and holomorphic with respect to the perturbation parameter $\epsilon$. Moreover, $w_{k_1}$ can be extended in $\tau$ variable to an infinite sector under $q-$exponential growth of $\kappa$ type.

A fix point theorem in an appropriate Banach space is studied.

\begin{prop}\label{prop321}
Let $\varpi>0$. Under the hypotheses made at the beginning of Section~\ref{seccion4} on the functions involved in the construction of the equation (\ref{e224}), if $R_{D}$ and $Q$ are chosen so that $\sup_{m\in\R}\frac{|R_{D}(im)|}{|Q(im)|}$ is small enough, and there exist small enough positive constants $\zeta_{\psi_{k_1}}$ and $\zeta_{\lambda,\ell}$ for every $1\le \ell\le D-1$ and all $\lambda\in I_{\ell}$ with
\begin{equation}\label{e286}
 \tilde{C}_{\lambda,\ell}\le\zeta_{\lambda,\ell},\qquad  C_{\psi_{k_1}}\le\zeta_{\psi_{k_1}},
\end{equation}
then the equation (\ref{e224}) admits a unique solution $w_{k_1}^d(\tau,m,\epsilon)$ in the space $\hbox{Exp}^{d}_{(\kappa,\beta,\mu,\alpha,\rho)}$ such that $\bigl\|w_{k_1}^{d}(\tau,m,\epsilon)\bigr\|_{(\kappa,\beta,\mu,\alpha,\rho)}\le \varpi$ for every $\epsilon\in D(0,\epsilon_0)$.  Moreover, this function is holomorphic with respect to $\epsilon$ in $D(0,\epsilon_0)$.
\end{prop}
\begin{proof}
Let $\epsilon\in D(0,\epsilon_0)$.

Let $\mathcal{H}_{\epsilon}^{k_1}$ be the map defined by
\begin{align*}
&\mathcal{H}_{\epsilon}^{k_1}(w(\tau,m)):=\frac{(q^{1/k_1})^{k_1(k_1-1)/2}}{(q^{1/k_1})^{(d_{D}+k_1)(d_{D}+k_1-1)/2}}\tau^{d_D}\sigma_{q}^{-d_{D}/\kappa}\frac{R_{D}(im)}{Q(im)}w(\tau,m)\nonumber\\
+&\sum_{\ell=1}^{D-1}\left(\sum_{\lambda\in I_{\ell}}\frac{\epsilon^{\Delta_{\lambda,\ell}-d_{\lambda,\ell}}(q^{1/k_1})^{k_1(k_1-1)/2}}{(q^{1/k_1})^{(d_{\lambda,\ell}+k_1)(d_{\lambda,\ell}+k_1-1)/2}}\tau^{d_{\lambda,\ell}}\sigma_q^{\delta_{\ell}-\frac{d_{\lambda,\ell}}{k_1}-1}\frac{1}{(2\pi)^{1/2}}\frac{1}{{Q(im)}}(C_{\lambda,\ell}(m,\epsilon)\ast^{R_{\ell}}w_{k_1}(\tau,m,\epsilon))\right)\\
&\hspace{10cm}+\frac{1}{Q(im)}\psi_{k_1}(\tau,m,\epsilon)
\end{align*}

Let $\varpi>0$ and $w(\tau,m)\in\hbox{Exp}^{d}_{(\kappa,\beta,\mu,\alpha,\rho)}$ with $\left\|w(\tau,m)\right\|_{(\kappa,\beta,\mu,\alpha,\rho)}\le\varpi$.

Taking into account (\ref{e92}) one can apply Lemma~\ref{lema123}, and from Proposition~\ref{prop82}, one gets
\begin{equation}\label{e260}
\left\|\frac{R_{D}(im)}{Q(im)} \tau^{d_D}\sigma_{q}^{-d_{D}/\kappa}w(\tau,m)\right\|_{(\kappa,\beta,\mu,\alpha,\rho)}\le C_{2}C_{R_{D}Q}\left\|w(\tau,m)\right\|_{(\kappa,\beta,\mu,\alpha,\rho)}\le C_{2}C_{R_{D}Q}\varpi, 
\end{equation}
for some $C_2>0$, and where $C_{R_DQ}=\sup_{m\in\R}|R_{D}(im)|/|Q(im)|$. Let $1\le \ell\le D-1$ and $\lambda\in I_{\ell}$. Bearing in mind (\ref{e180}) and from Proposition~\ref{prop82} and Proposition~\ref{prop143} we have
\begin{align}
&\left\|\epsilon^{\Delta_{\lambda,\ell}-d_{\lambda,\ell}}\tau^{d_{\lambda,\ell}}\sigma_q^{\delta_{\ell}-\frac{d_{\lambda,\ell}}{k_1}-1}\frac{1}{{Q(im)}}(C_{\lambda,\ell}(m,\epsilon)\ast^{R_{\ell}}w_{k_1}(\tau,m,\epsilon))\right\|_{(\kappa,\beta,\mu,\alpha,\rho)}\hspace{4cm} \nonumber \\
&\le C_{2} C_{3}\epsilon_{0}^{\Delta_{\lambda,\ell}-d_{\lambda,\ell}}\left\|C_{\lambda,\ell}(m,\epsilon)\right\|_{(\beta,\mu)}\left\|w(\tau,m)\right\|_{(\kappa,\beta,\mu,\alpha,\rho)}
 \le C_2C_3\epsilon_0^{\Delta_{\lambda,\ell}-d_{\lambda,\ell}} \tilde{C}_{\lambda,\ell}\varpi\nonumber \\
&\le C_2C_3\epsilon_0^{\Delta_{\lambda,\ell}-d_{\lambda,\ell}} \zeta_{\lambda,\ell}\varpi,
\label{e294}
\end{align}
for every $1\le \ell\le D-1$ and all $\lambda\in I_{\ell}$. We recall $\tilde{C}_{\lambda,\ell}$ is defined in (\ref{e223}). 

The polynomial $Q$ satisfies that $Q(im)\neq 0$ for all $m\in\R$. This entails that $|Q(im)|\ge C_{Q}$ for every $m\in\R$, for some $C_{Q}>0$. We depart from $\psi_{k_1}\in\hbox{Exp}^{d}_{(k_1,\beta,\mu,\alpha,\rho)}\subseteq\hbox{Exp}^{d}_{(\kappa,\beta,\mu,\alpha,\rho)}$, where the inclusion is a continuous map. Then, 
\begin{equation}\label{e315}
\left\|\frac{1}{Q(im)}\psi_{k_1}(\tau,m,\epsilon)\right\|_{(\kappa,\beta,\mu,\alpha,\rho)}\le \frac{1}{C_{Q}}\left\|\psi_{k_1}(\tau,m,\epsilon)\right\|_{(\kappa,\beta,\mu,\alpha,\rho)}\le \frac{1}{C_{Q}}C_{\psi_{k_1}}\le\frac{1}{C_{Q}}\zeta_{\psi_{k_1}}.
\end{equation}

Regarding (\ref{e260}), (\ref{e294})and (\ref{e315}), we get 
\begin{align*}
\left\|\mathcal{H}_{\epsilon}^{k_1}(w(\tau,m))\right\|_{(\kappa,\beta,\mu,\alpha,\rho)}&\le\frac{(q^{1/k_1})^{k_1(k_1-1)/2}}{(q^{1/k_1})^{(d_{D}+k_1)(d_{D}+k_1-1)/2}}C_2C_{R_{D}Q}\varpi\\
&+\sum_{\ell=1}^{D-1}\sum_{\lambda\in I_{\ell}}\frac{\epsilon_0^{\Delta_{\lambda,\ell}-d_{\lambda,\ell}}(q^{1/k_1})^{k_1(k_1-1)/2}}{(q^{1/k_1})^{(d_{\lambda,\ell}+k_1)(d_{\lambda,\ell}+k_1-1)/2}}\frac{1}{(2\pi)^{1/2}}C_2C_3\zeta_{\lambda,\ell}\varpi+\frac{1}{C_{Q}}\zeta_{\psi_{k_1}}
\end{align*}

Now, we choose $C_{R_{D}Q},\zeta_{\psi_{k_1}}$ and $\zeta_{\lambda,\ell}$ for every $1\le \ell\le D-1$ and $\lambda \in I_{\ell}$ such that
\begin{align*}
&\frac{(q^{1/k_1})^{k_1(k_1-1)/2}}{(q^{1/k_1})^{(d_{D}+k_1)(d_{D}+k_1-1)/2}}C_2C_{R_{D}Q}\varpi\\
&+\sum_{\ell=1}^{D-1}\sum_{\lambda\in I_{\ell}}\frac{\epsilon_0^{\Delta_{\lambda,\ell}-d_{\lambda,\ell}}(q^{1/k_1})^{k_1(k_1-1)/2}}{(q^{1/k_1})^{(d_{\lambda,\ell}+k_1)(d_{\lambda,\ell}+k_1-1)/2}}\frac{1}{(2\pi)^{1/2}}C_2C_3\zeta_{\lambda,\ell}\varpi+\frac{1}{C_{Q}}\zeta_{\psi_{k_1}}\le\varpi.
\end{align*}

One derives that the operator $\mathcal{H}^{k_1}_{\epsilon}$ maps $\overline{D}(0,\varpi)\subseteq\hbox{Exp}^{d}_{(\kappa,\beta,\mu,\alpha,\rho)}$ into itself. Let $w^{1}_{k_1},w^{2}_{k_1}\in \hbox{Exp}^{d}_{(\kappa,\beta,\mu,\alpha,\rho)}$, with $\bigl\|w^{j}_{k_1}\bigr\|_{(\kappa,\beta,\mu,\alpha,\rho)}\le\varpi$ for $j=1,2$.

Analogous estimates as before allow us to prove that
\begin{align*}
&\left\|\mathcal{H}^{k_1}_{\epsilon}(w^1_{k_1}(\tau,m))-\mathcal{H}^{k_1}_{\epsilon}(w^2_{k_1}(\tau,m))\right\|_{(\kappa,\beta,\mu,\alpha,\rho)}\\
&\le \frac{(q^{1/k_1})^{k_1(k_1-1)/2}C_2C_{R_{D}Q}}{(q^{1/k_1})^{(d_{D}+k_1)(d_{D}+k_1-1)/2}}\bigl\|w^{1}_{k_1}(\tau,m)-w^{2}_{k_1}(\tau,m)\bigr\|_{(\kappa,\beta,\mu,\alpha,\rho)}\\
&+\sum_{\ell=1}^{D-1}\sum_{\lambda\in I_{\ell}}\frac{\epsilon_0^{\Delta_{\lambda,\ell}-d_{\lambda,\ell}}(q^{1/k_1})^{k_1(k_1-1)/2}}{(q^{1/k_1})^{(d_{\lambda,\ell}+k_1)(d_{\lambda,\ell}+k_1-1)/2}}\frac{1}{(2\pi)^{1/2}}C_2C_3\zeta_{\lambda,\ell}\bigl\|w^{1}_{k_1}(\tau,m)-w^{2}_{k_1}(\tau,m)\bigr\|_{(\kappa,\beta,\mu,\alpha,\rho)}
\end{align*}

We choose $C_{R_{D}Q}$ and $\zeta_{\lambda,\ell}$ for every $1\le \ell\le D-1$ and $\lambda \in I_{\ell}$ such that
$$ \frac{(q^{1/k_1})^{k_1(k_1-1)/2}C_2C_{R_{D}Q}}{(q^{1/k_1})^{(d_{D}+k_1)(d_{D}+k_1-1)/2}}+\sum_{\ell=1}^{D-1}\sum_{\lambda\in I_{\ell}}\frac{\epsilon_0^{\Delta_{\lambda,\ell}-d_{\lambda,\ell}}(q^{1/k_1})^{k_1(k_1-1)/2}}{(q^{1/k_1})^{(d_{\lambda,\ell}+k_1)(d_{\lambda,\ell}+k_1-1)/2}}\frac{1}{(2\pi)^{1/2}}C_2C_3\zeta_{\lambda,\ell}\le \frac{1}{2},$$

and conclude that 
$$\left\|\mathcal{H}^{k_1}_{\epsilon}(w^1_{k_1}(\tau,m))-\mathcal{H}^{k_1}_{\epsilon}(w^2_{k_1}(\tau,m))\right\|_{(\kappa,\beta,\mu,\alpha,\rho)}\le \frac{1}{2}\left\|w^1_{k_1}(\tau,m))-w^2_{k_1}(\tau,m)\right\|_{(\kappa,\beta,\mu,\alpha,\rho)}.$$

The closed disc $\overline{D}(0,\varpi)\subseteq\hbox{Exp}^{d}_{(\kappa,\beta,\mu,\alpha,\rho)}$ is a complete metric space for the norm $\left\|\cdot\right\|_{(\kappa,\beta,\mu,\alpha,\rho)}$. Then, the operator $\mathcal{H}^{k_1}_{\epsilon}$ is a contractive map from $\overline{D}(0,\varpi)$ into itself. The classical contractive mapping theorem states the existence of a unique fixed point, say $w^{d}_{k_1}(\tau,m,\epsilon)$. Holomorphy of $w^{d}_{k_1}(\tau,m,\epsilon)$ with respect to $\epsilon$ is guaranteed by construction, and also one has
$w^{d}_{k_1}(\tau,m,\epsilon)$ is a solution of (\ref{e224}).
\end{proof}

The next step consists on studying the solutions of a second auxiliary problem, derived from (\ref{e149}). Its solution is linked to that of (\ref{e224}) by means of some appropriate $q-$Laplace transform which plays the role of an acceleration operator. The main aim of the following results is to conclude that the acceleration of the function $w_{k_{1}}^d$ obtained in Proposition~\ref{prop321} coincides with the analytic solution of the novel auxiliary equation under study. 

We first establish an accelerator-like result on the function $\psi_{k_{1}}$, defined in (\ref{e313}).

\begin{lemma}\label{lema403}
Let $\tilde{\delta}>0$. The function 
\begin{equation}\label{e399}
\tau\mapsto \psi_{k_2}(\tau,m,\epsilon):=\mathcal{L}^{d}_{q;1/\kappa}(h\mapsto \psi_{k_1}(h,m,\epsilon))(\tau)=\frac{1}{\pi_{q^{1/\kappa}}}\int_{L_{d}}\frac{\psi_{k_1}(u,m,\epsilon)}{\Theta_{q^{1/\kappa}}\left(\frac{u}{\tau}\right)}\frac{du}{u}
\end{equation}
is analytic in the set $\mathcal{R}_{d,\tilde{\delta}}$. Moreover, the function $(\tau,m)\mapsto \psi_{k_2}(\tau,m,\epsilon)$ is continuous for $m\in\R$ and $\tau\in\mathcal{R}_{d,\tilde{\delta}}$, and depends holomorphically on $\epsilon\in D(0,\epsilon_0)$. Moreover, there exist $C_{\psi_{k_2}}>0$ and $\nu\in\R$ such that
\begin{equation}\label{e405}
\left|\psi_{k_2}(\tau,m,\epsilon)\right|\le C_{\psi_{k_2}}e^{-\beta|m|}(1+|m|)^{-\mu}\exp\left(\frac{k_2\log^2|\tau|}{2\log(q)}+\nu\log|\tau|\right),
\end{equation}
for every $\tau\in\mathcal{R}_{d,\tilde{\delta}}$, $m\in\R$ and $\epsilon\in D(0,\epsilon_0)$.
\end{lemma}

\begin{proof}
We first rewrite (\ref{e268}).

Direct calculations allow us to affirm the existence of a constant $C_{41}$, only depending on $\delta$, such that
$$\frac{k_1\log^{2}|\tau+\delta|}{2\log(q)}+\alpha\log|\tau+\delta|\le \frac{k_1\log^{2}|\tau|}{2\log(q)}+\alpha\log|\tau|+C_{41},$$
for every $\tau\in S_{d}$. This entails (\ref{e268}) can be rewritten in the form
\begin{equation}\label{e406}
|\psi_{k_1}(\tau,m,\epsilon)|\le \tilde{C}_{\psi_{k_1}}\frac{e^{-\beta|m|}}{(1+|m|)^{\mu}}\exp\left(\frac{k_1\log^{2}|\tau|}{2\log(q)}+\alpha\log|\tau|\right),
\end{equation}
for some constant $\tilde{C}_{\psi_{k_1}}>0$. We recall that $k_1\le \kappa$ (see (\ref{e236})) so that
\begin{equation}\label{e416}
|\psi_{k_1}(\tau,m,\epsilon)|\le \tilde{C}_{\psi_{k_1}}\frac{e^{-\beta|m|}}{(1+|m|)^{\mu}}\exp\left(\frac{\kappa\log^{2}|\tau|}{2\log(q)}+\alpha\log|\tau|\right).
\end{equation}
This remains valid for all $\epsilon\in D(0,\epsilon_0)$, $\tau\in U_{d}$ and $m\in\R$. Moreover, for every $\tau\in \overline{D}(0,\rho)$, the functions $\log^2|\tau+\delta|$ and $\log|\tau+\delta|$ are upper bounded. As a consequence, there exists $\check{C}_{\psi_{k_1}}>0$ such that 
\begin{equation}\label{e410}
|\psi_{k_1}(\tau,m,\epsilon)|\le\check{C}_{\psi_{k_1}},
\end{equation}
for every $\epsilon\in D(0,\epsilon_0)$, $\tau\in \overline{D}(0,\rho)$ and $m\in\R$.
In view of (\ref{e410}) and (\ref{e416}) one can follow the construction described in Definition~\ref{def203} in order to affirm that $\psi_{k_2}$ is well-defined as considered in (\ref{e399}). More precisely, given $\tilde{\delta}>0$, $(\tau,m,\epsilon)\mapsto\psi_{k_2}(\tau,m,\epsilon)$ turns out to be a continuous and bounded function defined in $(\mathcal{R}_{d,\tilde{\delta}}\cap D(0,r_1))\times\R\times D(0,\epsilon_0)$ for any $0\le r_1\le q^{(1/2-\alpha)/\kappa}/2$, and holomorphic with respect to $\tau$ variable in $\mathcal{R}_{d,\tilde{\delta}}\cap D(0,r_1)$. In addition to that, from (\ref{e406}) and (\ref{e200}) one has
\begin{align}
&\left|\mathcal{L}_{q;1/\kappa}^{d}(h\mapsto\psi_{k_1}(h,m,\epsilon))(\tau)\right|=\left|\frac{1}{\pi_{q^{1/\kappa}}}\int_{L_{d}}\frac{\psi_{k_1}(u,m,\epsilon)}{\Theta_{q^{1/\kappa}}\left(\frac{u}{\tau}\right)}\frac{du}{u}\right|\nonumber \\
&\le \frac{\tilde{C}_{\psi_{k_1}}e^{-\beta|m|}|\tau|^{1/2}}{(1+|m|)^{\mu}C_{q,\kappa}\tilde{\delta}\pi_{q^{1/\kappa}}}\int_{0}^{\infty}\exp\left(\frac{k_1\log^2(r)}{2\log(q)}+\alpha\log(r)\right)\frac{1}{\exp\left(\frac{\kappa\log^2\left(\frac{r}{|\tau|}\right)}{2\log(q)}\right)r^{1/2}}\frac{dr}{r}\label{e426}
\end{align}

for every $\tau$ such that $\left|1+\frac{re^{id}}{\tau}\right|>\tilde{\delta}$ for all $r\ge0$, $m\in\R$ and $\epsilon\in D(0,\epsilon_0)$. One has
$$r^{\alpha-3/2}\exp\left(\frac{k_1\log^2r}{2\log(q)}-\frac{\kappa\log^2\left(r/|\tau|\right)}{2\log(q)}\right)=r^{\alpha-3/2+\frac{\kappa\log|\tau|}{\log(q)}}\exp\left(-\frac{(\kappa-k_1)\log^2(r)}{2\log(q)}-\frac{\kappa\log^2|\tau|}{2\log(q)}\right)$$
This last equality and (\ref{e426}) allow us to write
\begin{align}
\left|\mathcal{L}_{q;1/\kappa}^{d}(\psi_{k_1}(h,m,\epsilon))(\tau)\right|&\le 
\frac{\tilde{C}_{\psi_{k_1}}e^{-\beta|m|}|\tau|^{1/2}}{(1+|m|)^{\mu}C_{q,\kappa}\tilde{\delta}\pi_{q^{1/\kappa}}}\exp\left(-\frac{\kappa\log^2|\tau|}{2\log(q)}\right)\nonumber \\
&\times \int_{0}^{\infty}r^{\alpha-3/2+\frac{\kappa\log|\tau|}{\log(q)}}\exp\left(-\frac{(\kappa-k_1)\log^2(r)}{2\log(q)}\right)dr.\label{e435}
\end{align}

Let $\tau$ be chosen as above. We put $m_1=\alpha-3/2+\frac{\kappa\log|\tau|}{\log(q)}$, $m_{11}=m_{1}+1/2$, $m_{12}=m_1+2$ and $m_2=\frac{\kappa-k_1}{2\log(q)}$, and study
$$\int_{0}^{\infty}r^{m_{1}}e^{-m_2\log^2(r)}dr=\int_{0}^{1}\frac{1}{r^{1/2}}r^{m_{11}}e^{-m_2\log^2(r)}dr+\int_{1}^{\infty}\frac{1}{r^{2}}r^{m_{12}}e^{-m_2\log^2(r)}dr.$$

For $j=1,2$, the function $x\mapsto h_j(x)= x^{m_{1j}}\exp\left(-m_2\log^2(x)\right)$ attains its maximum value for $x\in[0,\infty)$ at $x_{j0}=\exp(\frac{m_{1j}}{2m_2})$. One has $h(x_{j0})=\exp(\frac{m_{1j}^2}{4m_2})$. Direct calculations show the existence of real constants $C_{42},C_{43}$, with $C_{42}>0$, only depending on $k_1,k_2,q,$ such that
$$\int_{0}^{\infty}r^{m_{1}}e^{-m_2\log^2(r)}dr\le C_{42}|\tau|^{C_{43}}\exp\left(\frac{\kappa^2\log^2|\tau|}{2\log(q)(\kappa-k_1)}\right).$$

Hence, (\ref{e435}) is estimated from above by
$$C_{44}\frac{\tilde{C}_{\psi_{k_1}}e^{-\beta|m|}|\tau|^{1/2+C_{43}}}{(1+|m|)^{\mu}}\exp\left(-\frac{\kappa\log^2|\tau|}{2\log(q)}\right)\exp\left(\frac{\kappa^2\log^2|\tau|}{2\log(q)(\kappa-k_1)}\right)$$
for some $C_{44}>0$ only depending on $k_1,k_2,q$.

Then, (\ref{e405}) follows from the fact that 
$$-\kappa+\frac{\kappa^2}{\kappa-k_1}=k_2,$$
and by taking $\nu=1/2+C_{43}$.
\end{proof}

\subsection{Analytic solutions for an auxiliary problem arising from the action of the formal $q-$Borel transform of order $k_2$}\label{seccion42}

In the previous section, we have studied the problem arisen from the application of the formal $q-$Borel transform of order $k_1$ to the equation (\ref{e149}). As a second step, we study a second auxiliary equation coming from the application of the formal $q-$Borel transform of order $k_2$ to the equation (\ref{e149}).

For that purpose, we multiply at both sides of the equality (\ref{e149}) by $T^{k_2}$ to get

\begin{align}&Q(im)T^{k_2}\sigma_{q}U(T,m,\epsilon)=T^{d_{D}+k_2}\sigma_{q}^{\frac{d_{D}}{k_2}+1}R_{D}(im)U(T,m,\epsilon)\nonumber\\
+&\sum_{\ell=1}^{D-1}\left(\sum_{\lambda\in I_{\ell}}T^{d_{\lambda,\ell}+k_2}\epsilon^{\Delta_{\lambda,\ell}-d_{\lambda,\ell}}\frac{1}{(2\pi)^{1/2}}\int_{-\infty}^{\infty}C_{\lambda,\ell}(m-m_1,\epsilon)R_{\ell}(im_1)U(q^{\delta_{\ell}}T,m_1,\epsilon)dm_1\right)\nonumber\\
&\hspace{10cm}+T^{k_2}\sigma_{q}F(T,m,\epsilon).\label{e472}
\end{align}

We take formal $q-$Borel transform of order $k_2$ at both sides of the previous equation. By means of the property held by formal $q-$Borel transform described in Proposicion~\ref{prop157}, we get

\begin{align}&Q(im)\frac{\tau^{k_2}}{(q^{1/k_2})^{k_2(k_2-1)/2}}\hat{w}_{k_2}(\tau,m,\epsilon)=R_{D}(im)\frac{\tau^{d_D+k_2}}{(q^{1/k_2})^{(d_{D}+k_2)(d_{D}+k_2-1)/2}}\hat{w}_{k_2}(\tau,m,\epsilon)\nonumber\\
+&\sum_{\ell=1}^{D-1}\left(\sum_{\lambda\in I_{\ell}}\frac{\epsilon^{\Delta_{\lambda,\ell}-d_{\lambda,\ell}}\tau^{d_{\lambda,\ell}+k_2}}{(q^{1/k_2})^{(d_{\lambda,\ell}+k_2)(d_{\lambda,\ell}+k_2-1)/2}}\sigma_q^{\delta_{\ell}-\frac{d_{\lambda,\ell}}{k_2}-1}\frac{1}{(2\pi)^{1/2}}(C_{\lambda,\ell}(m,\epsilon)\ast^{R_{\ell}}\hat{w}_{k_2}(\tau,m,\epsilon))\right)\nonumber\\
&\hspace{10cm}+\frac{\tau^{k_2}}{(q^{1/k_2})^{k_{2}(k_2-1)/2}}\hat{\psi}_{k_2}(\tau,m,\epsilon),\label{e479}
\end{align}
where 
$$\hat{w}_{k_2}(\tau,m,\epsilon):=\hat{\mathcal{B}}_{q;1/k_2}(U(T,m,\epsilon))(\tau),\qquad \hat{\psi}_{k_2}(\tau,m,\epsilon):=\hat{\mathcal{B}}_{q;1/k_2}(F(T,m,\epsilon))(\tau).$$

It is worth mentioning that we have assumed a $q-$Gevrey growth of order $k_1$ related to the elements $(F_{n}(m,\epsilon))_{n\ge0}$ (see (\ref{e270})). From the fact that $k_2>k_1$, one can only affirm that $\hat{\psi}_{k_2}(\tau,m,\epsilon)$ is a formal power series in $\tau$, with coefficients in the space $E_{(\beta,\mu)}$. This point is crucial to understand the cause of coming up to two different $q-$Borel-Laplace transformations to attain our aims.

We now proceed to substitute this formal element by an acceleration of $\psi_{k_1}$, named $\psi_{k_2}$, constructed in Lemma~\ref{lema403} and solve the equation arising from this substitution. Heuristically speaking, an excessive type of growth in the transformation of length $k_2$ is reduced in two steps: a first one related to $k_1$ ($k_1<k_2$) and a second step accelerating up to $k_2$. The splitting of the problem would help us to attain convergence.

Following this plan, we now consider the equation (\ref{e479}) in which $\hat{\psi}_{k_{2}}$ is substituted by $\psi_{k_2}$ constructed in Lemma~\ref{lema403}. Namely, we study the equation

\begin{align}
&Q(im)\frac{\tau^{k_2}}{(q^{1/k_2})^{k_2(k_2-1)/2}}w_{k_2}(\tau,m,\epsilon)=R_{D}(im)\frac{\tau^{d_D+k_2}}{(q^{1/k_2})^{(d_{D}+k_2)(d_{D}+k_2-1)/2}}w_{k_2}(\tau,m,\epsilon)\nonumber\\
+&\sum_{\ell=1}^{D-1}\left(\sum_{\lambda\in I_{\ell}}\frac{\epsilon^{\Delta_{\lambda,\ell}-d_{\lambda,\ell}}\tau^{d_{\lambda,\ell}+k_2}}{(q^{1/k_2})^{(d_{\lambda,\ell}+k_2)(d_{\lambda,\ell}+k_2-1)/2}}\sigma_q^{\delta_{\ell}-\frac{d_{\lambda,\ell}}{k_2}-1}\frac{1}{(2\pi)^{1/2}}(C_{\lambda,\ell}(m,\epsilon)\ast^{R_{\ell}}w_{k_2}(\tau,m,\epsilon))\right)\nonumber\\
&\hspace{10cm}+\frac{\tau^{k_2}}{(q^{1/k_2})^{k_{2}(k_2-1)/2}}\psi_{k_2}(\tau,m,\epsilon).\label{e492}
\end{align}

We now make this assumptions on $R_{D}$ and $Q$. Let $\eta_{Q,R_{D}}>0$ and $d_{Q,R_{D}}\in\R$ such that 
\begin{equation}\label{e496}
\frac{Q(im)}{R_{D}(im)}\in S_{Q,R_{D}},\quad m\in\R,
\end{equation}
where $S_{Q,R_{D}}$ is the unbounded set
$$S_{Q,R_{D}}=\left\{z\in\C:|z|\ge r_{Q,R_{D}},|\arg(z)-d_{Q,R_{D}}|\le \eta_{Q,R_{D}}\right\}.$$
We consider the polynomial 
\begin{equation}\label{e558}
P_{m}(\tau)=\frac{Q(im)}{(q^{1/k_2})^{k_2(k_2-1)/2}}-\frac{R_{D}(im)}{(q^{1/k_2})^{\frac{(d_{D}+k_2)(d_{D}+k_2-1)}{2}}}\tau^{d_D},
\end{equation}
and factorize it 
$$P_{m}(\tau)=-\frac{R_{D}(im)}{(q^{1/k_2})^{\frac{(d_{D}+k_2)(d_{D}+k_2-1)}{2}}}\prod_{l=0}^{d_{D}-1}(\tau-q_l(m)),$$
with
$$
q_{l}(m)=\left(\frac{|Q(im)|}{|R_{D}(im)|}(q^{1/k_2})^{\frac{(d_{D}+k_2)(d_{D}+k_2-1)-k_2(k_2-1)}{2}}\right)^{1/d_{D}}\nonumber\exp\left(\sqrt{-1}\left(\frac{1}{d_{D}}\arg\left(\frac{Q(im)}{R_{D}(im)}\right)+\frac{2\pi l}{d_{D}}\right)\right)\label{e506},
$$
for every $0\le l\le d_{D}-1$. Moreover, we establish some conditions on $S_{Q,R_{D}}$ with respect to $S_{d}$ and $\mathcal{R}_{d}^{b}$. Indeed, let $\rho_1>0$ such that $\mathcal{R}_{d}^{b}\subseteq \overline{D}(0,\rho_1)$ and assume
\begin{enumerate}
\item[1)] There exists $M_1>0$ such that
\begin{equation}\label{e513}
|\tau-q_{l}(m)|\ge M_{1}(1+|\tau|),
\end{equation}
for every $0\le l\le d_{D}-1$, all $m\in\R$ and  $\tau\in S_{d}\cup \overline{D}(0,\rho_1)$. Indeed, in view of (\ref{e496}) and (\ref{e506}), one may choose $r_{Q,R_{D}}$ so that $|q_l(m)|>2\rho_1$ for every $m\in\R$, $0\le l\le d_{D}-1$. In addition to that, $q_{\ell}(m)$ is a root of $P_m(\tau)$. This entails
$$q_{\ell}(m)^{d_{D}}=\frac{Q(im)}{R_{D}(im)}(q^{1/k_2})^{\frac{(d_{D}+k_2)(d_{D}+k_2-1)-k_2(k_2-1)}{2}}\in S_{Q,R_{D}}$$
for every $m\in\R$, in view of (\ref{e496}). As a matter of fact, one can choose small enough $\eta_{Q,R_{D}}$ so that every $q_{\ell}(m)$ lies in a finite family of infinite sectors with vertex at the origin and such that there exists $d\in\R$ so that $S_{d}$ does have empty intersection with all such infinite sectors.

Under these assumptions, one can choose $S_d$ under the property that $q_l(m)/\tau$ does not belong to some open disc centered at 1 for every $0\le l\le d_{D}-1$, $\tau\in S_{d}$ and $m\in\R$. This configuration guarantees (\ref{e513}) is fulfilled.
\item[2)] There exists $M_2>0$ and $l_0\in\{0,...,d_{D}-1\}$ such that
\begin{equation}\label{e520}
|\tau-q_{\ell_0}(m)|\ge M_2|q_{l_0}(m)|,
\end{equation}
for every $m\in\R$ and $\tau\in  S_{d}\cup\overline{D}(0,\rho_1)$. Under assumption 1), we notice that for any fixed $0\le l_0\le d_{D}-1$, the quotient $\tau/q_{l_0}(m)$ has positive distance to 1, for every $\tau\in\overline{D}(0,\rho_1)\cup S_{d}$, and $m\in\R$. This implies (\ref{e520}) for some small enough $M_2$.  
\end{enumerate}

Under the previous situation, one derives
\begin{align}
|P_{m}(\tau)|&\ge M_{1}^{d_{D}-1}M_2\frac{|R_{D}(im)|}{(q^{1/k_2})^{\frac{(d_{D}+k_2)(d_{D}+k_2-1)}{2}}}\left(\frac{|Q(im)|}{|R_{D}(im)|}(q^{1/k_2)^{\frac{(d_{D}+k_2)(d_{D}+k_2-1)-k_2(k_2-1)}{2}}}\right)^{1/d_{D}}\nonumber \\
&\times (1+|\tau|)^{d_{D}-1}\ge C_{P}(r_{Q,R_{D}})^{1/d_{D}}|R_{D}(im)|(1+|\tau|)^{d_{D}-1},\label{e529}
\end{align}
for some positive constant $C_P$, valid for every $\tau\in \overline{D}(0,\rho_1)\cup S_{d}$ and $m\in\R$.

Let $\mathcal{R}_{d}^{b}$ be a bounded sector with vertex at 0 and bisecting direction $d$. We assume $\mathcal{R}_{d}^{b}\subseteq \overline{D}(0,\rho_1)$.

\begin{prop}\label{prop571}
Let $\varpi>0$. Under the hypotheses made at the beginning of Section~\ref{seccion4} on the elements involved in the construction of equation (\ref{e492}), under assumptions 1) and 2) above, if there exist small enough positive constants $\zeta_{\psi_{k_2}},\zeta_{\lambda,\ell}$ for $1\le \ell\le D-1$ and $\lambda\in I_{\ell}$ such that 
\begin{equation}\label{e571}
\tilde{C}_{\lambda,\ell}\le\zeta_{\lambda,\ell}\qquad C_{\psi_{k_2}}\le\zeta_{\psi_{k_2}},
\end{equation}
then, for every $\epsilon\in D(0,\epsilon_0)$, the equation (\ref{e492}) admits a unique solution $w_{k_2}^{d}(\tau,m,\epsilon)$ in the space $\hbox{Exp}^{d}_{(k_2,\beta,\mu,\nu)}$, for $\nu\in\R$ determined in (\ref{e405}). Moreover, $\bigl\|w_{k_2}^{d}(\tau,m,\epsilon)\bigr\|_{(k_2,\beta,\mu,\nu)}\le \varpi$, and this function is holomorphic with respect to $\epsilon$ in $D(0,\epsilon_0)$.
\end{prop}

\begin{proof}

We recall $\tilde{C}_{\lambda,\ell}$ and $C_{\psi_{k_2}}$ are stated in (\ref{e223}) and (\ref{e405}) respectively.

Let $\epsilon\in D(0,\epsilon_0)$. We consider the map $\mathcal{H}_{\epsilon}^{k_2}$ defined by
\begin{align}
\mathcal{H}^{k_2}_{\epsilon}(w(\tau,m)):=&\frac{1}{P_{m}(\tau)}\sum_{\ell=1}^{D-1}\left(\sum_{\lambda\in I_{\ell}}\frac{\epsilon^{\Delta_{\lambda,\ell}-d_{\lambda,\ell}}\tau^{d_{\lambda,\ell}}\sigma_q^{\delta_{\ell}-\frac{d_{\lambda,\ell}}{k_2}-1}}{(q^{1/k_2})^{(d_{\lambda,\ell}+k_2)(d_{\lambda,\ell}+k_2-1)/2}}\frac{1}{(2\pi)^{1/2}}(C_{\lambda,\ell}(m,\epsilon)\ast^{R_{\ell}}w(\tau,m))\right)\nonumber\\
&+\frac{1}{(q^{1/k_2})^{k_{2}(k_2-1)/2}}\frac{1}{P_{m}(\tau)}\psi_{k_2}(\tau,m,\epsilon).\label{e578}
\end{align}

Let $\varpi>0$, and take $w(\tau,m)\in\hbox{Exp}^{d}_{(k_2,\beta,\mu,\nu)}$ with $\left\|w(\tau,m)\right\|_{(k_2,\beta,\mu,\nu)}\le\varpi$. For every $1\le \ell\le D-1$ and $\lambda\in I_{\ell}$ we write
\begin{align}
&\frac{\epsilon^{\Delta_{\lambda,\ell}-d_{\lambda,\ell}}}{P_{m}(\tau)}\tau^{d_{\lambda,\ell}}\sigma_{q}^{\delta_{\ell}-\frac{d_{\lambda,\ell}}{k_2}-1}\left(C_{\lambda,\ell}(m,\epsilon)\ast^{R_{\ell}}w(\tau,m)\right)\nonumber \\
&=\frac{\epsilon^{\Delta_{\lambda,\ell}-d_{\lambda,\ell}}}{C_{P}(r_{Q,R_{D}})^{1/d_{D}}}\frac{C_{P}(r_{Q,R_{D}})^{1/d_{D}}R_{D}(im)}{P_{m}(\tau)}\tau^{d_{\lambda,\ell}}\sigma_{q}^{\delta_{\ell}-\frac{d_{\lambda,\ell}}{k_2}-1}\frac{1}{R_{D}(im)}\left(C_{\lambda,\ell}(m,\epsilon)\ast^{R_{\ell}}w(\tau,m)\right) \label{e582}
\end{align}

In view of the properties described at the beginning of Section~\ref{seccion4}, one can apply Proposition~\ref{prop197} to obtain that
\begin{equation}\label{e588}
\left\|\frac{1}{R_{D}(im)}\left(C_{\lambda,\ell}(m,\epsilon)\ast^{R_{\ell}}w(\tau,m)\right)\right\|_{(k_2,\beta,\mu,\nu)}\le \tilde{C}_3\left\|C_{\lambda,\ell}\right\|_{(\beta,\mu)}\left\|w(\tau,m)\right\|_{(k_2,\beta,\mu,\nu)}\le \tilde{C}_3\tilde{C}_{\lambda,\ell}\varpi,
\end{equation}
for some $\tilde{C}_{3}>0$, valid for every $1\le \ell\le D-1$ and all $\lambda\in I_{\ell}$. Let $\gamma_1$ in Proposition~\ref{prop188} be the value $d_D-1$ and put $a_{\gamma_1}(\tau):=\frac{C_{P}(r_{Q,R_{D}})^{1/d_{D}}R_{D}(im)}{P_{m}(\tau)}$. From (\ref{e529}) one has $|a_{\gamma_1}(\tau)|\le\frac{1}{(1+|\tau|)^{d_{D}-1}}$ for every $\tau\in S_d\cup \overline{\mathcal{R}_{d}^{b}}$ and $m\in\R$. In addition to this, and bearing in mind (\ref{e180}), one can apply Proposition~\ref{prop188} to conclude that
$$\left\|\frac{C_{P}(r_{Q,R_{D}})^{1/d_{D}}R_{D}(im)}{P_{m}(\tau)}\tau^{d_{\lambda,\ell}}\sigma_{q}^{\delta_{\ell}-\frac{d_{\lambda,\ell}}{k_2}-1}\frac{1}{R_{D}(im)}\left(C_{\lambda,\ell}(m,\epsilon)\ast^{R_{\ell}}w(\tau,m)\right)\right\|_{(k_2,\beta,\mu,\nu)}$$
is upper bounded by
$$\tilde{C}_{2}\left\|\frac{1}{R_{D}(im)}\left(C_{\lambda,\ell}(m,\epsilon)\ast^{R_{\ell}}w(\tau,m)\right)\right\|_{(k_2,\beta,\mu,\nu)}\le\tilde{C}_{2} \tilde{C}_3\tilde{C}_{\lambda,\ell}\varpi\le \tilde{C}_{2} \tilde{C}_3\zeta_{\lambda,\ell}\varpi.$$

We observe that, without loss of generality, one can assume that $S_{d}\subseteq \mathcal{R}_{d,\tilde{\delta}}$.

Furthermore, as a consequence of (\ref{e529}) and (\ref{e571}), and taking into account that $\psi_{k_2}\in\hbox{Exp}^{d}_{(k_2,\beta,\mu,\nu)}$ in view of (\ref{e405}), one has
\begin{align*}
&\left\|\frac{1}{P_{m}(\tau)}\psi_{k_2}(\tau,m,\epsilon)\right\|_{(k_2,\beta,\mu,\nu)}\le \frac{1}{C_{P}(r_{Q,R_{D}})^{1/d_{D}}}\sup_{m\in\R}\frac{1}{|R_{D}(im)|}\left\|\psi_{k_2}(\tau,m,\epsilon)\right\|_{(k_2,\beta,\mu,\nu)}\\
&\le \frac{1}{C_{P}(r_{Q,R_{D}})^{1/d_{D}}}\sup_{m\in\R}\frac{1}{|R_{D}(im)|}C_{\psi_{k_2}}\le \frac{1}{C_{P}(r_{Q,R_{D}})^{1/d_{D}}}\sup_{m\in\R}\frac{1}{|R_{D}(im)|}\zeta_{\psi_{k_2}}.
\end{align*}

The previous estimates allow us to affirm that
\begin{align*}
\left\|\mathcal{H}_{\epsilon}^{k_2}(w(\tau,m))\right\|_{(k_2,\beta,\mu,\nu)}&\le \frac{\tilde{C}_2\tilde{C}_3}{C_P(r_{Q,R_D})^{1/d_{D}}(2\pi)^{1/2}}\varpi\sum_{\ell=1}^{D-1}\sum_{\lambda\in I_{\ell}}\frac{\zeta_{\lambda,\ell}\epsilon_0^{\Delta_{\lambda,\ell}-d_{\lambda,\ell}}}{(q^{1/k_2})^{(d_{\lambda,\ell}+k_2)(d_{\lambda,\ell}+k_2-1)/2}}\\
&+\frac{1}{(q^{1/k_2})^{k_{2}(k_2-1)/2}}\frac{1}{C_{P}(r_{Q,R_{D}})^{1/d_{D}}}\sup_{m\in\R}\frac{1}{|R_{D}(im)|}\zeta_{\psi_{k_2}}.
\end{align*}

We choose $\zeta_{\psi_{k_2}},\zeta_{\lambda,\ell}$ for every $1\le \ell\le D-1$ and $\lambda\in I_{\ell}$ such that
$$\frac{\tilde{C}_2\tilde{C}_3}{C_P(r_{Q,R_D})^{1/d_{D}}(2\pi)^{1/2}}\sum_{\ell=1}^{D-1}\sum_{\lambda\in I_{\ell}}\frac{\zeta_{\lambda,\ell}\epsilon_0^{\Delta_{\lambda,\ell}-d_{\lambda,\ell}}}{(q^{1/k_2})^{(d_{\lambda,\ell}+k_2)(d_{\lambda,\ell}+k_2-1)/2}}\le\frac{1}{2},\hspace{3cm}$$
$$\hspace{3cm}\frac{1}{(q^{1/k_2})^{k_{2}(k_2-1)/2}}\frac{1}{C_{P}(r_{Q,R_{D}})^{1/d_{D}}}\sup_{m\in\R}\frac{1}{|R_{D}(im)|}\zeta_{\psi_{k_2}}\le\frac{\varpi}{2}.$$

This yields $\left\|\mathcal{H}_{\epsilon}^{k_2}(w(\tau,m))\right\|_{(k_2,\beta,\mu,\nu)}\le\varpi$. The operator $\mathcal{H}_{\epsilon}^{k_2}(w(\tau,m))$ maps $\overline{D}(0,\varpi)\subseteq\hbox{Exp}^{d}_{(k_2,\beta,\mu,\nu)}$ into itself. Let $w^{1}_{k_{2}},w^{2}_{k_{2}}\in\hbox{Exp}^{d}_{(k_2,\beta,\mu,\nu)}$, with $\bigl\|w^{j}_{k_2}\bigr\|_{(\kappa,\beta,\mu,\nu)}\le\varpi$ for $j=1,2$.

Following analogous calculations as before we arrive at
\begin{align*}
&\left\|\mathcal{H}_{\epsilon}^{k_2}(w^1_{k_2}(\tau,m))-\mathcal{H}_{\epsilon}^{k_2}(w^2_{k_2}(\tau,m))\right\|_{(k_2,\beta,\mu,\nu)}\\
&\hspace{2cm}\le \frac{\tilde{C}_2\tilde{C}_3}{C_P(r_{Q,R_D})^{1/d_{D}}(2\pi)^{1/2}}\sum_{\ell=1}^{D-1}\sum_{\lambda\in I_{\ell}}\frac{\zeta_{\lambda,\ell}\epsilon_0^{\Delta_{\lambda,\ell}-d_{\lambda,\ell}}\left\|w^1_{k_2}(\tau,m)-w^2_{k_2}(\tau,m)\right\|_{(k_2,\beta,\mu,\nu)}}{(q^{1/k_2})^{(d_{\lambda,\ell}+k_2)(d_{\lambda,\ell}+k_2-1)/2}}.
\end{align*}

We choose $\zeta_{\lambda,\ell}$ for every $1\le \ell\le D-1$ and $\lambda\in I_{\ell}$ such that
$$\frac{\tilde{C}_2\tilde{C}_3}{C_P(r_{Q,R_D})^{1/d_{D}}(2\pi)^{1/2}}\sum_{\ell=1}^{D-1}\sum_{\lambda\in I_{\ell}}\frac{\zeta_{\lambda,\ell}\epsilon_0^{\Delta_{\lambda,\ell}-d_{\lambda,\ell}}}{(q^{1/k_2})^{(d_{\lambda,\ell}+k_2)(d_{\lambda,\ell}+k_2-1)/2}}\le\frac{1}{2},$$
and conclude

$$\left\|\mathcal{H}_{\epsilon}^{k_2}(w^1_{k_2}(\tau,m))-\mathcal{H}_{\epsilon}^{k_2}(w^2_{k_2}(\tau,m))\right\|_{(k_2,\beta,\mu,\nu)}\le\frac{1}{2}\left\|w^1_{k_2}(\tau,m)-w^2_{k_2}(\tau,m)\right\|_{(k_2,\beta,\mu,\nu)}.$$

The closed disc $\overline{D}(0,\varpi)\subseteq\hbox{Exp}^{d}_{(k_2,\beta,\mu,\nu)}$ is a complete metric space for the norm $\left\|\cdot\right\|_{(k_2,\beta,\mu,\nu)}$. The previous reasonings derive into the conclusion that the operator $\mathcal{H}^{k_2}_{\epsilon}$ is a contractive map from $\overline{D}(0,\varpi)$ into itself. The classical contractive mapping theorem states the existence of a unique fixed point, say $w^{d}_{k_2}(\tau,m,\epsilon)$. Moreover,$w^{d}_{k_2}(\tau,m,\epsilon)$ is a solution of (\ref{e492}) and also  holomorphy of $w^{d}_{k_2}(\tau,m,\epsilon)$ with respect to $\epsilon$ is attained by construction.
\end{proof}

The next result provides the link between the acceleration of $w_{k_{1}}^{d}$, obtained in Proposition~\ref{prop321}, and $w_{k_2}^{d}$ determined in Proposition~\ref{prop571}.Indeed, we prove they coincide as elements in an appropriate Banach space.

\begin{prop}\label{prop638}
Let us consider the function $w_{k_{1}}^{d}(\tau,m,\epsilon)$ constructed in Proposition~\ref{prop321}, solution of (\ref{e224}). For every $\tilde{\delta}>0$, the function
$$\tau\mapsto\mathcal{L}^{d}_{q;1/\kappa}(w_{k_{1}}^{d}(\tau,m,\epsilon)):=\mathcal{L}^{d}_{q;1/\kappa}\left(h\mapsto w_{k_{1}}^{d}(h,m,\epsilon)\right)(\tau)= \frac{1}{\pi_{q^{1/\kappa}}}\int_{L_{d}}\frac{w_{k_1}^d(h,m,\epsilon)}{\Theta_{q^{1/\kappa}}\left(\frac{u}{\tau}\right)}\frac{du}{u}$$
defines a bounded holomorphic function in $\mathcal{R}_{d,\tilde{\delta}}\cap D(0,r_{1})$, with $0<r_1\le q^{(\frac{1}{2}-\alpha)/\kappa}/2$ (we recall $\alpha\in\R$ is fixed at the beginning of Section~\ref{seccion4}). Moreover, for every $\epsilon\in D(0,\epsilon_0)$, the identity
\begin{equation}\label{e638}
\mathcal{L}^{d}_{q;1/\kappa}(w_{k_{1}}^{d})(\tau,m,\epsilon)=w_{k_2}^{d}(\tau,m,\epsilon)
\end{equation}
holds for $\tau\in S^{b}_{d}$, $m\in\R$ and $\epsilon\in D(0,\epsilon_0)$, where $\tilde{\rho}>0$ and $S^{b}_{d}$ is a finite sector of bisecting direction $d$.
\end{prop}
\begin{proof}
We recall $w_{k_{1}}^{d}\in\hbox{Exp}^{d}_{(\kappa,\beta,\mu,\alpha,\rho)}$, which implies there exists $C_{w_{k_1}}^{1}>0$ such that 
$$\bigl\|w^{d}_{k_{1}}(\tau,m,\epsilon)\bigr\|_{(\beta,\mu)}\le C_{w_{k_1}}^{1}\exp\left(\frac{\kappa\log^2|\tau+\delta|}{2\log(q)}+\alpha\log|\tau+\delta|\right),$$
for every $\tau\in\overline{D}(0,\rho)\cup U_{d}$, $\epsilon\in D(0,\epsilon_0)$. Direct calculations entail the existence  of $C_{w_{k_1}}^{2},C_{w_{k_1}}^{3}>0$ such that
$$\bigl\|w^{d}_{k_{1}}(\tau,m,\epsilon)\bigr\|_{(\beta,\mu)}\le C_{w_{k_1}}^{2},$$
for every $\tau\in\overline{D}(0,\rho)$, and
$$\bigl\|w^{d}_{k_{1}}(\tau,m,\epsilon)\bigr\|_{(\beta,\mu)}\le C_{w_{k_1}}^{3}\exp\left(\frac{\kappa\log^2|\tau|}{2\log(q)}+\alpha\log|\tau|\right),$$
for every $\tau\in U_{d}$ with $|\tau|\ge\rho$. Definition~\ref{def203} and Lemma~\ref{lema219} guarantee that for every $\tilde{\delta}>0$ the function $\mathcal{L}^{d}_{q;1/\kappa}(w_{k_{1}}^{d}(\tau,m,\epsilon))$ defines a bounded holomorphic function in $\mathcal{R}_{d,\tilde{\delta}}\cap D(0,r_1)$ for $0< r_1\le q^{(\frac{1}{2}-\alpha)/\kappa}/2$, and values in the space $E_{(\beta,\mu)}$.

We now give proof of the identity (\ref{e638}). For this purpose, we consider equation (\ref{e224}) satisfied by $w^{d}_{k_1}(\tau,m,\epsilon)$ and divide both sides by $\tau^{k_1}$. One has

\begin{align}&Q(im)\frac{1}{(q^{1/k_1})^{k_1(k_1-1)/2}}w^d_{k_1}(\tau,m,\epsilon)=\frac{\tau^{d_D}}{(q^{1/k_1})^{(d_{D}+k_1)(d_{D}+k_1-1)/2}}\sigma_{q}^{-d_{D}/\kappa}R_{D}(im)w^d_{k_1}(\tau,m,\epsilon)\nonumber\\
+&\sum_{\ell=1}^{D-1}\left(\sum_{\lambda\in I_{\ell}}\frac{\epsilon^{\Delta_{\lambda,\ell}-d_{\lambda,\ell}}\tau^{d_{\lambda,\ell}}}{(q^{1/k_1})^{(d_{\lambda,\ell}+k_1)(d_{\lambda,\ell}+k_1-1)/2}}\sigma_q^{\delta_{\ell}-\frac{d_{\lambda,\ell}}{k_1}-1}\frac{1}{(2\pi)^{1/2}}(C_{\lambda,\ell}(m,\epsilon)\ast^{R_{\ell}}w^d_{k_1}(\tau,m,\epsilon))\right)\nonumber\\
&\hspace{10cm}+\frac{1}{(q^{1/k_1})^{k_{1}(k_1-1)/2}}\psi_{k_1}(\tau,m,\epsilon).\label{e660}
\end{align}

We now take $q-$Laplace transform of order $\kappa$ and direction $d$ at both sides of (\ref{e660}). In view of Proposition~\ref{prop259}, one has
\begin{equation}\label{e664}
\mathcal{L}_{q;1/\kappa}^{d}\left(\tau^{d_D}\sigma_{q}^{-d_{D}/\kappa}w^d_{k_1}(\tau,m,\epsilon)\right)=(q^{1/\kappa})^{d_{D}(d_{D}-1)/2}\tau^{d_{D}}\mathcal{L}^{d}_{q;1/\kappa}(w_{k_{1}}^{d}(\tau,m,\epsilon)).
\end{equation}
For every $1\le \ell\le D-1$ and $\lambda\in I_{\ell}$, and also
\begin{align}
&\mathcal{L}_{q;1/\kappa}^{d}\left(\tau^{d_{\lambda,\ell}}\sigma_{q}^{\delta_{\ell}-\frac{d_{\lambda,\ell}}{k_1}-1}(C_{\lambda,\ell}(m,\epsilon)\ast^{R_{\ell}}w^d_{k_1}(\tau,m,\epsilon))\right)\hspace{2cm}\nonumber \\
&\hspace{2cm}=(q^{1/\kappa})^{d_{\lambda,\ell}(d_{\lambda,\ell}-1)/2}\tau^{d_{\lambda,\ell}}\sigma_{q}^{-\frac{d_{\lambda,\ell}}{k_2}+\delta_{\ell}-1}\mathcal{L}^{d}_{q;1/\kappa}(C_{\lambda,\ell}(m,\epsilon)\ast^{R_{\ell}}w^d_{k_1}(\tau,m,\epsilon)).\label{e668}
\end{align}

We now proceed to justify the change in the order of integration in the expression 
$$\mathcal{L}^{d}_{q;1/\kappa}(C_{\lambda,\ell}(m,\epsilon)\ast^{R_{\ell}}w^d_{k_1}(\tau,m,\epsilon)).$$
 For every $m_1\in\R$ and $r\ge0$ we define the function
$$(m_1,r)\mapsto \Xi(m_1,r):=\frac{C_{\lambda,\ell}(m-m_1)R_{\ell}(im_{1})w^{d}_{k_{1}}(re^{id},m_1,\epsilon)}{\Theta_{q^{1/\kappa}}\left(\frac{re^{id}}{\tau}\right)r},$$
for $\tau\in\mathcal{R}_{d,\tilde{\delta}}\cap D(0,r_1)$, $m\in\R$ and $\epsilon\in D(0,\epsilon_0)$.

Taking into account (\ref{e223}), Proposition~\ref{prop321} and (\ref{e200}), one has
$$|\Xi(m_1,r)|\le \frac{\tilde{C}_{\lambda,\ell}e^{-|m-m_1|\beta}|R_{\ell}(im_1)|e^{-\beta|m_1|}}{C_{q,\kappa}(1+|m-m_1|)^{\mu}(1+|m_1|)^{\mu}}\frac{\varpi|\tau|^{1/2}}{r^{3/2}\tilde{\delta}}\frac{\exp\left(\frac{\kappa\log^2|re^{id}+\delta|}{2\log(q)}+\alpha\log|re^{id}+\delta|\right)}{\exp\left(\frac{\kappa\log^2\left(\frac{r}{|\tau|}\right)}{2\log(q)}\right)},$$
for $m_1\in\R$, $r\ge0$.

Integrability follows from identically reasoning as that were applied at (\ref{e159}) regarding integrability with respect to variable $m_1$, and the next estimates on the expression
\begin{equation}\label{e683}
\frac{\exp\left(\frac{\kappa\log^2|re^{id}+\delta|}{2\log(q)}+\alpha\log|re^{id}+\delta|\right)}{r^{3/2}\exp\left(\frac{\kappa\log^2\left(\frac{r}{|\tau|}\right)}{2\log(q)}\right)}.
\end{equation}

On one hand, if $r>\rho$, then (\ref{e683}) is upper bounded by
$$r^{-3/2}(r+\delta)^{\alpha}\exp\left(\frac{\kappa\log^2(r+\delta)}{2\log(q)}-\frac{\kappa\log^2\left(\frac{r}{|\tau|}\right)}{2\log(q)}\right)$$
$$= e^{-\frac{\kappa\log^2|\tau|}{2\log(q)}}r^{-3/2}(r+\delta)^{\alpha}\exp\left(\frac{\kappa(\log^2(r+\delta)-\log^2(r))}{2\log(q)}+\frac{\kappa\log(r)\log|\tau|}{\log(q)}\right).$$
There exists $C_{51}>0$ such that $\log^2(r+\delta)-\log^2(r)\le C_{51}$ for every $r>\rho$, and under the assumption that $|\tau|<r_1$, the term $r^{-3/2}(r+\delta)^{\alpha}\exp\left(\frac{\kappa\log(r)\log|\tau|}{\log(q)}\right)$ is upper bounded by $r^{-2}$. This yields integrability of (\ref{e683}) with resèct to $r$.

On the other hand, if $0\le r\le \rho$, (\ref{e683}) is upper estimated by $C_{52}r^{-3/2}\exp\left(-\frac{\kappa\log^2\left(\frac{r}{|\tau|}\right)}{2\log(q)}\right),$
for some $C_{52}>0$. This function is rewritten in the form
$$C_{52}r^{-3/2}\exp\left(-\frac{\kappa\log^2(r)}{2\log(q)}\right)\exp\left(-\frac{\kappa\log^2|\tau|}{2\log(q)}\right)\exp\left(\frac{\kappa\log(\rho)\log|\tau|}{\log(q)} \right),$$
which is integrable with respect to $r$, for $r\in[0,\rho]$.

The function $\Xi(m_1,r)$ is integrable in $\R\times[0,\infty)$ for $\tau\in\mathcal{R}_{d,\tilde{\delta}}\cap D(0,r_1)$, $m\in\R$ and $\epsilon\in D(0,\epsilon_0)$. One can apply Fubini's Theorem to conclude that

\begin{equation}\label{e701}
\mathcal{L}^{d}_{q;1/\kappa}(C_{\lambda,\ell}(m,\epsilon)\ast^{R_{\ell}}w^d_{k_1}(\tau,m,\epsilon))=C_{\lambda,\ell}(m,\epsilon)\ast^{R_{\ell}}\mathcal{L}^{d}_{q;1/\kappa}(w^d_{k_1})(\tau,m,\epsilon),
\end{equation}
$\tau\in\mathcal{R}_{d,\tilde{\delta}}\cap D(0,r_1)$, $m\in\R$ and $\epsilon\in D(0,\epsilon_0)$.

By means of (\ref{e664}), (\ref{e668}), (\ref{e701}) and (\ref{e399}) one obtains that  (\ref{e660}) is transformed from the application of $q-$Laplace transform of order $\kappa$ along direction $d$ into
\begin{align}&Q(im)\frac{\mathcal{L}^{d}_{q;1/\kappa}(w^d_{k_1})(\tau,m,\epsilon)}{(q^{1/k_1})^{k_1(k_1-1)/2}}=\frac{(q^{1/\kappa})^{d_{D}(d_{D}-1)/2}}{(q^{1/k_1})^{(d_{D}+k_1)(d_{D}+k_1-1)/2}}\tau^{d_{D}}R_{D}(im)\mathcal{L}^{d}_{q;1/\kappa}(w^d_{k_1})(\tau,m,\epsilon)\nonumber\\
+&\sum_{\ell=1}^{D-1}\left(\sum_{\lambda\in I_{\ell}}\frac{\epsilon^{\Delta_{\lambda,\ell}-d_{\lambda,\ell}}(q^{1/\kappa})^{d_{\lambda,\ell}(d_{\lambda,\ell}-1)/2}}{(q^{1/k_1})^{(d_{\lambda,\ell}+k_1)(d_{\lambda,\ell}+k_1-1)/2}}\frac{\tau^{d_{\lambda,\ell}}\sigma_{q}^{-\frac{d_{\lambda,\ell}}{k_2}+\delta_{\ell}-1}}{(2\pi)^{1/2}}(C_{\lambda,\ell}(m,\epsilon)\ast^{R_{\ell}}\mathcal{L}^{d}_{q;1/\kappa}(w^d_{k_1})(\tau,m,\epsilon))\right)\nonumber\\
&\hspace{10cm}+\frac{1}{(q^{1/k_1})^{k_{1}(k_1-1)/2}}\psi_{k_2}(\tau,m,\epsilon).\label{e710}
\end{align}

We multiply both sides of the equation by $\tau^{k_2}(q^{1/k_1})^{k_1(k_1-1)/2}(q^{1/k_2})^{-k_2(k_2-1)/2}$ to obtain

\begin{align}&Q(im)\frac{\tau^{k_2}\mathcal{L}^{d}_{q;1/\kappa}(w^d_{k_1})(\tau,m,\epsilon)}{(q^{1/k_2})^{k_2(k_2-1)/2}}=A(k_1,k_2,\kappa,d_{D})\tau^{d_{D}+k_2}R_{D}(im)\mathcal{L}^{d}_{q;1/\kappa}(w^d_{k_1})(\tau,m,\epsilon)\nonumber\\
+&\sum_{\ell=1}^{D-1}\left(\sum_{\lambda\in I_{\ell}}\epsilon^{\Delta_{\lambda,\ell}-d_{\lambda,\ell}}B(k_1,k_2,\kappa,d_{\lambda,\ell})\frac{\tau^{d_{\lambda,\ell}}\sigma_{q}^{-\frac{d_{\lambda,\ell}}{k_2}+\delta_{\ell}-1}}{(2\pi)^{1/2}}(C_{\lambda,\ell}(m,\epsilon)\ast^{R_{\ell}}\mathcal{L}^{d}_{q;1/\kappa}(w^d_{k_1})(\tau,m,\epsilon))\right)\nonumber\\
&\hspace{10cm}+\frac{\tau^{k_2}}{(q^{1/k_2})^{k_{2}(k_2-1)/2}}\psi_{k_2}(\tau,m,\epsilon),\label{e717}
\end{align}
where
$$A(k_1,k_2,\kappa,d_{D})=\frac{(q^{1/k_1})^{k_1(k_1-1)/2}(q^{1/\kappa})^{d_{D}(d_{D}-1)/2}}{(q^{1/k_2})^{k_2(k_2-1)/2}(q^{1/k_1})^{(d_{D}+k_1)(d_{D}+k_1-1)/2}},$$
and 
$$B(k_1,k_2,\kappa,d_{\lambda,\ell})\frac{(q^{1/k_1})^{k_1(k_1-1)/2}(q^{1/\kappa})^{d_{\lambda,\ell}(d_{\lambda,\ell}-1)/2}}{(q^{1/k_2})^{k_2(k_2-1)/2}(q^{1/k_1})^{(d_{\lambda,\ell}+k_1)(d_{\lambda,\ell}+k_1-1)/2}},$$
for every $1\le \ell\le D-1$and $\lambda\in I_{\ell}$.

Usual estimates allow us to prove that 
$$A(k_1,k_2,\kappa,d_{D})=(q^{1/k_2})^{-(d_{D}+k_2)(d_{D}+k_2-1)/2},\quad B(k_1,k_2,\kappa,d_{\lambda,\ell})=(q^{1/k_2})^{-(d_{\lambda,\ell}+k_2)(d_{\lambda,\ell}+k_2-1)/2}.$$
Taking these values into (\ref{e717}) we observe that
\begin{align}&Q(im)\frac{\tau^{k_2}\mathcal{L}^{d}_{q;1/\kappa}(w^d_{k_1})(\tau,m,\epsilon)}{(q^{1/k_2})^{k_2(k_2-1)/2}}=R_{D}(im)\frac{\tau^{d_D+k_2}\mathcal{L}^{d}_{q;1/\kappa}(w^d_{k_1})(\tau,m,\epsilon)}{(q^{1/k_2})^{(d_{D}+k_2)(d_{D}+k_2-1)/2}}\nonumber\\
+&\sum_{\ell=1}^{D-1}\left(\sum_{\lambda\in I_{\ell}}\frac{\epsilon^{\Delta_{\lambda,\ell}-d_{\lambda,\ell}}\tau^{d_{\lambda,\ell}+k_2}\sigma_q^{\delta_{\ell}-\frac{d_{\lambda,\ell}}{k_2}-1}}{(q^{1/k_2})^{(d_{\lambda,\ell}+k_2)(d_{\lambda,\ell}+k_2-1)/2}}\frac{1}{(2\pi)^{1/2}}(C_{\lambda,\ell}(m,\epsilon)\ast^{R_{\ell}}\mathcal{L}^{d}_{q;1/\kappa}(w^d_{k_1})(\tau,m,\epsilon))\right)\nonumber\\
&\hspace{10cm}+\frac{\tau^{k_2}}{(q^{1/k_2})^{k_{2}(k_2-1)/2}}\psi_{k_2}(\tau,m,\epsilon).\label{e732}
\end{align}
recovering equation (\ref{e492}). This yields $\mathcal{L}^{d}_{q;1/\kappa}(w^d_{k_1})(\tau,m,\epsilon)$ is a solution of (\ref{e492}), for $(\tau,m,\epsilon)\in(\mathcal{R}_{d,\tilde{\delta}}\cap D(0,r_1))\times \R\times D(0,\epsilon_0)$. 

Let $S_{d}^{b}$ be a bounded sector of bisecting direction $d$ such that $S_{d}^{b}\subseteq (\mathcal{R}_{d,\tilde{\delta}}\cap D(0,r_1))\cap S_{d}$. We recall this is always possible from the definition of $\mathcal{R}_{d,\tilde{\delta}}$. One firstly observes that both, $\mathcal{L}^{d}_{q;1/\kappa}(w^d_{k_1})(\tau,m,\epsilon)$ and $w_{k_2}^d(\tau,m,\epsilon)$ are continuous complex functions defined on $S_{d}^{b}\times\R\times D(0,\epsilon_0)$ and holomorphic with respect to $\tau$ (resp. $\epsilon$) on $S_{d}^{b}$ (resp. $D(0,\epsilon_0)$). This assertion can be checked when regarding $w_{k_2}^d(\tau,m,\epsilon)$ in Proposition~\ref{prop571}, and for $\mathcal{L}^{d}_{q;1/\kappa}(w^d_{k_1})(\tau,m,\epsilon)$ from the properties which are endowed by this function from $w^d_{k_1}(\tau,m,\epsilon)$ which were pointed out in Proposition~\ref{prop321}. 

Let $\epsilon\in D(0,\epsilon_0)$ and put $\Omega=\min\{\alpha,\nu\}$, where $\nu$ is stated in Lemma~\ref{lema403}. We define the auxiliary Banach space $H_{(\beta,\mu,k_2,\Omega)}$ consisting of all continuous complex functions $(\tau,m)\mapsto h(\tau,m)$, defined on $S_{d}^{b}\times\R$, holomorphic with respect to $\tau$ variable in $S_{d}^{b}$, such that
$$\left\|h(\tau,m)\right\|_{H_{(\beta,\mu,k_2,\Omega)}}:=\sup_{\tau\in S_{d}^{b},m\in \R}(1+|m|)^{\mu}e^{\beta|m|}\exp\left(-\frac{k_2\log^2|\tau|}{2\log(q)}-\Omega\log|\tau|\right)|h(\tau,m)|<\infty.$$ 
We observe that, for every $\epsilon\in D(0,\epsilon_0)$, $\mathcal{L}^{d}_{q;1/\kappa}(w^d_{k_1})(\tau,m,\epsilon)$ belongs to $H_{(\beta,\mu,k_2,\Omega)}$ because 
$$\sup_{\tau\in\mathcal{R}_{d,\tilde{\delta}}\cap D(0,r_1)}\left\|\mathcal{L}^{d}_{q;1/\kappa}(w^d_{k_1})(\tau,m,\epsilon)\right\|_{(\beta,\mu)}<\infty.$$
This implies the existence of positive constants $B_{\mathcal{L}(w_1)},C_{\mathcal{L}(w_1)}$ such that
\begin{align*}
\left|\mathcal{L}^{d}_{q;1/\kappa}(w^d_{k_1})(\tau,m,\epsilon)\right|&\le B_{\mathcal{L}(w_1)}(1+|m|)^{-\mu}e^{-\beta|m|}\\
 & \le C_{\mathcal{L}(w_1)}(1+|m|)^{-\mu}e^{-\beta|m|}\exp\left(\frac{k_2\log^2|\tau|}{2\log(q)}+\Omega\log|\tau|\right),
\end{align*}
for all$\tau\in S_{b}^{d}$,$m\in\R$. Also, $w^{d}_{k_2}(\tau,m,\epsilon)$ belongs to $H_{(\beta,\mu,k_2,\Omega)}$ in view of (\ref{prop571}), and also $\psi_{k_2}(\tau,m,\epsilon)$ belongs to $H_{(\beta,\mu,k_2,\Omega)}$ taking into account (\ref{e405}). 

We conclude the proof of (\ref{e638}) by demonstrating that the operator $\mathcal{H}_{k_2}^{\epsilon}$ defined in (\ref{e578}) admits a unique fixed point when restricted to the elements in certain closed disc in $H_{(\beta,\mu,k_2,\Omega)}$, whilst $\mathcal{L}^{d}_{q;1/\kappa}(w^d_{k_1})(\tau,m,\epsilon)$ and $w^{d}_{k_2}(\tau,m,\epsilon)$ are both fixed points belonging to that closed disc in $H_{(\beta,\mu,k_2,\Omega)}$.

For that purpose, one can state analogous results as Lemma~\ref{lema182}, Proposition~\ref{prop188} and Proposition~\ref{prop197} when considering the Banach space $H_{(\beta,\mu,k_2,\Omega)}$. This result follows exact arguments as there, so we omit the details. One can reproduce the same steps as in the proof of Proposition~\ref{prop571} to conclude that, if there exist small enough positive constants $\zeta_{\psi_{k_2}}$ and $\zeta_{\lambda,\ell}$ for $1\le \ell\le D-1$ and $\lambda\in I_{\ell}$ such that (\ref{e571}) holds, then equation (\ref{e492}) admits a unique solution in the space $H_{(\beta,\mu,k_2,\Omega)}$. Bearing in mind that $\mathcal{L}^{d}_{q;1/\kappa}(w^d_{k_1})(\tau,m,\epsilon)$ and $w^{d}_{k_2}(\tau,m,\epsilon)$ both solve (\ref{e492}) and belong to $\overline{B}(0,\varpi)\subseteq H_{(\beta,\mu,k_2,\Omega)}$, they both coincide in the domain $S_{d}^{b}\times \R\times D(0,\epsilon_0)$, and the result follows.
\end{proof}

\begin{corol}\label{coro777}
Under the hypotheses made on Proposition~\ref{prop638}, for every $m\in\R$ and $\epsilon\in D(0,\epsilon_0)$, the function $\tau\mapsto\mathcal{L}^{d}_{q;1/\kappa}(w_{k_{1}}^{d}(\tau,m,\epsilon))$, holomorphic and bounded on the domain $\mathcal{R}_{d,\tilde{\delta}}\cap D(0,r_1)$, can be analytically prolonged to an infinite sector of bisecting direction $d$, $S_{d}$, and there exists $C_{w_{k_2}}>0$ such that
\begin{equation}\label{e755}
\left|\mathcal{L}^{d}_{q;1/\kappa}(w_{k_{1}}^{d}(\tau,m,\epsilon))\right|\le C_{w_{k_2}}(1+|m|)^{-\mu}e^{-\beta|m|}\exp\left(\frac{k_2\log^2|\tau|}{2\log(q)}+\nu\log|\tau|\right),
\end{equation}
for every $\tau\in S_{d}$. Here, $\nu$ is obtained in Lemma~\ref{lema403}. In addition to that, the extension is continuous for $(m,\epsilon)\in\R\times D(0,\epsilon_0)$ and holomorphic with respect to $\epsilon\in D(0,\epsilon_0)$. 
\end{corol}

\section{Analytic solutions of a linear initial value Cauchy problem}\label{seccion5}
Let $k_1,k_2,D$ be positive integers such that $k_1<k_2$ and $D\ge 3$. Let $\kappa>0$ be defined by (\ref{e236}). Let $q\in\R$ with $q>1$ and assume that for every $1\le \ell\le D-1$, $I_{\ell}$ is a finite nonempty subset of nonnegative integers.

Let $d_{D}\ge1$ be an integer. For every $1\le \ell\le D-1$, we consider an integer $\delta_{\ell}\ge1$. In  addition, for each $\lambda\in I_{\ell}$, we choose integers $d_{\lambda,\ell}\ge 1$ , $\Delta_{\lambda,\ell}\ge0$.
We make the assumption that
\begin{equation}\label{e796}
\quad\delta_1=1,\quad \delta_{\ell}<\delta_{\ell+1}\quad ,
\end{equation}
for every $1\le\ell\le D-1$.

We assume that 
\begin{equation}\label{e800}
\Delta_{\lambda,\ell}\ge d_{\lambda,\ell},\quad\frac{d_{\lambda,\ell}}{k_2}+1\ge \delta_\ell,\qquad \frac{d_{D}-1}{k_2}+1\ge \delta_{\ell}
\end{equation}
for every $1\le \ell\le D-1$ and all $\lambda\in I_{\ell}$. Let $Q,R_{\ell}\in\C[X]$ with
\begin{equation}\label{e804}
\deg(Q)\ge\deg(R_{D})\ge \deg(R_{\ell}), \quad Q(im)\neq0,\quad R_{D}(im)\neq 0,
\end{equation}
for all $1\le \ell\le D-1$ and $m\in\R$.

We require the existence of an unbounded sector 
$$S_{Q,R_{D}}=\left\{z\in\C:|z|\ge r_{Q,R_{D}},|\arg(z)-d_{Q,R_{D}}|\le \eta_{Q,R_{D}}\right\},$$
for some $r_{Q,R_{D}},\eta_{Q,R_{D}}>0$, such that 
\begin{equation}\label{e814}
\frac{Q(im)}{R_{D}(im)}\in S_{Q,R_{D}},\quad m\in\R.
\end{equation}

\begin{defin}\label{def817}
Let $\varsigma\ge 2$ be an integer. Let $\mathcal{E}_{p}$ be an open sector with vertex at the origin and radius $\epsilon_0$ for every $0\le p\le \varsigma-1$ and such that $\mathcal{E}_{j}\cap\mathcal{E}_{k}\neq0$ for every $0\le j,k\le\varsigma-1$ if and only if $|j-k|\le1$ (under the notation $\mathcal{E}_{\varsigma}:=\mathcal{E}_{0}$) and such that $\cup_{p=0}^{\varsigma-1}\mathcal{E}_{p}=\mathcal{U}\setminus\{0\}$, for some neighborhood of the origin, $\mathcal{U}$. A family $(\mathcal{E}_p)_{0\le p\le \varsigma -1}$ satisfying these properties is known as a good covering in $\C^{\star}$. 
\end{defin}

\begin{defin}\label{def818}
Let $(\mathcal{E}_{p})_{0\le p\le \varsigma-1}$ be a good covering in $\C^{\star}$. Let $\mathcal{T}$ be an open bounded sector with vertex at 0 and radius $r_{\mathcal{T}}>0$. We make the assumption that

\begin{equation}\label{e824}
0<\epsilon_0,r_{\mathcal{T}}<1,\quad \nu+\frac{k_2}{\log(q)}\log(r_{\mathcal{T}})<0,\quad \alpha+\frac{\kappa}{\log(q)}\log(\epsilon_0 r_{\mathcal{T}})<0,\quad \epsilon_0r_{\mathcal{T}}\le q^{\left(\frac{1}{2}-\nu\right)/k_2}/2
\end{equation}
for $\nu$ constructed in Lemma~\ref{lema403}.

We consider a family of unbounded sectors $U_{\mathfrak{d}{p}}$ with bisecting direction $\mathfrak{d}_{p}\in\R$ and a family of open domains $\mathcal{R}_{\mathfrak{d}_{p}}^{b}:=\mathcal{R}_{\mathfrak{d}_{p},\tilde{\delta}}\cap D(0,\epsilon_0r_{\mathcal{T}})$, where
$$\mathcal{R}_{\mathfrak{d}_{p},\tilde{\delta}}=\left\{T\in\C^{\star}:\left|1+\frac{e^{i\mathfrak{d}_{p}}}{T}r\right|>\tilde{\delta},\quad\hbox{for every }r\ge0\right\}.$$

We assume $\mathfrak{d}_{p}$, $0\le p\le \varsigma-1$, are chosen so that some conditions are satisfied. On order to enumerate them, we denote $q_{\ell}(m)$ the roots of the polynomial $P_m(\tau)$, defined in (\ref{e558}). We take an unbounded sector with vertex at 0 and bisecting direction $\mathfrak{d}_{p}$, $S_{\mathfrak{d}_p}$, $0\le p\le \varsigma-1$; and we choose $\rho>0$ such that:
\begin{itemize}
\item[1)] There exists $M_1>0$ such that (\ref{e513}) holds for all $m\in\R$, $\tau\in S_{\mathfrak{d}_{p}}\cup \overline{D}(0,\rho)$, all $0\le p\le \varsigma-1$ and all $0\le l\le d_{D}-1$.
\item[2)] There exists $M_2>0$ and $l_0\in\{0,...,d_{D}-1\}$ such that (\ref{e520}) holds for every $m\in\R$, $\tau\in S_{\mathfrak{d}_{p}}\cup\overline{D}(0,\rho)$, and all $0\le p\le \varsigma-1$.
\item[3)] For every $0\le p\le \varsigma-1$ we have $\mathcal{R}_{\mathfrak{d}_{p}}^{b}\cap \mathcal{R}_{\mathfrak{d}_{p+1}}^{b}\neq\emptyset$, and for all $t\in\mathcal{T}$ and $\epsilon\in\mathcal{E}_{p}$, we have that $\epsilon t\in\mathcal{R}_{\mathfrak{d}_{p}}^{b}$. Here we have put $\mathcal{R}_{\mathfrak{d}_{\varsigma}}^{b}:=\mathcal{R}_{\mathfrak{d}_{0}}^{b}$. 
\end{itemize}
The family $\{(\mathcal{R}_{\mathfrak{d}_{p},\tilde{\delta}})_{0\le p\le \varsigma-1},D(0,\rho),\mathcal{T}\}$ is said to be associated to the good covering $(\mathcal{E}_{p})_{0\le p\le \varsigma-1}$.
\end{defin}

We consider a good covering $(\mathcal{E}_{p})_{0\le p\le \varsigma-1}$ and a family $\{(\mathcal{R}_{\mathfrak{d}_{p},\tilde{\delta}})_{0\le p\le \varsigma-1},D(0,\rho),\mathcal{T}\}$ associated to it. For every $0\le p\le \varsigma-1$ we study the next initial value Cauchy problem
\begin{align}
&Q(\partial_z)\sigma_qu^{\mathfrak{d}_{p}}(t,z,\epsilon)=(\epsilon t)^{d_{D}}\sigma_{q}^{\frac{d_{D}}{k_2}+1}R_{D}(\partial_{z})u^{\mathfrak{d}_{p}}(t,z,\epsilon)\nonumber\\
&\hspace{3cm}+\sum_{\ell=1}^{D-1}\left(\sum_{\lambda\in I_{\ell}}t^{d_{\lambda,\ell}}\epsilon^{\Delta_{\lambda,\ell}}\sigma_{q}^{\delta_\ell}c_{\lambda,\ell}(z,\epsilon)R_{\ell}(\partial_z)u^{\mathfrak{d}_{p}}(t,z,\epsilon)\right)+\sigma_{q}f^{\mathfrak{d}_{p}}(t,z,\epsilon).\label{e771}
\end{align}

The operator $\sigma_q$ acts on variable $t$. The coefficients $c_{\lambda,\ell}(z,\epsilon)$ with $1\le \ell\le D-1$ and $\lambda\in I_{\ell}$, and the forcing term $f^{\mathfrak{d}_{p}}(t,z,\epsilon)$ are constructed as follows. For every $1\le \ell\le D-1$ and $\lambda\in I_{\ell}$ and every integer $n\ge0$, we consider the functions $m\mapsto C_{\lambda,\ell}(m,\epsilon)$ and $m\mapsto F_{n}(m,\epsilon)$ belonging to the space $E_{(\beta,\mu)}$, for some $\beta>0$ and $\mu>\deg(R_{D})+1$. We assume all these functions depend holomorphically on $\epsilon\in D(0,\epsilon_0)$. Moreover, we assume there exist $\tilde{C}_{\lambda,\ell},C_{F}>0$ such that (\ref{e223}) and (\ref{e270}) hold for all $1\le \ell\le D-1$, $\lambda\in I_{\ell}$, $n\ge0$ and $\epsilon\in D(0,\epsilon_0)$.Then, we put
\begin{equation}\label{e855b}
c_{\lambda,\ell}(z,\epsilon)=\mathcal{F}^{-1}(m\mapsto C_{\lambda,\ell}(m,\epsilon))(z),
\end{equation}
which, for every $1\le \ell\le D-1$, $\lambda\in I_{\ell}$, defines a bounded holomorphic function on $H_{\beta'}\times D(0,\epsilon_0)$ for any $0<\beta'<\beta$. We assume the formal power series
$$\psi_{k_1}(\tau,m,\epsilon)=\sum_{n\ge0}F_{n}(m,\epsilon)\frac{\tau^n}{(q^{1/k_1})^{\frac{n(n-1)}{2}}},$$
which is convergent on the disc $D(0,\rho)$, can be analytically continued with respect to $\tau$ as a function $\tau\mapsto\psi_{k_{1}}^{\mathfrak{d}_{p}}(\tau,m,\epsilon)$ on an infinite sector $U_{\mathfrak{d}_{p}}$ of bisecting direction $\mathfrak{d}_{p}$, and $\psi_{k_{1}}^{\mathfrak{d}_{p}}(\tau,m,\epsilon)\in \hbox{Exp}^{\mathfrak{d}_{p}}_{(k_1,\beta,\mu,\alpha,\rho)}$ for some $\alpha>0$, and such that there exists $\zeta_{\psi_{k_1}}>0$ with
\begin{equation}\label{e859}
\left\|\psi_{k_{1}}^{\mathfrak{d}_{p}}(\tau,m,\epsilon)\right\|_{(k_1,\beta,\mu,\alpha,\rho)} \le\zeta_{\psi_{k_1}},
\end{equation}
which does not depend on $\epsilon\in D(0,\epsilon_0)$. Lemma~\ref{lema403} guarantees the function
$$\psi_{k_2}^{\mathfrak{d}_{p}}(\tau,m,\epsilon):=\mathcal{L}^{\mathfrak{d}_{p}}_{q;1/\kappa}(h\mapsto \psi_{k_1}^{\mathfrak{d}_{p}}(h,m,\epsilon))(\tau)$$
is an element of the space $\hbox{Exp}^{\mathfrak{d}_{p}}_{(k_2,\beta,\mu,\nu)}$, for some $\nu\in\R$.

Moreover, we get a constant $\zeta_{\psi_{k_2}}>0$ with
\begin{equation}\label{e868}
\left\|\psi_{k_2}^{\mathfrak{d}_{p}}(\tau,m,\epsilon)\right\|_{(k_2,\beta,\mu,\nu)}\le \zeta_{\psi_{k_2}},
\end{equation}
for every $\epsilon\in D(0,\epsilon_0)$. Without loss of generality, one can reduce the opening of sector $S_{\mathfrak{d}_{p}}$ so that it might be considered the corresponding one involved in the definition of the space $\hbox{Exp}^{\mathfrak{d}_{p}}_{(k_2,\beta,\mu,\nu)}$. In view of the proof of Lemma~\ref{lema403}, the constant $\zeta_{\psi_{k_2}}$ depends on $\zeta_{\psi_{k_1}}$ in such a way that $\zeta_{\psi_{k_2}}(\zeta_{\psi_{k_1}})\to 0$ when $\zeta_{\psi_{k_1}}$ tends to 0. One can apply $q-$Laplace transform of order $k_2$ to the function $\psi_{k_2}^{\mathfrak{d}_{p}}$ in $\tau$ variable and, in direction $\mathfrak{d}_{p}$, and obtain that the function 
\begin{equation}\label{e875b}
F^{\mathfrak{d}_{p}}(T,m,\epsilon):=\mathcal{L}^{\mathfrak{d}_{p}}_{q;1/k_2}(\tau\mapsto\psi_{k_2}^{\mathfrak{d}_{p}}(\tau,m,\epsilon))(T),
\end{equation}
is a holomorphic function with respect to $T$ variable in the set $\mathcal{R}_{\mathfrak{d}_{p},\tilde{\delta}}\cap D(0,r_1)$ for any $0<r_1\le q^{\left(\frac{1}{2}-\nu\right)/k_2}/2$.

We define the forcing term $f^{\mathfrak{d}_{p}}(t,z,\epsilon)$ by
\begin{equation}\label{e876}
f^{\mathfrak{d}_{p}}(t,z,\epsilon):=\mathcal{F}^{-1}\left(m\mapsto F^{\mathfrak{d}_{p}}(\epsilon t,m,\epsilon)\right)(z),
\end{equation}
which turns out to be a bounded holomorphic function defined on $\mathcal{T}\times H_{\beta'}\times\mathcal{E}_{p}$ provided that (\ref{e824}) holds.

The next results provide estimates of the difference of two consecutive solutions of the equation (\ref{e771}) with respect to the perturbation parameter. They can be of two different nature depending on the existence or not of singularities of some auxiliary equation in between the integration lines where the solutions of (\ref{e771}) are constructed. This phenomena, studied in the sequel, turns out to be the reason for different levels to appear on the asymptotic behavior of the solution with respect to the perturbation parameter.  

We first describe the procedure to solve equation (\ref{e771}), and the nature of its solution, crucial in the asymptotic behavior to be described afterwards.

\begin{theo}\label{teo872}
Under the construction made at the beginning of Section~\ref{seccion5}, assume that the conditions (\ref{e796}), (\ref{e800}), (\ref{e804}) and (\ref{e814}) hold. Let $(\mathcal{E}_{p})_{0\le p\le \varsigma -1}$ be a good covering in $\C^{\star}$, for which a family $\{(\mathcal{R}_{\mathfrak{d}_{p},\tilde{\delta}})_{0\le p\le \varsigma-1},D(0,\rho),\mathcal{T}\}$ associated to this covering is considered.

Then, there exist large enough $r_{Q,R_{D}}$, and constants $\zeta_{\psi}>0$ and $\zeta_{\lambda,\ell}>0$ for $0\le \ell\le D-1$ and $\lambda\in I_{\ell}$ such that if
$$\tilde{C}_{F}\le \zeta_{\psi},\quad \tilde{C}_{\lambda,\ell}\le \zeta_{\lambda,\ell},$$
then, for every $0\le p\le\varsigma-1$, one can construct a solution $u^{\mathfrak{d}_{p}}(t,z,\epsilon)$ of (\ref{e771}), which defines a holomorphic function on $\mathcal{T}\times H_{\beta'}\times\mathcal{E}_{p}$, for every $0<\beta'<\beta$.
\end{theo}
\begin{proof}
Let $0\le p\le \varsigma-1$ and consider the equation
\begin{align}&Q(im)\sigma_{q}U^{\mathfrak{d}_{p}}(T,m,\epsilon)=T^{d_{D}}\sigma_{q}^{\frac{d_{D}}{k_2}+1}R_{D}(im)U^{\mathfrak{d}_{p}}(T,m,\epsilon)\nonumber\\
+&\sum_{\ell=1}^{D-1}\left(\sum_{\lambda\in I_{\ell}}T^{d_{\lambda,\ell}}\epsilon^{\Delta_{\lambda,\ell}-d_{\lambda,\ell}}\frac{1}{(2\pi)^{1/2}}\int_{-\infty}^{\infty}C_{\lambda,\ell}(m-m_1,\epsilon)R_{\ell}(im_1)U^{\mathfrak{d}_{p}}(q^{\delta_{\ell}}T,m_1,\epsilon)dm_1\right)\nonumber \\
&\hspace{8cm}+\sigma_{q}F^{\mathfrak{d}_{p}}(T,m,\epsilon).\label{e882}
\end{align}
Under an appropriate choice of the constants $\zeta_{\psi}$ and $\zeta_{\lambda,\ell}$ for all $0\le \ell\le D-1$ and $\lambda\in I_{\ell}$, one can follow the constructions in Section~\ref{seccion4} and the properties of $q-$Laplace transformation described in Proposition~\ref{prop259} in order to apply Proposition~\ref{prop571} and obtain a solution $U^{\mathfrak{d}_{p}}(T,m,\epsilon)$ of (\ref{e882}). This function can be written as the $q-$Laplace transform of order $k_2$ in the form
\begin{equation}\label{e886}
U^{\mathfrak{d}_{p}}(T,m,\epsilon)=\frac{1}{\pi_{q^{1/k_2}}}\int_{L_{\gamma_p}}\frac{w^{\mathfrak{d}_{p}}_{k_2}(u,m,\epsilon)}{\Theta_{q^{1/k_2}}\left(\frac{u}{T}\right)}\frac{du}{u},
\end{equation}
where $L_{\gamma_p}=\R_{+}e^{\sqrt{-1}\gamma_{p}}\subseteq S_{\mathfrak{d}_{p}}\cup\{0\}$ is a halfline with direction depending on $T$. Here, $w^{\mathfrak{d}_{p}}_{k_2}(\tau,m,\epsilon)$ defines a continuous function on $(\overline{\mathcal{R}_{\mathfrak{d}_{p}}^{b}}\cup S_{\mathfrak{d}_{p}})\times\R\times D(0,\epsilon_0)$, which is holomorphic with respect to $(\tau,\epsilon)$ in $(\mathcal{R}_{\mathfrak{d}_{p}}^{b}\cup S_{\mathfrak{d}_{p}})\times D(0,\epsilon_0)$ for every $m\in\R$. Moreover, it satisfies there exists $C_{w^{\mathfrak{d}_{p}}_{k_2}}>0$ such that
\begin{equation}\label{e890}
|w^{\mathfrak{d}_{p}}_{k_2}(\tau,m,\epsilon)|\le C_{w^{\mathfrak{d}_{p}}_{k_2}}\frac{1}{(1+|m|)^{\mu}}e^{-\beta |m|}\exp\left(\frac{k_2\log^2|\tau|}{2\log(q)}+\nu\log|\tau|\right),
\end{equation}
for some $\nu\in\R$. This is valid for $\tau\in\mathcal{R}_{\mathfrak{d}_{p}}^{b}\cup S_{\mathfrak{d}_{p}}$, $m\in\R$ and $\epsilon\in D(0,\epsilon_0)$. Taking into account Proposition~\ref{prop638}, the function $w^{\mathfrak{d}_{p}}_{k_2}(\tau,m,\epsilon)$ is the analytic continuation with respect to $\tau$ variable of the function given by
\begin{equation}\label{e894}
\tau\mapsto \frac{1}{\pi_{q^{1/\kappa}}}\int_{L_{\gamma^1_p}}\frac{w^{\mathfrak{d}_{p}}_{k_1}(u,m,\epsilon)}{\Theta_{q^{1/\kappa}}\left(\frac{u}{T}\right)}\frac{du}{u},
\end{equation}
where $L_{\gamma^1_p}=\R_{+}e^{\sqrt{-1}\gamma^1_{p}}\subseteq S_{\mathfrak{d}_{p}}\cup\{0\}$ is a halfline with direction depending on $\tau$, and analytic in  the set $\mathcal{R}_{\mathfrak{d}_{p},\tilde{\delta}}\cap D(0,r_{1})$, for $0<r_1\le q^{\left(\frac{1}{2}-\alpha\right)/\kappa}/2$, for some $\alpha\in\R$. The function $w^{\mathfrak{d}_{p}}_{k_1}(\tau,m,\epsilon)$ is a continuous function defined on $(\overline{D}(0,\rho)\cup U_{\mathfrak{d}_{p}})\times \R\times D(0,\epsilon_0)$ and holomorphic with respect to $(\tau, \epsilon)$ on $(D(0,\rho)\cup U_{\mathfrak{d}_{p}})\times D(0,\epsilon_0)$ for every $m\in\R$. In addition to that, one has 

\begin{equation}\label{e898}
|w^{\mathfrak{d}_{p}}_{k_1}(\tau,m,\epsilon)|\le C_{w^{\mathfrak{d}_{p}}_{k_1}}\frac{1}{(1+|m|)^{\mu}}e^{-\beta|m|}\exp\left(\frac{\kappa\log^2|\tau+\delta|}{2\log(q)}+\alpha\log|\tau+\delta|\right),
\end{equation}
for some $C_{w^{\mathfrak{d}_{p}}_{k_1}},\delta>0$, valid for every $\tau\in(D(0,\rho)\cup U_{\mathfrak{d}_{p}})$, $m\in\R$ and $\epsilon\in D(0,\epsilon_0)$. Indeed, $w^{\mathfrak{d}_{p}}_{k_1}$ is the extension of a function $w_{k_1}(\tau,m,\epsilon)$, common for every $0\le p\le \varsigma-1$, which is continuous on $\overline{D}(0,\rho)\times\R\times D(0,\epsilon_0)$ and holomorphic with respect to $(\tau,\epsilon)$ on $D(0,\rho)\times D(0,\epsilon_0)$.

The bounds attained in (\ref{e890}) with respect to $m$ variable are transmitted to the function $U^{\mathfrak{d}_{p}}(T,m,\epsilon)$ described in (\ref{e886}). This guarantees one can define $\mathcal{F}^{-1}(m\mapsto U^{\mathfrak{d}_{p}}(T,m,\epsilon))(z)$ in such a way that the function

\begin{align}
u^{\mathfrak{d}_{p}}(t,z,\epsilon):=&\mathcal{F}^{-1}(m\mapsto U^{\mathfrak{d}_{p}}(\epsilon t,m,\epsilon))(z) \nonumber \\
=&\frac{1}{(2\pi)^{1/2}}\frac{1}{\pi_{q^{1/k_2}}}\int_{-\infty}^{\infty}\int_{L_{\gamma_p}}\frac{w^{\mathfrak{d}_{p}}_{k_2}(u,m,\epsilon)}{\Theta_{q^{1/k_2}}\left(\frac{u}{\epsilon t}\right)}\frac{du}{u}\exp(izm)dm, \label{e910}
\end{align}

defines a bounded holomorphic function on $\mathcal{T}\times H_{\beta'}\times\mathcal{E}_{p}$, in view of 3) in Definition~\ref{def818}.

The properties held by inverse Fourier transform, described in Proposition~\ref{prop267}, allow us to conclude that $u^{\mathfrak{d}_{p}}(t,z,\epsilon)$ is a solution of the equation (\ref{e771}) defined on $\mathcal{T}\times H_{\beta'}\times\mathcal{E}_{p}$.
\end{proof}

\begin{prop}\label{prop925}
Let $0\le p \le \varsigma -1$. Under the hypotheses of Theorem~\ref{teo872}, assume that the unbounded sectors $U_{\mathfrak{d}_{p}}$ and $U_{\mathfrak{d}_{p+1}}$ are wide enough so that $U_{\mathfrak{d}_{p}}\cap U_{\mathfrak{d}_{p+1}}$ contains the sector $U_{\mathfrak{d}_{p},\mathfrak{d}_{p+1}}=\{\tau\in\C^{\star}:\arg(\tau)\in[\mathfrak{d}_{p},\mathfrak{d}_{p+1}]\}$. Then, there exist $K_1>0$ and $K_2\in\R$ such that
\begin{align}
&|u^{\mathfrak{d}_{p+1}}(t,z,\epsilon)-u^{\mathfrak{d}_{p}}(t,z,\epsilon)|\le K_{1}\exp\left(-\frac{k_2}{2\log(q)}\log^2|\epsilon|\right)|\epsilon|^{K_2},\nonumber\\
&\hspace{3cm}|f^{\mathfrak{d}_{p+1}}(t,z,\epsilon)-f^{\mathfrak{d}_{p}}(t,z,\epsilon)|\le K_{1}\exp\left(-\frac{k_2}{2\log(q)}\log^2|\epsilon|\right)|\epsilon|^{K_2}, \label{e927}
\end{align}
for every $t\in\mathcal{T}$, $z\in H_{\beta'}$, and $\epsilon\in\mathcal{E}_{p}\cap\mathcal{E}_{p+1}$. 

\end{prop}
\begin{proof}
Let $0\le p \le \varsigma -1$. Taking into account that $U_{\mathfrak{d}_{p},\mathfrak{d}_{p+1}}\subseteq U_{\mathfrak{d}_{p}}\cap U_{\mathfrak{d}_{p+1}}$, we observe from the construction of the functions $U^{\mathfrak{d}_{p}}$ and $U^{\mathfrak{d}_{p+1}}$ that $\mathcal{L}^{\mathfrak{d}_{p}}_{q;1/\kappa}(w_{k_1}^{\mathfrak{d}_{p}})(\tau,m,\epsilon)$ and $\mathcal{L}^{\mathfrak{d}_{p+1}}_{q;1/\kappa}(w_{k_1}^{\mathfrak{d}_{p+1}})(\tau,m,\epsilon)$ coincide in the domain $(\mathcal{R}_{\mathfrak{d}_{p}}^{b}\cap \mathcal{R}_{\mathfrak{d}_{p+1}}^{b})\times \R\times D(0,\epsilon_0)$. This entails the existence of $w_{k_2}^{\mathfrak{d}_{p},\mathfrak{d}_{p+1}}(\tau,m,\epsilon)$, holomorphic with respect to $\tau$ on $\mathcal{R}_{\mathfrak{d}_{p}}^{b}\cup \mathcal{R}_{\mathfrak{d}_{p+1}}^{b}$, continuous with respect to $m\in\R$ and holomorphic with respect to $\epsilon$ in $D(0,\epsilon_0)$ which coincides with $\mathcal{L}^{\mathfrak{d}_{p}}_{q;1/\kappa}(w_{k_1}^{\mathfrak{d}_{p}})(\tau,m,\epsilon)$ on $\mathcal{R}^{b}_{\mathfrak{d}_{p}}\times\R\times D(0,\epsilon_0)$ and also with $\mathcal{L}^{\mathfrak{d}_{p+1}}_{q;1/\kappa}(w_{k_1}^{\mathfrak{d}_{p+1}})(\tau,m,\epsilon)$ on $\mathcal{R}^{b}_{\mathfrak{d}_{p+1}}\times\R\times D(0,\epsilon_0)$.

Let $\rho_{\mathfrak{d}_{p},\mathfrak{d}_{p+1}}$ be such that $\rho_{\mathfrak{d}_{p},\mathfrak{d}_{p+1}}e^{i\gamma_{p}}\subseteq\mathcal{R}^{b}_{\mathfrak{d}_{p}}$ and $\rho_{\mathfrak{d}_{p},\mathfrak{d}_{p+1}}e^{i\gamma_{p+1}}\subseteq\mathcal{R}^{b}_{\mathfrak{d}_{p+1}}$. The function 
$$u\mapsto \frac{w_{k_2}^{\mathfrak{d}_{p},\mathfrak{d}_{p+1}}(u,m,\epsilon)}{\Theta_{q^{1/k_2}}\left(\frac{u}{\epsilon t}\right)}$$
is holomorphic on $\mathcal{R}^{b}_{\mathfrak{d}_{p}}\cup \mathcal{R}^{b}_{\mathfrak{d}_{p+1}}$ for all $(m,\epsilon)\in\R\times (\mathcal{E}_{p}\cap \mathcal{E}_{p+1})$  and its integral along the closed path constructed by concatenation of the segment starting at the origin and with ending point fixed at $\rho_{\mathfrak{d}_{p},\mathfrak{d}_{p+1}}e^{i\gamma_{p}}$, the arc of circle with radius $\rho_{\mathfrak{d}_{p},\mathfrak{d}_{p+1}}$ connecting $\rho_{\mathfrak{d}_{p},\mathfrak{d}_{p+1}}e^{i\gamma_{p}}$ with $\rho_{\mathfrak{d}_{p},\mathfrak{d}_{p+1}}e^{i\gamma_{p+1}}\subseteq\mathcal{R}^{b}_{\mathfrak{d}_{p+1}}$, and the segment from $\rho_{\mathfrak{d}_{p},\mathfrak{d}_{p+1}}e^{i\gamma_{p+1}}$ to 0, vanishes. The difference $u^{\mathfrak{d}_{p+1}}-u^{\mathfrak{d}_{p}}$ can be written in the form

$$u^{\mathfrak{d}_{p+1}}(t,z,\epsilon)-u^{\mathfrak{d}_{p}}(t,z,\epsilon)$$
\begin{align}
&=\frac{1}{(2\pi)^{1/2}}\frac{1}{\pi_{q^{1/k_2}}}\int_{-\infty}^{\infty}\int_{L_{\gamma_{p+1},\rho_{\mathfrak{d}_{p},\mathfrak{d}_{p+1}}}}\frac{w^{\mathfrak{d}_{p+1}}_{k_2}(u,m,\epsilon)}{\Theta_{q^{1/k_2}}\left(\frac{u}{\epsilon t}\right)}\exp(izm)\frac{du}{u}dm,\hspace{3cm}\nonumber\\ 
&\hspace{2cm}-\frac{1}{(2\pi)^{1/2}}\frac{1}{\pi_{q^{1/k_2}}}\int_{-\infty}^{\infty}\int_{L_{\gamma_p,\rho_{\mathfrak{d}_{p},\mathfrak{d}_{p+1}}}}\frac{w^{\mathfrak{d}_{p}}_{k_2}(u,m,\epsilon)}{\Theta_{q^{1/k_2}}\left(\frac{u}{\epsilon t}\right)}\exp(izm)\frac{du}{u}dm\nonumber\\
&\hspace{4cm}+\frac{1}{(2\pi)^{1/2}}\frac{1}{\pi_{q^{1/k_2}}}\int_{-\infty}^{\infty}\int_{C_{\rho_{\mathfrak{d}_{p},\mathfrak{d}_{p+1}},\gamma_p,\gamma_{p+1}}}\frac{w^{\mathfrak{d}_{p},\mathfrak{d}_{p+1}}_{k_2}(u,m,\epsilon)}{\Theta_{q^{1/k_2}}\left(\frac{u}{\epsilon t}\right)}\exp(izm)\frac{du}{u}dm,\label{e943}
\end{align}
where $L_{\gamma_{j},\rho_{\mathfrak{d}_{p},\mathfrak{d}_{p+1}}}=[\rho_{\mathfrak{d}_{p},\mathfrak{d}_{j}},+\infty)e^{i\gamma_{j}}$ for $j\in\{p,p+1\}$ and $C_{\rho_{\mathfrak{d}_{p},\mathfrak{d}_{p+1}},\gamma_p,\gamma_{p+1}}$ is the arc of circle connecting $\rho_{\mathfrak{d}_{p},\mathfrak{d}_{p+1}}e^{i\gamma_{p}}$ with $\rho_{\mathfrak{d}_{p},\mathfrak{d}_{p+1}}e^{i\gamma_{p+1}}$ (see Figure 1).

\begin{figure}[h]
	\centering
		\label{fig:conf0}
		\includegraphics[width=0.50\textwidth]{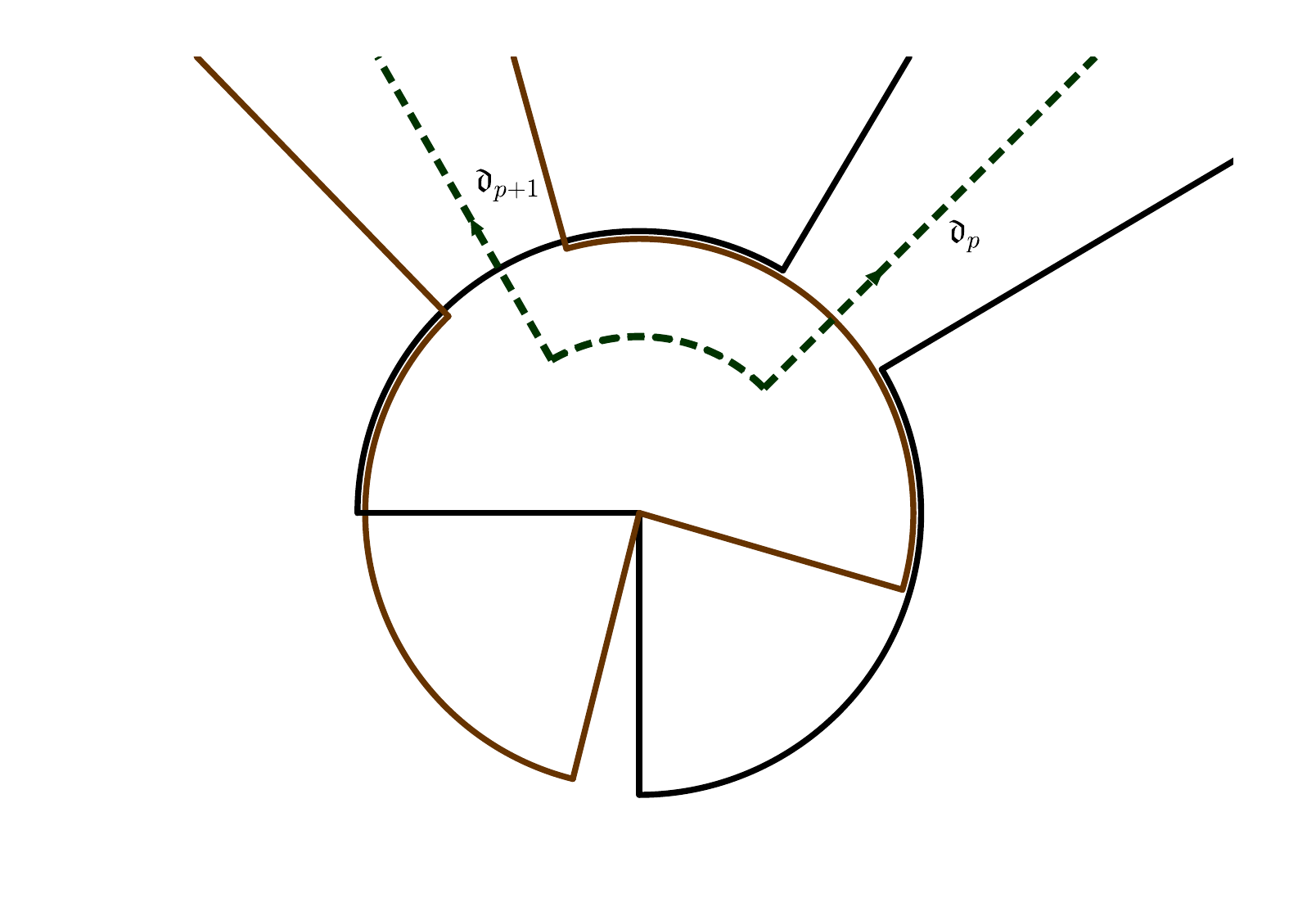}
	\caption{Deformation of the path of integration, first case.}
\end{figure}

Let us put
$$I_{1}:=\left|\frac{1}{(2\pi)^{1/2}}\frac{1}{\pi_{q^{1/k_2}}}\int_{-\infty}^{\infty}\int_{L_{\gamma_{p+1},\rho_{\mathfrak{d}_{p},\mathfrak{d}_{p+1}}}}\frac{w^{\mathfrak{d}_{p+1}}_{k_2}(u,m,\epsilon)}{\Theta_{q^{1/k_2}}\left(\frac{u}{\epsilon t}\right)}\exp(izm)\frac{du}{u}dm\right|.$$
In view of (\ref{e890}) and (\ref{e200}), one has
\begin{align}
I_{1}&\le \frac{\tilde{C}_{w_{k_2}^{\mathfrak{d}_{p+1}}}}{C_{q,k_2}\tilde{\delta}(2\pi)^{1/2}}\frac{|\epsilon t|^{1/2}}{\pi_{q^{1/k_2}}}\int_{-\infty}^{\infty}e^{-\beta|m|-m\Im(z)}\frac{dm}{(1+|m|)^{\mu}}\nonumber\\
&\times\int_{\rho_{\mathfrak{d}_{p},\mathfrak{d}_{p+1}}}^{\infty}\exp\left(\frac{k_2\log^2|u|}{2\log(q)}+\nu\log|u|\right)|u|^{-3/2}\exp\left(-\frac{k_2\log^2\left(\frac{|u|}{|\epsilon t|}\right)}{2\log(q)}\right)d|u|.\label{e953}
\end{align}
We recall that we have restricted the domain on the variable $z$ such that $|\Im(z)|\le\beta'<\beta$. Then, the first integral in the previous expression in convergent, and one derives
$$I_{1}\le \frac{\tilde{C}_{w_{k_2}^{\mathfrak{d}_{p+1}}}}{(2\pi)^{1/2}}\frac{(\epsilon_0 r_{\mathcal{T})^{1/2}}}{\pi_{q^{1/k_2}}}\int_{\rho_{\mathfrak{d}_{p},\mathfrak{d}_{p+1}}}^{\infty}\exp\left(\frac{k_2\log^2|u|}{2\log(q)}\right)\exp\left(-\frac{k_2\log^2\left(\frac{|u|}{|\epsilon t|}\right)}{2\log(q)}\right)|u|^{\nu-3/2}d|u|,$$
for some $\tilde{C}_{w_{k_2}^{\mathfrak{d}_{p+1}}}>0$.
We derive
\begin{align*}
\exp\left(\frac{k_2\log^2|u|}{2\log(q)}\right)\exp\left(-\frac{k_2\log^2\left(\frac{|u|}{|\epsilon t|}\right)}{2\log(q)}\right)&=\exp\left(\frac{k_2}{2\log(q)}(-\log^2|\epsilon|-2\log|\epsilon|\log|t|-\log^2|t|)\right)\\
&\times \exp\left(\frac{k_2}{\log(q)}(\log|u|\log|\epsilon|+\log|u|\log|t|)\right).
\end{align*}
From the assumption that $0<\epsilon_0<1$ and $0<r_{\mathcal{T}}<1$, we get
\begin{equation}\label{e964}
\exp\left(-\frac{k_2}{\log(q)}\log|\epsilon|\log|t|\right)\le|\epsilon|^{-\frac{k_2}{\log(q)}\log(r_{\mathcal{T}})},\quad \exp\left(\frac{k_2}{\log(q)}\log|u|\log|\epsilon|\right)\le |\epsilon|^{\frac{k_2}{\log(q)}\log(\rho_{\mathfrak{d}_{p},\mathfrak{d}_{p+1}})},
\end{equation}
for $t\in\mathcal{T}$, $\epsilon\in\mathcal{E}_{p}\cap\mathcal{E}_{p+1}$, $|u|\ge \rho_{\mathfrak{d}_{p},\mathfrak{d}_{p+1}},$ and also
\begin{align}
&\exp\left(\frac{k_2}{\log(q)}\log|u|\log|t|\right)\le|t|^{\frac{k_2}{\log(q)}\log(\rho_{\mathfrak{d}_{p},\mathfrak{d}_{p+1}})}, \hbox{ if } \rho_{\mathfrak{d}_{p},\mathfrak{d}_{p+1}}\le |u|\le 1\nonumber \\
&\hspace{5cm}\exp\left(\frac{k_2}{\log(q)}\log|u|\log|t|\right)\le |u|^{\frac{k_2}{\log(q)}\log(r_{\mathcal{T}})},\hbox{ if } |u|\ge 1,\label{e968}
\end{align}
for $t\in\mathcal{T}$. In addition to that, there exists $K_{k_2,\rho_{\mathfrak{d}_{p},\mathfrak{d}_{p+1}},q}>0$ such that
\begin{equation}\label{e973}
\sup_{x>0}x^{\frac{k_2}{\log(q)}\log(\rho_{\mathfrak{d}_{p},\mathfrak{d}_{p+1}})}\exp\left(-\frac{k_2}{2\log(q)}\log^2(x)\right)\le K_{k_2,\rho_{\mathfrak{d}_{p},\mathfrak{d}_{p+1}},q}.
\end{equation}

In view of (\ref{e964}), (\ref{e968}), (\ref{e973}), and bearing in mind that (\ref{e824}) holds, we deduce there exist $\tilde{K}^{1}\in\R$, $\tilde{K}^{2}>0$ such that
$$ \exp\left(\frac{k_2\log^2|u|}{2\log(q)}\right)\exp\left(-\frac{k_2\log^2\left(\frac{|u|}{|\epsilon t|}\right)}{2\log(q)}\right)|u|^{\nu}\le \tilde{K}^2\exp\left(-\frac{k_2}{2\log(q)}\log^2|\epsilon|\right)|\epsilon|^{\tilde{K}^{1}},$$
for $t\in\mathcal{T}$, $r\ge \rho_{\mathfrak{d}_{p},\mathfrak{d}_{p+1}}$, and $\epsilon\in\mathcal{E}_{p}\cap\mathcal{E}_{p+1}$. Provided this last inequality, 
we arrive at
\begin{align}
I_{1}&\le\frac{\tilde{K}^{2}C_{w_{k_2}^{\mathfrak{d}_{p+1}}}}{C_{q,k_2}\tilde{\delta}(2\pi)^{1/2}}\frac{(\epsilon_0 r_{\mathcal{T})^{1/2}}}{\pi_{q^{1/k_2}}}\int_{\rho_{\mathfrak{d}_{p},\mathfrak{d}_{p+1}}}^{\infty}\frac{d|u|}{|u|^{3/2}}\exp\left(-\frac{k_2}{2\log(q)}\log^2|\epsilon|\right)|\epsilon|^{\tilde{K}^{1}}\nonumber \\
&=\tilde{K}^3\exp\left(-\frac{k_2}{2\log(q)}\log^2|\epsilon|\right)|\epsilon|^{\tilde{K}^{1}}\label{e983},
\end{align}
for some $\tilde{K}^{3}>0$, for all $t\in\mathcal{T}$, $z\in H_{\beta'}$, and $\epsilon\in\mathcal{E}_{p}\cap\mathcal{E}_{p+1}$.

We can estimate in the same manner the expression
$$I_{2}:=\left|\frac{1}{(2\pi)^{1/2}}\frac{1}{\pi_{q^{1/k_2}}}\int_{-\infty}^{\infty}\int_{L_{\gamma_{p},\rho_{\mathfrak{d}_{p},\mathfrak{d}_{p+1}}}}\frac{w^{\mathfrak{d}_{p}}_{k_2}(u,m,\epsilon)}{\Theta_{q^{1/k_2}}\left(\frac{u}{\epsilon t}\right)}\exp(izm)\frac{du}{u}dm\right|.$$
to arrive at the existence of $\tilde{K}^{4}>0$ such that
\begin{equation}\label{e990}
I_{2}\le\tilde{K}^4\exp\left(-\frac{k_2}{2\log(q)}\log^2|\epsilon|\right)|\epsilon|^{\tilde{K}^{1}},
\end{equation}
for all $t\in\mathcal{T}$, $z\in H_{\beta'}$, and $\epsilon\in\mathcal{E}_{p}\cap\mathcal{E}_{p+1}$. We now provide upper bounds for the quantity
$$I_{3}:=\left|\frac{1}{(2\pi)^{1/2}}\frac{1}{\pi_{q^{1/k_2}}}\int_{-\infty}^{\infty}\int_{C_{\rho_{\mathfrak{d}_{p},\mathfrak{d}_{p+1}},\gamma_p,\gamma_{p+1}}}\frac{w^{\mathfrak{d}_{p},\mathfrak{d}_{p+1}}_{k_2}(u,m,\epsilon)}{\Theta_{q^{1/k_2}}\left(\frac{u}{\epsilon t}\right)}\exp(izm)\frac{du}{u}dm\right|.$$
The estimates in (\ref{e890}) and (\ref{e200}) allow us to obtain the existence of $C^{\mathfrak{d}_{p}\mathfrak{d}_{p+1}}_{w_{k_2}}>0$ such that
$$I_3\le\frac{C_{w_{k_2}}^{\mathfrak{d}_{p},\mathfrak{d}_{p+1}}}{(2\pi)^{1/2}}\frac{1}{\pi_{q^{1/k_2}}}\frac{\epsilon_0^{1/2}}{C_{q,k_2}\tilde{\delta}\rho_{\mathfrak{d}_{p},\mathfrak{d}_{p+1}}^{1/2}}\int_{-\infty}^{\infty}\frac{e^{-\beta|m|-m\Im(z)}}{(1+|m|)^{\mu}}dm|\gamma_{p+1}-\gamma_{p}||t|^{1/2}\exp\left(-\frac{k_2\log^2\left(\frac{\rho_{\mathfrak{d}_{p}\mathfrak{d}_{p+1}}}{|\epsilon t|}\right)}{2\log(q)}\right),$$
for all $t\in\mathcal{T}$, $z\in H_{\beta'}$, and $\epsilon\in\mathcal{E}_{p}\cap\mathcal{E}_{p+1}$. We can follow analogous arguments as in the previous steps to provide upper estimates of the expression
$$|t|^{1/2}\exp\left(-\frac{k_2\log^2\left(\frac{\rho_{\mathfrak{d}_{p}\mathfrak{d}_{p+1}}}{|\epsilon t|}\right)}{2\log(q)}\right).$$
Indeed, 
\begin{align*}
|t|^{1/2}\exp\left(-\frac{k_2\log^2\left(\frac{\rho_{\mathfrak{d}_{p}\mathfrak{d}_{p+1}}}{|\epsilon t|}\right)}{2\log(q)}\right)&=\exp\left(-\frac{k_2\log^2(\rho_{\mathfrak{d}_{p}\mathfrak{d}_{p+1}})}{2\log(q)}\right)|\epsilon|^{\frac{k_2\log(\rho_{\mathfrak{d}_{p}\mathfrak{d}_{p+1}})}{\log(q)}}|t|^{\frac{k_2\log(\rho_{\mathfrak{d}_{p}\mathfrak{d}_{p+1}})}{\log(q)}}\\
&\times \exp\left(\frac{k_2}{2\log(q)}(-\log^2|\epsilon|-2\log|\epsilon|\log|t|-\log^2|t|)\right)|t|^{1/2}.
\end{align*}
From the assumption $0\le \epsilon_0 <1$ we check that
$$\exp\left(-\frac{k_2}{\log(q)}\log|\epsilon|\log|t|\right)\le|\epsilon|^{-\frac{k_2}{\log(q)}\log(r_{\mathcal{T}})},$$
for $t\in\mathcal{T}$, $\epsilon\in\mathcal{E}_{p}\cap\mathcal{E}_{p+1}$. Gathering (\ref{e973}), we get the existence of $\tilde{K}^{5}\in\R$, $\tilde{K}^{6}>0$ such that
$$|t|^{1/2}\exp\left(-\frac{k_2\log^2\left(\frac{\rho_{\mathfrak{d}_{p}\mathfrak{d}_{p+1}}}{|\epsilon t|}\right)}{2\log(q)}\right)\le \tilde{K}^{6}\exp\left(-\frac{k_2}{2\log(q)}\log^2|\epsilon|\right)|\epsilon|^{\tilde{K}^{5}},$$
to conclude that 
\begin{equation}\label{e1009}
I_{3}\le \tilde{K}^{7}\exp\left(-\frac{k_2}{2\log(q)}\log^2|\epsilon|\right)|\epsilon|^{\tilde{K}^{5}},
\end{equation}
for some $\tilde{K}^{7}>0$, all $t\in\mathcal{T}$, $z\in H_{\beta'}$, and $\epsilon\in\mathcal{E}_{p}\cap\mathcal{E}_{p+1}$.
We conclude the proof of this result in view of (\ref{e983}), (\ref{e990}), (\ref{e1009}) and the decomposition (\ref{e943}).

in order to obtain analogous estimates for the forcing term $f^{\mathfrak{d}_{p}}$, one can follow analogous estimates as for $u^{\mathfrak{d}_{p}}$ under the consideration of the estimates in (\ref{e868}).
\end{proof}

We now state the second situation one can find when estimating the difference of two consecutive solutions. For these purpose, we enunciate the next
\begin{lemma}\label{lema1031}
Let $0\le p \le \varsigma -1$. Under the hypotheses of Theorem~\ref{teo872}, assume that $U_{\mathfrak{d}_{p}}\cap U_{\mathfrak{d}_{p+1}}=\emptyset$. Then, there exist $K_p^{\mathcal{L}}>0$, $M_p^{\mathcal{L}}\in\R$ such that
\begin{equation}\label{e1020}
\left|\mathcal{L}_{q;1/\kappa}^{\mathfrak{d}_{p+1}}(w_{k_1}^{\mathfrak{d}_{p+1}})(\tau,m,\epsilon)-\mathcal{L}_{q;1/\kappa}^{\mathfrak{d}_{p}}(w_{k_1}^{\mathfrak{d}_{p}})(\tau,m,\epsilon)\right|\le K_p^{\mathcal{L}}e^{-\beta|m|}(1+|m|)^{-\mu}\exp\left(-\frac{\kappa}{2\log(q)}\log^2|\tau|\right)|\tau|^{M_p^{\mathcal{L}}},
\end{equation}
for every $\epsilon\in(\mathcal{E}_{p}\cap\mathcal{E}_{p+1})$, $\tau\in(\mathcal{R}^{b}_{\mathfrak{d}_{p}}\cap\mathcal{R}^{b}_{\mathfrak{d}_{p+1}})$ and $m\in\R$.
\end{lemma}
\begin{proof}
We first recall that, without loss of generality, the intersection $\mathcal{R}_{\mathfrak{d}_{p}}^d\cap \mathcal{R}_{\mathfrak{d}_{p+1}}^d$ can be assumed to be a nonempty set because one  can vary $\tilde{\delta}$ in advance to be as close to 0 as desired.

Analogous arguments as in the beginning of the proof of Proposition~\ref{prop925} allow us to write
$$\mathcal{L}_{q;1/\kappa}^{\mathfrak{d}_{p+1}}(w_{k_1}^{\mathfrak{d}_{p+1}})(\tau,m,\epsilon)-\mathcal{L}_{q;1/\kappa}^{\mathfrak{d}_{p}}(w_{k_1}^{\mathfrak{d}_{p}})(\tau,m,\epsilon)$$
in the form  
\begin{align}
&=\frac{1}{\pi_{q^{1/\kappa}}}\int_{L_{\gamma_{p+1},\rho_{\mathfrak{d}_{p},\mathfrak{d}_{p+1}}}}\frac{w^{\mathfrak{d}_{p+1}}_{k_1}(u,m,\epsilon)}{\Theta_{q^{1/\kappa}}\left(\frac{u}{\tau}\right)}\frac{du}{u},\hspace{3cm}\nonumber\\ 
&\hspace{2cm}-\frac{1}{\pi_{q^{1/\kappa}}}\int_{L_{\gamma_p,\rho_{\mathfrak{d}_{p},\mathfrak{d}_{p+1}}}}\frac{w^{\mathfrak{d}_{p}}_{k_1}(u,m,\epsilon)}{\Theta_{q^{1/\kappa}}\left(\frac{u}{\tau}\right)}\frac{du}{u}\nonumber\\
&\hspace{4cm}+\frac{1}{\pi_{q^{1/\kappa}}}\int_{C_{\rho_{\mathfrak{d}_{p},\mathfrak{d}_{p+1}},\gamma_p,\gamma_{p+1}}}\frac{w_{k_1}(u,m,\epsilon)}{\Theta_{q^{1/\kappa}}\left(\frac{u}{\tau}\right)}\frac{du}{u},\label{e943b}
\end{align}
where $\rho_{\mathfrak{d}_{p},\mathfrak{d}_{p+1}}$, $L_{\gamma_{p},\rho_{\mathfrak{d}_{p},\mathfrak{d}_{p+1}}}$, $L_{\gamma_{p+1},\rho_{\mathfrak{d}_{p},\mathfrak{d}_{p+1}}}$ and $C_{\rho_{\mathfrak{d}_{p},\mathfrak{d}_{p+1}},\gamma_p,\gamma_{p+1}}$ are constructed in Proposition~\ref{prop925}.

In view of (\ref{e898}) and (\ref{e200}), one has
$$I^{\mathcal{L}}_1:=\left|\frac{1}{\pi_{q^{1/\kappa}}}\int_{L_{\gamma_{p},\rho_{\mathfrak{d}_{p},\mathfrak{d}_{p+1}}}}\frac{w^{\mathfrak{d}_{p}}_{k_1}(u,m,\epsilon)}{\Theta_{q^{1/\kappa}}\left(\frac{u}{\tau}\right)}\frac{du}{u}
\right|$$
\begin{align*}
&\le \frac{C_{w^{\mathfrak{d}_{p}}_{k_1}}}{C_{q,\kappa}\tilde{\delta}}\frac{|\tau|^{1/2}}{(1+|m|)^{\mu}}e^{-\beta|m|}\int_{\rho_{\mathfrak{d}_{p},\mathfrak{d}_{p+1}}}^{\infty}\frac{\exp\left(\frac{\kappa\log^2|re^{i\gamma_{p}}+\delta|}{2\log(q)}+\alpha\log|re^{i\gamma_{p}}+\delta|\right)}{\exp\left(\frac{\kappa}{2}\frac{\log^2\left(\frac{r}{|\tau|}\right)}{\log(q)}\right)}\frac{dr}{r^{3/2}}\\
&\le K_{p,1}^{\mathcal{L}}|\tau|^{1/2}(1+|m|)^{-\mu}e^{-\beta|m|}\int_{\rho_{\mathfrak{d}_{p},\mathfrak{d}_{p+1}}}^{\infty}\frac{\exp\left(\frac{\kappa\log^2r}{2\log(q)}+\alpha\log r\right)}{\exp\left(\frac{\kappa}{2}\frac{\log^2\left(\frac{r}{|\tau|}\right)}{\log(q)}\right)}\frac{dr}{r^{3/2}}
\end{align*}
for some $K_{p,1}^{\mathcal{L}}>0$. Usual calculations, and taking into account the choice of $\alpha$ in (\ref{e824}) one derives the previous expression equals
$$K_{p,1}^{\mathcal{L}}|\tau|^{1/2}(1+|m|)^{-\mu}e^{-\beta|m|}\exp\left(-\frac{\kappa}{2\log(q)}\log^2|\tau|\right)\int_{\rho_{\mathfrak{d}_{p},\mathfrak{d}_{p+1}}}^{\infty}r^{\frac{\kappa\log|\tau|}{\log(q)}+\alpha-3/2}dr,$$
which yields
\begin{equation}\label{e1058}
I^{\mathcal{L}}_1\le K_{p,2}^{\mathcal{L}}(1+|m|)^{-\mu}e^{-\beta|m|}\exp\left(-\frac{\kappa}{2\log(q)}\log^2|\tau|\right),
\end{equation}
for some $K_{p,2}^{\mathcal{L}}>0$
Analogous arguments allow us to obtain the existence of $K_{p,3}^{\mathcal{L}}>0$
such that
\begin{equation}\label{e1064}
\left|\frac{1}{\pi_{q^{1/\kappa}}}\int_{L_{\gamma_{p+1},\rho_{\mathfrak{d}_{p},\mathfrak{d}_{p+1}}}}\frac{w^{\mathfrak{d}_{p+1}}_{k_1}(u,m,\epsilon)}{\Theta_{q^{1/\kappa}}\left(\frac{u}{\tau}\right)}\frac{du}{u}
\right|\le K_{p,3}^{\mathcal{L}}(1+|m|)^{-\mu}e^{-\beta|m|}\exp\left(-\frac{\kappa}{2\log(q)}\log^2|\tau|\right).
\end{equation}

We write 
$$I^{\mathcal{L}}_{2}:=\left|\frac{1}{\pi_{q^{1/\kappa}}}\int_{C_{\rho_{\mathfrak{d}_{p},\mathfrak{d}_{p+1}},\gamma_p,\gamma_{p+1}}}\frac{w_{k_1}(u,m,\epsilon)}{\Theta_{q^{1/\kappa}}\left(\frac{u}{\tau}\right)}\frac{du}{u}\right|.$$
Regarding (\ref{e898}) and (\ref{e200}), one derives that
\begin{align*}
I^{\mathcal{L}}_{2}&\le\frac{C_{w^{\mathfrak{d}_{p}}_{k_1}}}{\pi_{q^{1/\kappa}}}\frac{e^{-\beta|m|}}{(1+|m|)^{\mu}}\frac{|\tau|^{1/2}}{\rho_{\mathfrak{d}_{p},\mathfrak{d}_{p+1}}^{1/2}C_{q,\kappa}\tilde{\delta}}\int_{\gamma_p}^{\gamma_{p+1}}\frac{\exp\left(\frac{\kappa\log^2|\rho_{\mathfrak{d}_{p},\mathfrak{d}_{p+1}}e^{i\theta}+\delta|}{2\log(q)}+\alpha\log|\rho_{\mathfrak{d}_{p},\mathfrak{d}_{p+1}}e^{i\theta}+\delta|\right)}{\exp\left(\frac{\kappa}{2}\frac{\log^2\left(\frac{\rho_{\mathfrak{d}_{p},\mathfrak{d}_{p+1}}}{|\tau|}\right)}{\log(q)}\right)}d\theta\\
&\le K_{p,4}^{\mathcal{L}}|\tau|^{1/2}\frac{e^{-\beta|m|}}{(1+|m|)^{\mu}}\exp\left(-\frac{\kappa}{2}\frac{\log^2\left(\frac{\rho_{\mathfrak{d}_{p},\mathfrak{d}_{p+1}}}{|\tau|}\right)}{\log(q)}\right)
\end{align*}
with 
$$K_{p,4}^{\mathcal{L}}=|\gamma_{p+1}-\gamma_{p}|\frac{C_{w^{\mathfrak{d}_{p}}_{k_1}}}{\pi_{q^{1/\kappa}}}\frac{1}{\rho_{\mathfrak{d}_{p},\mathfrak{d}_{p+1}}^{1/2}C_{q,\kappa}\tilde{\delta}}\exp\left(\frac{\kappa\log^2(\rho_{\mathfrak{d}_{p},\mathfrak{d}_{p+1}}+\delta)}{2\log(q)}+\alpha\log(\rho_{\mathfrak{d}_{p},\mathfrak{d}_{p+1}}+\delta)\right).$$
Let $K_{p,5}^{\mathcal{L}}=K_{p,4}^{\mathcal{L}}\exp(-\frac{\kappa}{2\log(q)}\log^2(\rho_{\mathfrak{d}_{p},\mathfrak{d}_{p+1}}))$. It is straight to check that
\begin{equation}\label{e1076}
I^{\mathcal{L}}_{2}\le K_{p,5}^{\mathcal{L}}|\tau|^{1/2+\frac{\kappa\log(\rho_{\mathfrak{d}_{p},\mathfrak{d}_{p+1}})}{\log(q)}}\frac{e^{-\beta|m|}}{(1+|m|)^{\mu}}\exp\left(-\frac{\kappa}{2}\frac{\log^2|\tau|}{\log(q)}\right).
\end{equation}

From (\ref{e1058}), (\ref{e1064}) and (\ref{e1076}), put into (\ref{e943b}), we conclude the result.
\end{proof}

\begin{prop}\label{prop1047}
Let $0\le p \le \varsigma -1$. Under the hypotheses of Theorem~\ref{teo872}, assume that $U_{\mathfrak{d}_{p}}\cap U_{\mathfrak{d}_{p+1}}=\emptyset$. Then, there exist $K_3>0$ and $K_4\in\R$ such that
\begin{align}
&|u^{\mathfrak{d}_{p+1}}(t,z,\epsilon)-u^{\mathfrak{d}_{p}}(t,z,\epsilon)|\le K_{3}\exp\left(-\frac{k_1}{2\log(q)}\log^2|\epsilon|\right)|\epsilon|^{K_4},\nonumber \\
&\hspace{3cm}|f^{\mathfrak{d}_{p+1}}(t,z,\epsilon)-f^{\mathfrak{d}_{p}}(t,z,\epsilon)|\le K_{3}\exp\left(-\frac{k_1}{2\log(q)}\log^2|\epsilon|\right)|\epsilon|^{K_4}.\label{e1049a}
\end{align}
for every $t\in\mathcal{T}$, $z\in H_{\beta'}$, and $\epsilon\in\mathcal{E}_{p}\cap\mathcal{E}_{p+1}$. 
\end{prop}
\begin{proof}
Let $0\le p\le \varsigma-1$. Under the assumptions of the enunciate, one can not proceed as in the proof of Proposition~\ref{prop925} for there does not exist a common function for both $p$ and $p+1$, defined in $\mathcal{R}^{b}_{\mathfrak{d}_{p}}\cup \mathcal{R}^{b}_{\mathfrak{d}_{p+1}}$ in the variable of integration, when applying q-Laplace transform. However, one can use the analytic continuation property (\ref{e894}) and write the difference $u^{\mathfrak{d}_{p+1}}-u^{\mathfrak{d}_{p}}$ as follows. Let $\rho_{\mathfrak{d}_{p},\mathfrak{d}_{p+1}}$ be such that $\rho_{\mathfrak{d}_{p},\mathfrak{d}_{p+1}}e^{i\gamma_{p}}\in\mathcal{R}^{b}_{\mathfrak{d}_{p}}$ and $\rho_{\mathfrak{d}_{p},\mathfrak{d}_{p+1}}e^{i\gamma_{p+1}}\in\mathcal{R}^{b}_{\mathfrak{d}_{p+1}}$, and let $\theta_{p,p+1}\in\R$ be such that $\rho_{\mathfrak{d}_{p},\mathfrak{d}_{p+1}}e^{\theta_{p,p+1}}$ lies in both $\mathcal{R}^{b}_{\mathfrak{d}_{p}}$ and  $\mathcal{R}^{b}_{\mathfrak{d}_{p+1}}$. We write
$$u^{\mathfrak{d}_{p+1}}(t,z,\epsilon)-u^{\mathfrak{d}_{p}}(t,z,\epsilon)$$
\begin{align}
&=\frac{1}{(2\pi)^{1/2}}\frac{1}{\pi_{q^{1/k_2}}}\int_{-\infty}^{\infty}\int_{L_{\gamma_{p+1},\rho_{\mathfrak{d}_{p},\mathfrak{d}_{p+1}}}}\frac{w^{\mathfrak{d}_{p+1}}_{k_2}(u,m,\epsilon)}{\Theta_{q^{1/k_2}}\left(\frac{u}{\epsilon t}\right)}\exp(izm)\frac{du}{u}dm,\hspace{3cm}\nonumber\\ 
&\hspace{2cm}-\frac{1}{(2\pi)^{1/2}}\frac{1}{\pi_{q^{1/k_2}}}\int_{-\infty}^{\infty}\int_{L_{\gamma_p,\rho_{\mathfrak{d}_{p},\mathfrak{d}_{p+1}}}}\frac{w^{\mathfrak{d}_{p}}_{k_2}(u,m,\epsilon)}{\Theta_{q^{1/k_2}}\left(\frac{u}{\epsilon t}\right)}\exp(izm)\frac{du}{u}dm\nonumber\\
&\hspace{4cm}-\frac{1}{(2\pi)^{1/2}}\frac{1}{\pi_{q^{1/k_2}}}\int_{-\infty}^{\infty}\int_{C_{\rho_{\mathfrak{d}_{p},\mathfrak{d}_{p+1}},\theta_{p,p+1},\gamma_{p+1}}}\frac{w^{\mathfrak{d}_{p},\mathfrak{d}_{p+1}}_{k_2}(u,m,\epsilon)}{\Theta_{q^{1/k_2}}\left(\frac{u}{\epsilon t}\right)}\exp(izm)\frac{du}{u}dm\nonumber\\
&\hspace{4cm}+\frac{1}{(2\pi)^{1/2}}\frac{1}{\pi_{q^{1/k_2}}}\int_{-\infty}^{\infty}\int_{C_{\rho_{\mathfrak{d}_{p},\mathfrak{d}_{p+1}},\theta_{p,p+1},\gamma_{p}}}\frac{w^{\mathfrak{d}_{p},\mathfrak{d}_{p+1}}_{k_2}(u,m,\epsilon)}{\Theta_{q^{1/k_2}}\left(\frac{u}{\epsilon t}\right)}\exp(izm)\frac{du}{u}dm\nonumber\\
&+\frac{1}{(2\pi)^{1/2}}\frac{1}{\pi_{q^{1/k_2}}}\int_{-\infty}^{\infty}\int_{L_{0,\rho_{\mathfrak{d}_{p},\mathfrak{d}_{p+1}},\theta_{p,p+1}}}\frac{\mathcal{L}_{q;1/\kappa}^{\mathfrak{d}_{p+1}}(w_{k_1}^{\mathfrak{d}_{p+1}})(\tau,m,\epsilon)-\mathcal{L}_{q;1/\kappa}^{\mathfrak{d}_{p}}(w_{k_1}^{\mathfrak{d}_{p}})(\tau,m,\epsilon)}{\Theta_{q^{1/k_2}}\left(\frac{u}{\epsilon t}\right)}\exp(izm)\frac{du}{u}dm.
\label{e1049}
\end{align}
Here, we have denoted $L_{\gamma_{j},\rho_{\mathfrak{d}_{p},\mathfrak{d}_{p+1}}}=[\rho_{\mathfrak{d}_{p},\mathfrak{d}_{j}},+\infty)e^{i\gamma_{j}}$ for $j\in\{p,p+1\}$, $C_{\rho_{\mathfrak{d}_{p},\mathfrak{d}_{p+1}},\theta_{p,p+1},\gamma_{p+1}}$ is the arc of circle connecting $\rho_{\mathfrak{d}_{p},\mathfrak{d}_{p+1}}e^{i\gamma_{p+1}}$ with $\rho_{\mathfrak{d}_{p},\mathfrak{d}_{p+1}}e^{i\theta_{p,p+1}}$, $C_{\rho_{\mathfrak{d}_{p},\mathfrak{d}_{p+1}},\theta_{p,p+1},\gamma_{p}}$ is the arc of circle connecting $\rho_{\mathfrak{d}_{p},\mathfrak{d}_{p+1}}e^{i\gamma_{p}}$ with $\rho_{\mathfrak{d}_{p},\mathfrak{d}_{p+1}}e^{i\theta_{p,p+1}}$, $L_{0,\rho_{\mathfrak{d}_{p},\mathfrak{d}_{p+1}},\theta_{p,p+1}}=[0,\rho_{\mathfrak{d}_{p},\mathfrak{d}_{p+1}}]e^{i\theta_{p,p+1}}$, as it is shown in Figure 2.

\begin{figure}[h]
	\centering
		\includegraphics[width=0.50\textwidth]{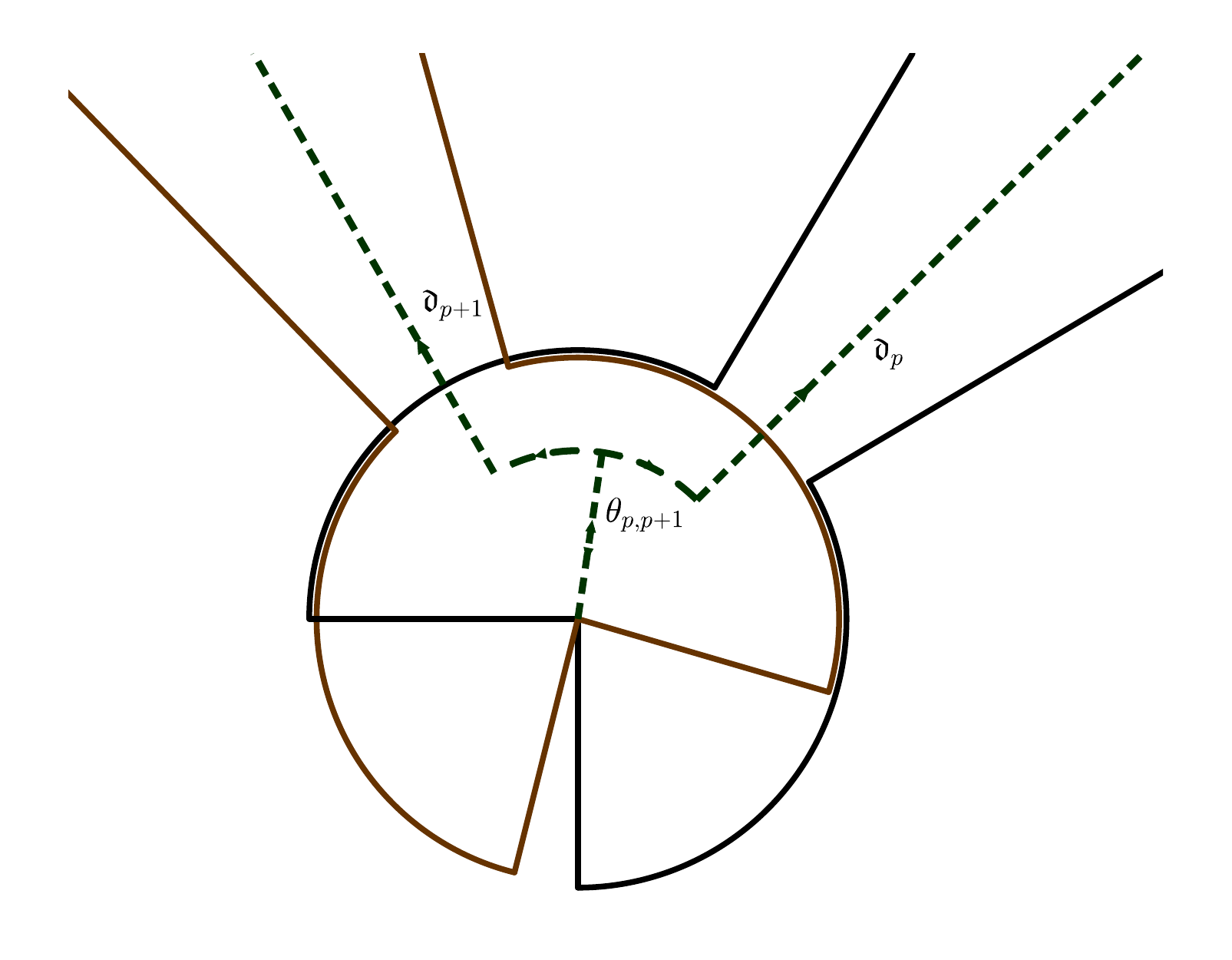}
	\label{fig:conf1}
	\caption{Deformation of the path of integration, second case.}
\end{figure}

Following the same line of arguments as those in the proof of Proposition~\ref{prop925}, we can guarantee the existence of $\hat{K}^{j}>0$ and $\hat{K}^{k}\in\R$ for $1\le j\le 4$ and $5\le k\le 8$ such that

\begin{align}
&J_{1}:=\left|\frac{1}{(2\pi)^{1/2}}\frac{1}{\pi_{q^{1/k_2}}}\int_{-\infty}^{\infty}\int_{L_{\gamma_{p+1},\rho_{\mathfrak{d}_{p},\mathfrak{d}_{p+1}}}}\frac{w^{\mathfrak{d}_{p+1}}_{k_2}(u,m,\epsilon)}{\Theta_{q^{1/k_2}}\left(\frac{u}{\epsilon t}\right)}\exp(izm)\frac{du}{u}dm\right|\nonumber\\
&\hspace{6cm}\le \hat{K}_{1}\exp\left(-\frac{k_2}{2\log(q)}\log^2|\epsilon|\right)|\epsilon|^{\hat{K}^{5}},\nonumber\\ 
&J_{2}:=\left|\frac{1}{(2\pi)^{1/2}}\frac{1}{\pi_{q^{1/k_2}}}\int_{-\infty}^{\infty}\int_{L_{\gamma_p,\rho_{\mathfrak{d}_{p},\mathfrak{d}_{p+1}}}}\frac{w^{\mathfrak{d}_{p}}_{k_2}(u,m,\epsilon)}{\Theta_{q^{1/k_2}}\left(\frac{u}{\epsilon t}\right)}\exp(izm)\frac{du}{u}dm\right|\nonumber\\
&\hspace{6cm}\le \hat{K}_{2}\exp\left(-\frac{k_2}{2\log(q)}\log^2|\epsilon|\right)|\epsilon|^{\hat{K}^{6}},\nonumber\\
&J_{3}:=\left|\frac{1}{(2\pi)^{1/2}}\frac{1}{\pi_{q^{1/k_2}}}\int_{-\infty}^{\infty}\int_{C_{\rho_{\mathfrak{d}_{p},\mathfrak{d}_{p+1}},\theta_{p,p+1},\gamma_{p+1}}}\frac{w^{\mathfrak{d}_{p},\mathfrak{d}_{p+1}}_{k_2}(u,m,\epsilon)}{\Theta_{q^{1/k_2}}\left(\frac{u}{\epsilon t}\right)}\exp(izm)\frac{du}{u}dm\right|\nonumber\\
&\hspace{6cm}\le \hat{K}_{3}\exp\left(-\frac{k_2}{2\log(q)}\log^2|\epsilon|\right)|\epsilon|^{\hat{K}^{7}},\nonumber\\
&J_{4}:=\left|\frac{1}{(2\pi)^{1/2}}\frac{1}{\pi_{q^{1/k_2}}}\int_{-\infty}^{\infty}\int_{C_{\rho_{\mathfrak{d}_{p},\mathfrak{d}_{p+1}},\theta_{p,p+1},\gamma_{p}}}\frac{w^{\mathfrak{d}_{p},\mathfrak{d}_{p+1}}_{k_2}(u,m,\epsilon)}{\Theta_{q^{1/k_2}}\left(\frac{u}{\epsilon t}\right)}\exp(izm)\frac{du}{u}dm\right|\nonumber\\
&\hspace{6cm}\le \hat{K}_{4}\exp\left(-\frac{k_2}{2\log(q)}\log^2|\epsilon|\right)|\epsilon|^{\hat{K}^{8}}\label{e1080}
\end{align}  

We now give estimates for
\begin{equation}\label{e1088}
J_{5}:=\left|\frac{1}{(2\pi)^{1/2}}\frac{1}{\pi_{q^{1/k_2}}}\int_{-\infty}^{\infty}\int_{L_{0,\rho_{\mathfrak{d}_{p},\mathfrak{d}_{p+1}},\theta_{p,p+1}}}\frac{\mathcal{L}_{q;1/\kappa}^{\mathfrak{d}_{p+1}}(w_{k_1}^{\mathfrak{d}_{p+1}})(u,m,\epsilon)-\mathcal{L}_{q;1/\kappa}^{\mathfrak{d}_{p}}(w_{k_1}^{\mathfrak{d}_{p}})(u,m,\epsilon)}{\Theta_{q^{1/k_2}}\left(\frac{u}{\epsilon t}\right)}\exp(izm)\frac{du}{u}dm\right|.
\end{equation}
In view of Lemma~\ref{lema1031} and (\ref{e200}), one has 
$$J_{5}\le\frac{K_p^{\mathcal{L}}}{(2\pi)^{1/2}}\frac{1}{\pi_{q^{1/k_2}}}\int_{-\infty}^{\infty}e^{-\beta|m|-\Im(z)m}\frac{dm}{(1+|m|)^{\mu}}\int_{0}^{\rho_{\mathfrak{d}_{p},\mathfrak{d}_{p+1}}}\frac{\exp\left(-\frac{\kappa}{2\log(q)}\log^2|u|\right)|u|^{M_p^{\mathcal{L}}}}{C_{q,k_2}\tilde{\delta}\exp\left(\frac{k_2}{2}\frac{\log^2\left|\frac{u}{\epsilon t}\right|}{\log(q)}\right)\left|\frac{u}{\epsilon t}\right|^{1/2}}\frac{d|u|}{|u|}.$$
We recall that $z\in H_{\beta'}$ for some $\beta'<\beta$. Then, there exists $K_{31}>0$ such that
$$J_{5}\le\frac{K_p^{\mathcal{L}}K_{31}}{(2\pi)^{1/2}}\frac{|\epsilon|^{1/2}r_{\mathcal{T}}^{1/2}}{\pi_{q^{1/k_2}}C_{q,k_2}\tilde{\delta}}\int_{0}^{\rho_{\mathfrak{d}_{p},\mathfrak{d}_{p+1}}}\frac{\exp\left(-\frac{\kappa}{2\log(q)}\log^2|u|\right)|u|^{M_p^{\mathcal{L}}}}{\exp\left(\frac{k_2}{2}\frac{\log^2\left|\frac{u}{\epsilon t}\right|}{\log(q)}\right)}\frac{d|u|}{|u|^{3/2}}.$$
%Let $K_{32}:=\frac{K_p^{\mathcal{L}}K_{31}}{(2\pi)^{1/2}}\frac{r_{\mathcal{T}}^{1/2}}{\pi_{q^{1/k_2}}C_{q,k_2}\tilde{\delta}}$.

We now proceed to prove the expression
$$\int_{0}^{\rho_{\mathfrak{d}_{p},\mathfrak{d}_{p+1}}}\frac{\exp\left(-\frac{\kappa}{2\log(q)}\log^2|u|\right)}{\exp\left(\frac{k_2}{2}\frac{\log^2\left|\frac{u}{\epsilon t}\right|}{\log(q)}\right)}\exp\left(\frac{k_1}{2\log(q)}\log^2|\epsilon|\right)\frac{d|u|}{|u|^{3/2-M_p^{\mathcal{L}}}}$$
is upper bounded by a positive constant times a certain power of $|\epsilon|$ for every $\epsilon\in(\mathcal{E}_{p}\cap \mathcal{E}_{p+1})$ and $t\in\mathcal{T}$. This concludes the existence of $K_{32}>0$ such that
\begin{equation}\label{e1109}J_{5}\le K_{32}|\epsilon|^{1/2}\exp\left(-\frac{k_1}{2\log(q)}\log^2|\epsilon|\right),
\end{equation}
for every $\epsilon\in(\mathcal{E}_{p}\cap \mathcal{E}_{p+1})$, $t\in\mathcal{T}$ and $z\in H_{\beta'}$.

Indeed, we have 
$$\int_{0}^{\rho_{\mathfrak{d}_{p},\mathfrak{d}_{p+1}}}\frac{\exp\left(-\frac{\kappa}{2\log(q)}\log^2|u|\right)}{\exp\left(\frac{k_2}{2}\frac{\log^2\left(\frac{|u|}{|\epsilon t|}\right)}{\log(q)}\right)}\exp\left(\frac{k_1}{2\log(q)}\log^2|\epsilon |\right)\frac{d|u|}{|u|^{3/2-M_p^{\mathcal{L}}}}$$
equals
\begin{equation}\label{e1111}
\exp\left(\frac{k_1}{2\log(q)}\log^2|\epsilon|-\frac{k_2}{2\log(q)}\log^2|\epsilon t|\right)\int_{0}^{\rho_{\mathfrak{d}_{p},\mathfrak{d}_{p+1}}}\exp\left(-\frac{(\kappa+k_2)}{2\log(q)}\log^2|u|\right)|u|^{\frac{k_2\log|\epsilon t|}{\log(q)}-\frac{3}{2}+M_p^{\mathcal{L}}}d|u|.
\end{equation}
Given $m_1\in\R$ and $m_2>0$, the function $[0,\infty)\ni x\mapsto H(x)=x^{m_1}\exp(-m_2\log^2(x))$ attains its maximum value at $x_0=\exp(\frac{m_1}{2m_2})$ with $H(x_0)=\exp(\frac{m_1^2}{4m_2})$. This yields and upper bound for the integrand in (\ref{e1111}); the expression in (\ref{e1111}) is estimated from above by
\begin{align}
&\rho_{\mathfrak{d}_{p},\mathfrak{d}_{p+1}}\exp\left(\frac{(M_{p}^{\mathcal{L}}-3/2)^2\log(q)}{2(\kappa+k_2)}\right) \exp\left(\frac{1}{2\log(q)}(\frac{k_2^2}{\kappa+k_2}-k_2+k_1)\log^2|\epsilon|\right)|\epsilon|^{\frac{k_2(M_{p}^{\mathcal{L}}-3/2)}{\kappa+k_2}}\nonumber\\
&\times\exp\left(\frac{1}{2\log(q)}(\frac{k_2^2}{\kappa+k_2}-k_2)\log^2|t|\right)|t|^{\frac{k_2(M_{p}^{\mathcal{L}}-3/2)}{\kappa+k_2}}\nonumber\\
&\times \exp\left(\frac{1}{\log(q)}(\frac{k_2^2}{\kappa+k_2}-k_2)\log|\epsilon|\log|t|\right).\label{e1156}
\end{align}
The second line in (\ref{e1156}) is upper bounded for every $t$ because $\frac{k_2^2}{\kappa+k_2}<k_2$ and also, one has an upper bound for the third line in (\ref{e1156}) is 1. Regarding (\ref{e824}), and taking into account that
$$\frac{k_2^2}{\kappa+k_2}-k_2=-k_1,$$
the expression (\ref{e1156}) is upper bounded by
$$K_{33}|\epsilon|^{\frac{k_2(M_{p}^{\mathcal{L}}-3/2)}{\kappa+k_2}}$$
for some $K_{33}>0$. The conclusion is achieved.

The result follows from (\ref{e1049}), (\ref{e1080}) and (\ref{e1109}).

The proof for the estimates of $f^{\mathfrak{d}_{p}}$ is analogous as that for $u^{\mathfrak{d}_{p}}$. In this case, one has to take into account (\ref{e859}), (\ref{e868}) and (\ref{e876}) to apply the same arguments as above.

\end{proof}

\textbf{Example:} Let $A>0$. An example of problem under study in this work is given by the equation
\begin{align*}
\partial_z(\partial_z+iA)^2u(qt,z,\epsilon)&=(\epsilon t)^3\partial_zu(q^{5/2}t,z,\epsilon)+t\epsilon c_{1,1}(z,\epsilon)u(qt,z,\epsilon)\\
&+t^2\epsilon^3 c_{1,2}(z,\epsilon)\partial_{z}u(q^2t,z,\epsilon)+f(qt,z,\epsilon),
\end{align*}

with $c_{1,1},c_{1,2}$ and $f$ constructed following the procedure described at the beginning of Section~\ref{seccion5}.
 
Here, $D=3$, $R_{1}=1$, $R_{2}=R_{3}(x)=x$ and $Q(x)=(x+iA)^2R_{3}(x)$.
Every condition on the constants are satisfied. Also, we observe that one can choose large enough $A>0$ in order to choose large enough $r_{Q,r_{D}}>0$, with the only forbidden direction given by the negative real ray. 

\section{Existence of formal series solutions in the complex parameter and asymptotic expansion in two levels}\label{seccion6}

In the first part of this section, we develop a two-level q-analog of Ramis-Sibuya theorem. This result provides the tool to guarantee the existence of a formal power series in the perturbation parameter which formally solves the main problem and such that it asymptotically represents the analytic solution of that equation.

This asymptotic representation is held in the sense of $1-$asymptotic expansions of certain positive order.

\begin{defin}
Let $V$ be a bounded open sector with vertex at 0 in $\C$. Let $(\mathbb{F},\left\|\cdot\right\|_{\mathbb{F}})$ be a complex Banach space. Let $q\in\R$ with $q>1$ and let $k$ be a positive integer. We say that a holomorphic function $f:V\to\mathbb{F}$ admits the formal power series $\hat{f}(\epsilon)=\sum_{n\ge0}f_n\epsilon^n\in\mathbb{F}[[\epsilon]]$ as its $q-$Gevrey asymptotic expansion of order $1/k$ if for every open subsector $U$ with $(\overline{U}\setminus\{0\})\subseteq V$, there exist $A,C>0$ such that
$$\left\|f(\epsilon)-\sum_{n=0}^{N}f_n\epsilon^n\right\|_{\mathbb{F}}\le CA^{N+1}q^{\frac{N(N+1)}{2k}}|\epsilon|^{N+1},$$
for every $\epsilon\in U$, and $N\ge0$.
\end{defin}

The set of functions which admit null q-Gevrey asymptotic expansion of certain positive order are characterized as follows. The proof of this result, already stated in~\cite{ma15}, provides the $q-$analog of Theorem XI-3-2 in~\cite{hssi}.

\begin{lemma}
A holomorphic function $f:V\to\mathbb{F}$ admits the null formal power series $\hat{0}\in\mathbb{F}[[\epsilon]]$ as its $q-$Gevrey asymptotic expansion of order $1/k$ if and only if for every open subsector $U$ with $(\overline{U}\setminus\{0\})\subseteq V$ there exist constants $K_1\in\R$ and $K_2>0$ with
$$\left\|f(\epsilon)\right\|_{\mathbb{F}}\le K_2\exp\left(-\frac{k}{2\log(q)}\log^2|\epsilon|\right)|\epsilon|^{K_1},$$
for all $\epsilon\in U$.
\end{lemma}

We recall the one-level version of the $q-$analog of Ramis-Sibuya theorem proved in~\cite{ma15}.

\begin{theo}($q-$RS)\label{teoqrs}
Let $(\mathbb{F},\left\|\cdot\right\|_{\mathbb{F}})$ be a Banach space and $(\mathcal{E}_{p})_{0\le p\le \varsigma-1}$ be a good covering in $\C^{\star}$. For every $0\le p\le \varsigma-1$, let $G_{p}(\epsilon)$ be a holomorphic function from $\mathcal{E}_{p}$ into $\mathbb{F}$ and let the cocycle $\Delta_{p}(\epsilon)=G_{p+1}(\epsilon)-G_{p}(\epsilon)$ be a holomorphic function from $Z_{p}=\mathcal{E}_{p}\cap\mathcal{E}_{p+1}$ into $\mathbb{F}$ (we put $\mathcal{E}_{\varsigma}=\mathcal{E}_{0}$ and $G_{\varsigma}=G_{0}$). We also make the further assumptions:
\begin{enumerate}
\item[\textbf{1)}] The functions $G_p(\epsilon)$ are bounded as $\epsilon$ tends to 0 on $\mathcal{E}_{p}$ for every $0\le p\le \varsigma -1$.
\item[\textbf{2)}] For all $0\le p\le \varsigma-1$, the function $\Delta_{p}(\epsilon)$ is $q-$exponentially flat of order $k$ on $Z_p$, i.e. there exist constants $C_{p}^1\in\R$ and $C_p^2>0$ such that
$$\left\|\Delta_p(\epsilon)\right\|_{\mathbb{F}}\le C_p^2|\epsilon|^{C_p^1}\exp\left(-\frac{k}{2\log(q)}\log^2|\epsilon|\right),$$
for every $\epsilon\in Z_p$, all $0\le p\le \varsigma-1$.
\end{enumerate}
Then, there exists a formal power series $\hat{G}(\epsilon)\in\mathbb{F}[[\epsilon]]$ which is the common $q-$Gevrey asymptotic expansion of order $1/k$ of the functions $G_p(\epsilon)$ on $\mathcal{E}_{p}$, which is common for all $0\le p\le \varsigma-1$.
\end{theo}

The next result leans on the one level version of the $q-$analog of Ramis Sibuya theorem, and states a two level result in this framework. It is straight to generalize this result to a higher number of levels, but for practical purposes, we develop it just in two. 

\begin{theo}\label{teo1215}
Let $(\mathbb{F},\left\|\cdot\right\|_{\mathbb{F}})$ be a Banach space and $(\mathcal{E}_{p})_{0\le p\le \varsigma-1}$ be a good covering in $\C^{\star}$. Let $0<k_1<k_2$, consider a holomorphic function $G_{p}:\mathcal{E}_{i}\to\mathbb{F}$ for every $0\le p\le \varsigma-1$ and put $\Delta_{p}(\epsilon)=G_{p+1}(\epsilon)-G_{p}( \epsilon)$ for every $\epsilon\in Z_{p}:=\mathcal{E}_{p}\cap\mathcal{E}_{p+1}$. Moreover, we assume:
\begin{enumerate}
\item[\textbf{1)}] The functions $G_p(\epsilon)$ are bounded as $\epsilon$ tends to 0 on $\mathcal{E}_{p}$ for every $0\le p\le \varsigma -1$.
\item[\textbf{2)}] There exist nonempty sets $I_1,I_2\subseteq\{0,1,...,\varsigma-1\}$ such that $I_1\cup I_2=\{0,1,...,\varsigma-1\}$ and $I_1\cap I_2=\emptyset$. Also,
\begin{itemize}
\item[-] for every $p\in I_1$ there exist constants $K_1>0$, $M_1\in\R$ such that
$$\left\|\Delta_p(\epsilon)\right\|_{\mathbb{F}}\le K_1|\epsilon|^{M_1}\exp\left(-\frac{k_1}{2\log(q)}\log^2|\epsilon|\right),\quad \epsilon\in Z_p,$$
\item[-] and, for every $p\in I_2$ there exist constants $K_2>0$, $M_2\in\R$ such that
$$\left\|\Delta_p(\epsilon)\right\|_{\mathbb{F}}\le K_2|\epsilon|^{M_2}\exp\left(-\frac{k_2}{2\log(q)}\log^2|\epsilon|\right),\quad \epsilon\in Z_p.$$
\end{itemize}
\end{enumerate}
Then, there exists a convergent power series $a(\epsilon)\in\mathbb{F}\{\epsilon\}$ defined on some neighborhood of the origin and $\hat{G}^1(\epsilon),\hat{G}^2(\epsilon)\in\mathbb{F}[[\epsilon]]$ such that $G_p$ can be written in the form
$$G_p(\epsilon)=a(\epsilon)+G_{p}^{1}(\epsilon)+G_{p}^{2}(\epsilon).$$
$G_{p}^1(\epsilon)$ is holomorphic on $\mathcal{E}_{p}$ and admits $\hat{G}^{1}(\epsilon)$ as its $q-$Gevrey asymptotic expansion of order $1/k_1$ on $\mathcal{E}_{p}$, for every $p\in I_1$; whilst $G_{p}^2(\epsilon)$ is holomorphic on $\mathcal{E}_{p}$ and admits $\hat{G}^{2}(\epsilon)$ as its $q-$Gevrey asymptotic expansion of order $1/k_2$ on $\mathcal{E}_{p}$, for every $p\in I_2$.
\end{theo}
\begin{proof}
For every $0\le i\le \varsigma$, we define the functions $\Delta_{i}^j(\epsilon)\in\mathcal{O}(Z_i)$ for $j=1,2$, by
$$\Delta_{i}^{j}(\epsilon)=\left\{ 
\begin{array}{cc}
\Delta_{i}(\epsilon) &\hbox{ if } i\in I_j\\
0 & \hbox{ if } i\in\{0,1,...,\varsigma-1\}\setminus I_{j}
\end{array}\right.
$$
for $\epsilon\in Z_j$. 

As an introductory lemma, we provide the following result without proof which can be found in Lemma 8 of~\cite{ma15}, and it rests on the arguments of Lemma XI-2-6 from~\cite{hssi}.

\begin{lemma}\label{lema1226}
Under the assumptions of Theorem~\ref{teoqrs}, for every $0\le i\le \varsigma-1$ and $j=1,2$, there exist bounded holomorphic functions $\Psi_{i}^j:\mathcal{E}_{i}\to\C$ such that
$$\Delta_{i}^j(\epsilon)=\Psi_{i+1}^j(\epsilon)-\Psi_{i}^j(\epsilon),$$
for every $\epsilon\in Z_i$ ($\Psi_\varsigma^j(\epsilon):=\Psi_0^j(\epsilon)$). Moreover, there exist $\varphi_m^j\in\mathbb{F}$, for every $m\ge0$, such that for every $0\le i\le \varsigma-1$ and any closed proper subsector $\mathcal{W}\subseteq \mathcal{E}_{p}$, with vertex at 0, there exist $\hat{K}_{p},\hat{M}_{p}>0$ with
$$ \left\|\Psi_{i}^j(\epsilon)- \sum_{m=0}^{M}\varphi^j_m\epsilon^m \right\|_{\mathbb{F}}\le \hat{K}_{p}(\hat{M}_{p})^{M+1}q^{\frac{(M+1)M}{2k_j}}|\epsilon|^{M+1}, $$
for every $\epsilon\in\mathcal{W}$, and $M\ge0$.
\end{lemma}

We now consider the bounded holomorphic functions $a_i(\epsilon)=G_i(\epsilon)-\Psi_{i}^1(\epsilon)-\Psi_{i}^2(\epsilon)$, for all $0\le i\le \varsigma-1$, and $\epsilon\in\mathcal{E}_i$. By definition, for $j=1,2$ and $i\in I_j$ we have
$$a_{i+1}(\epsilon)-a_i(\epsilon)=G_{i+1}(\epsilon)-\Delta_{i}^1(\epsilon)-\Delta_{i}^2(\epsilon)=G_{i+1}(\epsilon)-G_i(\epsilon)-\Delta_i(\epsilon)=0,$$
for $\epsilon\in Z_i$. Therefore, each $a_i(\epsilon)$ is the restriction on $\mathcal{E}_i$ of a holomorphic function $a(\epsilon)$, defined on a neighborhood of the origin but zero. Indeed, $a(\epsilon)$ is bounded on $\cup_{0\le i\le \varsigma-1}\mathcal{E}_i$ so the origin turns out to be a removable singularity and, as a consequence, $a(\epsilon)$ defines a convergent power series on the neighborhood of the origin $\left(\cup_{0\le i\le \varsigma-1}\mathcal{E}_i\right)\cup\{0\}$.

One can finish the proof by rewriting
$$G_i(\epsilon)=a(\epsilon)+\Psi_{i}^1(\epsilon)+\Psi_i^2(\epsilon),$$
and bearing in mind Lemma~\ref{lema1226}. 
\end{proof}

We conclude this section with the main result in the work in which we guarantee the existence of a formal solution of the main problem (\ref{e771}), written as a formal power series in the perturbation parameter, with coefficients in an appropriate Banach space, say $\hat{u}(t,z,\epsilon)$. Moreover, it represents, in some sense to be precised, each solution $u^{\mathfrak{d}_{p}}(t,z,\epsilon)$ of the problem (\ref{e771}). 

This result is based on the existence of a common formal power series $\hat{f}(t,z,\epsilon)$ which is the $q-$Gevrey asymptotic expansion of order $1/k_1$, seen as a formal power series in the perturbation parameter $\epsilon$ with coefficients in a certain Banach space, of every $f^{\mathfrak{d}_{p}}$ on $\mathcal{E}_{p}$.

From now on, $\mathbb{F}$ stands for the Banach space of bounded holomorphic functions defined on $\mathcal{T}\times H_{\beta'}$, with the supremum norm,where $\beta'<\beta$, as above.

\begin{lemma}\label{lema1271}
Under the hypotheses on Theorem~\ref{teo872}, there exists a formal power series 
$$\hat{f}(t,z,\epsilon)=\sum_{m\ge0}f_m(t,z)\frac{\epsilon^m}{m!},$$
with $f_m(t,z)\in\mathbb{F}$ for $m\ge0$, which is the common $q-$Gevrey asymptotic expansion of order $1/k_1$ on $\mathcal{E}_{p}$ of the functions $f^{\mathfrak{d}_{p}}$, seen as holomorphic functions from $\mathcal{E}_{p}$ to $\mathbb{F}$, for all $0\le p\le \varsigma-1$.
\end{lemma}
\begin{proof}
Let $0\le p\le \varsigma-1$. We consider the function $f^{\mathfrak{d}_{p}}$ constructed in (\ref{e876}), and define $G^{f}_p(\epsilon):=(t,z)\mapsto f^{\mathfrak{d}_{p}}(t,z,\epsilon)$, which is a holomorphic and bounded function from $\mathcal{E}_{p}$ into $\mathbb{F}$. Regarding (\ref{e927}) and (\ref{e1049a}), and taking into account that $k_1<k_2$, we have that (\ref{e1049a}) holds for every $0\le p\le \varsigma-1$. This yields the cocycle $\Delta_p^f(\epsilon):=G_{p+1}^f(\epsilon)-G_{p}^f(\epsilon)$ satisfies the conditions of Theorem~\ref{teoqrs} for $k=k_1$, and one concludes the result by the application of Theorem~\ref{teoqrs}.
\end{proof}

\begin{theo}\label{teo1281}
Under the hypotheses of Theorem~\ref{teo872},  there exists a formal power series 
$$\hat{u}(t,z,\epsilon)=\sum_{m\ge0}h_m(t,z)\frac{\epsilon^m}{m!}\in\mathbb{F}[[\epsilon]],$$
formal solution of the equation
\begin{align}
&Q(\partial_z)\sigma_q\hat{u}(t,z,\epsilon)=(\epsilon t)^{d_{D}}\sigma_{q}^{\frac{d_{D}}{k_2}+1}R_{D}(\partial_{z})\hat{u}(t,z,\epsilon)\nonumber\\
&\hspace{3cm}+\sum_{\ell=1}^{D-1}\left(\sum_{\lambda\in I_{\ell}}t^{d_{\lambda,\ell}}\epsilon^{\Delta_{\lambda,\ell}}\sigma_{q}^{\delta_\ell}c_{\lambda,\ell}(z,\epsilon)R_{\ell}(\partial_z)\hat{u}(t,z,\epsilon)\right)+\sigma_{q}\hat{f}(t,z,\epsilon).\label{e771b}
\end{align}
Moreover, $\hat{u}(t,z,\epsilon)$ turns out to be the common $q-$Gevrey asymptotic expansion of order $1/k_1$ on $\mathcal{E}_{p}$ of the function $u^{\mathfrak{d}_{p}}$, seen as holomorphic function from $\mathcal{E}_{p}$ into $\mathbb{F}$, for $0\le p\le \varsigma-1$. In addition to that, $\hat{u}$ is of the form
$$\hat{u}(t,z,\epsilon)=a(t,z,\epsilon)+\hat{u}_1(t,z,\epsilon)+\hat{u}_2(t,z,\epsilon),$$
where $a(t,z,\epsilon)\in\mathbb{F}\{\epsilon\}$ and $\hat{u}_1(t,z,\epsilon),\hat{u}_2(t,z,\epsilon)\in\mathbb{F}[[\epsilon]]$ and such that for every $0\le p\le \varsigma-1$, the function $u^{\mathfrak{d}_p}$ can be written in the form
$$u^{\mathfrak{d}_{p}}(t,z,\epsilon)=a(t,z,\epsilon)+u^{\mathfrak{d}_{p}}_1(t,z,\epsilon)+u^{\mathfrak{d}_{p}}_2(t,z,\epsilon),$$
where $\epsilon\mapsto u_1^{\mathfrak{d}_{p}}(t,z,\epsilon)$ is a $\mathbb{F}$-valued function that admits $\hat{u}_1(t,z,\epsilon)$ as its $q-$Gevrey asymptotic expansion of order $1/k_1$ on $\mathcal{E}_{p}$ and also $\epsilon\mapsto u_2^{\mathfrak{d}_{p}}(t,z,\epsilon)$ is a $\mathbb{F}$-valued function that admits $\hat{u}_2(t,z,\epsilon)$ as its $q-$Gevrey asymptotic expansion of order $1/k_2$ on $\mathcal{E}_{p}$.
\end{theo}
\begin{proof}
For every $0\le p\le \varsigma -1$, one can consider the function $u^{\mathfrak{d}_{p}}(t,z,\epsilon)$ constructed in Theorem~\ref{teo872}. We define $G_{p}(\epsilon):=(t,z)\mapsto u_p(t,z,\epsilon)$, which is a holomorphic and bounded function from $\mathcal{E}_{p}$ into $\mathbb{F}$. In view of Proposition~\ref{prop925} and Proposition~\ref{prop1047}, one can split the set $\{0,1,...,\varsigma-1\}$ in two nonempty subsets of indices, $I_1$ and $I_2$ with $\{0,1,...,\varsigma-1\}=I_1\cup I_2$ and such that $I_1$ (resp. $I_2$) consists of all the elements in $\{0,1,...,\varsigma-1\}$ such that $U_{\mathfrak{d}_{p}}\cap U_{\mathfrak{d}_{p+1}}$ contains the sector $U_{\mathfrak{d}_{p},\mathfrak{d}_{p+1}}$, as defined in Proposition~\ref{prop925} (resp. $U_{\mathfrak{d}_{p}}\cap U_{\mathfrak{d}_{p+1}}=\emptyset$). From (\ref{e927}) and (\ref{e1049a}) one can apply Theorem~\ref{teo1215} and deduce the existence of formal power series $\hat{G}^1(\epsilon),\hat{G}^2(\epsilon)\in\mathbb{F}[[\epsilon]]$, a convergent power series $a(\epsilon)\in\mathbb{F}\{\epsilon\}$ and holomorphic functions $G_{p}^1(\epsilon),G_{p}^2(\epsilon)$ defined on $\mathcal{E}_{p}$ and with values in $\mathbb{F}$ such that 
$$G_p(\epsilon)=a(\epsilon)+G^1_p(\epsilon)+G^2_p(\epsilon),$$
and for $j=1,2$, one has $G_{p}^j(\epsilon)$ admits $\hat{G}^j(\epsilon)$ as its $q-$Gevrey asymptotic expansion or order $1/k_j$ on $\mathcal{E}_{p}$. 
We put 
$$\hat{u}(t,z,\epsilon)=\sum_{m\ge0}h_m(t,z)\frac{\epsilon^m}{m!}:=a(\epsilon)+\hat{G}^1_p(\epsilon)+\hat{G}^2_p(\epsilon).$$

It only rests to prove that $\hat{u}(t,z,\epsilon)$ is the solution of (\ref{e771b}). Indeed, since $u^{\mathfrak{d}_{p}}$ (resp. $f^{\mathfrak{d}_{p}}$) admits $\hat{u}(t,z,\epsilon)$ (resp. $\hat{f}$) as its $q-$Gevrey asymptotic expansion of order $1/k_1$ on $\mathcal{E}_{p}$, we have that
\begin{align}
&\lim_{\epsilon\to 0,\epsilon\in\mathcal{E}_{p}}\sup_{t\in\mathcal{T},z\in H_{\beta'}}|\partial_\epsilon^m u^{\mathfrak{d}_{p}}(t,z,\epsilon)- h_m(t,z)|=0,\nonumber \\
&\hspace{3cm} \lim_{\epsilon\to 0,\epsilon\in\mathcal{E}_{p}}\sup_{t\in\mathcal{T},z\in H_{\beta'}}|\partial_\epsilon^m f^{\mathfrak{d}_{p}}(t,z,\epsilon)- f_m(t,z)|=0,\label{e1300}
\end{align}
for every $0\le p\le \varsigma -1$ and $m\ge0$. Let $p\in\{0,1,...,\varsigma-1\}$. By construction, the function $u^{\mathfrak{d}_{p}}(t,z,\epsilon)$ solves equation (\ref{e771}). We take derivatives of order $m\ge 0$ with respect to $\epsilon$ at both sides of equation (\ref{e771}) and deduce that $\partial_\epsilon^mu^{\mathfrak{d}_{p}}(t,z,\epsilon)$ satisfies
\begin{align}
&Q(\partial_z)\sigma_q(\partial_\epsilon^m u^{\mathfrak{d}_{p}})(t,z,\epsilon)=\sum_{m_1+m_2=m}\frac{m!}{m_1!m_2!}\partial_\epsilon^{m_1}(\epsilon^{d_{D}})t^{d_{D}}\sigma_{q}^{\frac{d_{D}}{k_2}+1}R_{D}(\partial_{z})(\partial_{\epsilon}^m u^{\mathfrak{d}_{p}})\nonumber\\
&+\sum_{\ell=1}^{D-1}\sum_{\lambda\in I_{\ell}}\sum_{m_1+m_2+m_3=m}\frac{m!}{m_1!m_2!m_3!}(\partial_{\epsilon}^{m_1}\epsilon^{\Delta_{\lambda,\ell}})t^{d_{\lambda,\ell}}(\partial_{\epsilon}^{m_2}c_{\lambda,\ell}(z,\epsilon))R_{\ell}(\partial_z)\sigma_{q}^{\delta_\ell}(\partial_\epsilon^m u^{\mathfrak{d}_{p}})\nonumber\\
&\hspace{6cm}+\sigma_{q}(\partial_\epsilon^m f^{\mathfrak{d}_{p}})(t,z,\epsilon),\label{e771c}
\end{align}
for every $(t,z,\epsilon)\in\mathcal{T}\times H_{\beta'}\times\mathcal{E}_{p}$. We let $\epsilon\to0$ in (\ref{e771c}) and obtain the recursion formula
\begin{align}
&Q(\partial_z)\sigma_qh_{m}(t,z)=\frac{m!}{(m-d_{D})!}t^{d_{D}}\sigma_{q}^{\frac{d_{D}}{k_2}+1}R_{D}h_{m-d_{D}}(t,z)\nonumber\\
&+\sum_{\ell=1}^{D-1}\sum_{\lambda\in I_{\ell}}\sum_{m_2+m_3=m-\Delta_{\lambda,\ell}}\frac{m!}{m_2!m_3!}t^{d_{\lambda,\ell}}(\partial_{\epsilon}^{m_2}c_{\lambda,\ell}(z,0))R_{\ell}(\partial_z)\sigma_{q}^{\delta_\ell}h_{m_3}(t,z)+\sigma_{q}f_{m}(t,z),\label{e771d}
\end{align}
for every $m\ge \max\{d_{D},\max_{1\le \ell\le D-1,\lambda\in I_{\ell}}\Delta_{\lambda,\ell}\}$, and all $(t,z)\in\mathcal{T}\times H_{\beta'}$. Bearing in mind that $c_{\lambda,\ell}$ (constructed in (\ref{e855b})) is holomorphic with respect to $\epsilon$ in a neighborhood of the origin, in such neighborhood one has
\begin{equation}\label{e1316}
c_{\lambda,\ell}(z,\epsilon)=\sum_{m\ge0}\frac{(\partial_\epsilon^m c_\ell)(z,0)}{m!}\epsilon^m,
\end{equation}
for every $1\le \ell\le D-1$ and $\lambda\in I_{\ell}$. By direct calculations on the recursion (\ref{e771d}) and (\ref{e1316}) one concludes that the formal power series $\hat{u}(t,z,\epsilon)=\sum_{m\ge0}h_m(t,z)\epsilon^m/m!$ is a solution of equation (\ref{e771b}).  

\end{proof}

\section{Application}\label{seccion7}

Let $\textbf{D}\ge3$ and $k_1\ge1$ be integers. Let $q\in\R$ with $q>1$ and assume that for every $1\le \ell\le \textbf{D}-1$, $\textbf{I}_{\ell}$ is a finite nonempty subset of nonnegative integers.

Let $\textbf{d}_{\textbf{D}}\ge1$ be an integer. For every $1\le \ell\le \textbf{D}-1$, we consider $\boldsymbol{\delta}\ge1$ and for each $\lambda\in\textbf{I}_{\ell}$, we take integers $\textbf{d}_{\lambda,\ell}\ge1$, $\boldsymbol{\Delta}_{\lambda,\ell}\ge0$. We assume that 
$$\quad\boldsymbol{\delta}_1=1,\quad \boldsymbol{\delta}_{\ell}<\boldsymbol{\delta}_{\ell+1}\quad ,$$
for every $1\le\ell\le \textbf{D}-1$ and all $\lambda\in \textbf{I}_{\ell}$, and also
\begin{equation}\label{e1336}
\boldsymbol{\Delta}_{\lambda,\ell}\ge \textbf{d}_{\lambda,\ell}, \quad \frac{\textbf{d}_{\lambda,\ell}}{k_1}+1\ge \boldsymbol{\delta}_{\ell}, \quad \frac{\textbf{d}_{\textbf{D}}-1}{k_1}+1\ge\boldsymbol{\delta}_{\ell},
\end{equation}
for every $1\le \ell\le \textbf{D}-1$, and all $\lambda\in \textbf{I}_{\ell}$.

We consider $\textbf{Q},\textbf{R}_{\ell}\in\C[X]$ for all $1\le \ell\le \textbf{D}$ with  
$$\deg(\textbf{Q})\ge\deg(\textbf{R}_{\textbf{D}})\ge \deg(\textbf{R}_{\ell}), \quad \textbf{Q}(im)\neq0,\quad \textbf{R}_{\textbf{D}}(im)\neq 0$$
for every $m\in\R$ and all $0\le \ell\le \textbf{D}-1$. We take $\beta>0$ and $\mu>\deg(\textbf{Q})+1$.

For every $1\le \ell\le \textbf{D}-1$ and all $\lambda\in \textbf{I}_{\ell}$, the function $m\mapsto \textbf{C}_{\lambda,\ell}(m,\epsilon)$ belongs to the space $E_{(\beta,\mu)}$ for some $\beta>0$ and $\mu>\deg(\textbf{R}_{\textbf{D}})+1$. In addition to that, we assume this functions depend holomorphically on $\epsilon\in D(0,\epsilon_0)$, and also the existence of apositive constant $\tilde{\textbf{C}}_{\lambda,\ell}$ such that 
\begin{equation}\label{e223b}
\left\|\textbf{C}_{\lambda,\ell}(m,\epsilon)\right\|_{(\beta,\mu)}\le \tilde{\textbf{C}}_{\lambda,\ell},
\end{equation}
for every $\epsilon\in D(0,\epsilon_0)$.

Let $N_0\ge0$ be an integer. We choose $\textbf{F}(T,m,\epsilon)=\sum_{n=0}^{N_0}\textbf{F}_{n}(m,\epsilon)T^n\in E_{(\beta,\mu)}[[T]]$ which depend holomorphically on $\epsilon\in D(0,\epsilon_0)$. We assume there exist positive constants $\mathbf{K}_0,\mathbf{T}_0$ such that 
$$\left\|\textbf{F}_{n}(m,\epsilon)\right\|_{(\beta,\mu)}\le \mathbf{K}_0\frac{1}{\mathbf{T}_0^{n}},$$
for every $0\le n\le N_0$ and $\epsilon\in D(0,\epsilon_0)$. Under this last condition, $\hat{\mathcal{B}}_{q,1/k_1}(\textbf{F}(T,m,\epsilon))(\tau)$ is an element of $\hbox{Exp}^{\mathfrak{d}_{p}}_{(k_1,\beta,\mu,\alpha,\rho)}$, for some $\alpha\in\R$ and $\rho>0$, and every $p$, which will be denoted by $\boldsymbol{\psi}_{k_1}^{\mathfrak{d}_{p}}(\tau,m,\epsilon)$. Following the same arguments as in Lemma 5 in~\cite{ma15}, one can check that
\begin{equation}\label{e268b}
|\boldsymbol{\psi}_{k_1}^{\mathfrak{d}_{p}}(\tau,m,\epsilon)|\le \textbf{C}_{\boldsymbol{\psi}_{k_1}}\frac{e^{-\beta|m|}}{(1+|m|)^{\mu}}\exp\left(\frac{k_1\log^{2}|\tau+\delta|}{2\log(q)}+\alpha\log|\tau+\delta|\right),
\end{equation}
for some $\textbf{C}_{\boldsymbol{\psi}_{k_1}}>0$, valid for all $\epsilon\in D(0,\epsilon_0)$, $\tau\in U_{\mathfrak{d}_{p}}\cup D(0,\rho)$ and $m\in\R$. 

We consider $k_2>k_1$ as chosen in Section~\ref{seccion4}, and $\kappa$ determined by (\ref{e236}). From the fact that the $q-$Laplace transform of some order provides an extension of the inverse operator for the formal $q-$Borel transform of the same order when defined on monomials, it is direct to check that 
\begin{equation}\label{e1458}
\sum_{n=0}^{N_0}\textbf{F}_n(m,\epsilon)T^n=\mathcal{L}_{q;1/k_2}\mathcal{L}_{q;1/\kappa}\boldsymbol{\psi}_{k_1}^{\mathfrak{d}_{p}}(\tau,m,\epsilon),
\end{equation}
for every $\epsilon\in D(0,\epsilon_0)$, $m\in\R$. 

We depart from the formal power series $F(T,m,\epsilon)=\sum_{n\ge0}F_n(m,\epsilon)T^n$, where $F_n(m,\epsilon)$ are defined after (\ref{e771}), in Section~\ref{seccion5}, and assume that this formal power series formally solves the equation

\begin{align}&\textbf{Q}(im)\sigma_{q}F(T,m,\epsilon)=T^{\textbf{d}_{\textbf{D}}}\sigma_{q}^{\frac{\textbf{d}_{\textbf{D}}}{k_1}+1}\textbf{R}_{\textbf{D}}(im)F(T,m,\epsilon)\nonumber\\
+&\sum_{\ell=1}^{\textbf{D}-1}\left(\sum_{\lambda\in \textbf{I}_{\ell}}T^{\textbf{d}_{\lambda,\ell}}\epsilon^{\boldsymbol{\Delta}_{\lambda,\ell}-\textbf{d}_{\lambda,\ell}}\frac{1}{(2\pi)^{1/2}}\int_{-\infty}^{\infty}\textbf{C}_{\lambda,\ell}(m-m_1,\epsilon)\textbf{R}_{\ell}(im_1)F(q^{\boldsymbol{\delta}_{\ell}}T,m_1,\epsilon)dm_1\right)\nonumber\\
&\hspace{6cm}+\sigma_{q}\textbf{F}(T,m,\epsilon).\label{e149b}
\end{align}
\begin{prop}
Assume that $F_{p}(m,\epsilon)$ are fixed elements in $E_{(\beta,\mu)}$ which depend holomorphically on $\epsilon\in D(0,\epsilon_0)$, for $$p\in\left\{0,1,...,\max\{\max_{1\le \ell\le \textbf{D}-1,\lambda\in \textbf{I}_{\ell}}\textbf{d}_{\lambda,\ell},\textbf{d}_{\textbf{D}}\}\right\}.$$
Then, there exists a unique formal solution power series $F(T,m,\epsilon)=\sum_{n\ge0}F_{n}(m,\epsilon)T^n$,
where $m\mapsto F_{n}(m,\epsilon)$ is an element of $E_{(\beta,\mu)}$ which depend holomorphically on $\epsilon\in D(0,\epsilon_0)$.
\end{prop}
\begin{proof}
By plugging the formal power series $F(T,m,\epsilon)$ into equation (\ref{e149b}), one gets that its coefficients satisfy the following recursion formula in order to be a formal solution of (\ref{e149b}):
\begin{align*}
&F_{n}(m,\epsilon)=\frac{\textbf{R}_{\textbf{D}}(im)}{\textbf{Q}(im)}q^{(\frac{\textbf{d}_{\textbf{D}}}{k_1}+1)(n-\textbf{d}_{\textbf{D}})-n}F_{n-\textbf{d}_{\textbf{D}}}(m,\epsilon)\\
&+\frac{1}{\textbf{Q}(im)}\sum_{\ell=1}^{\textbf{D}-1}\sum_{\lambda\in \textbf{I}_{\ell}}\epsilon^{\boldsymbol{\Delta}_{\lambda,\ell}-\textbf{d}_{\lambda,\ell}}\frac{1}{(2\pi)^{1/2}}\int_{-\infty}^{\infty}\textbf{C}_{\lambda,\ell}(m-m_1,\epsilon)\textbf{R}_{\ell}(im_1)q^{\boldsymbol{\delta}_{n}(n-\textbf{d}_{\lambda,\ell})-n}F_{n-\textbf{d}_{\lambda,\ell}}dm_1\\
&\hspace{8cm}+\frac{1}{\textbf{Q}(im)}\textbf{F}_n,
\end{align*}
for every $n\ge\max\{\max_{1\le \ell\le \textbf{D}-1,\lambda\in \textbf{I}_{\ell}}\textbf{d}_{\lambda,\ell},\textbf{d}_{\textbf{D}}\}$.
In view of Lemma~\ref{lema123} and Proposition~\ref{prop128}, the coefficients $F_n$ belong to $E_{(\beta,\mu)}$, and depend holomorphically on $\epsilon\in D(0,\epsilon_0)$.
\end{proof}

We multiply equation (\ref{e149b}) by $T^{k_1}$ and apply the formal $q-$Borel transform of order $k_1$ on both sides of the resulting equation. Regarding Proposition~\ref{prop157}, we arrive at

\begin{align}&\textbf{Q}(im)\frac{\tau^{k_1}}{(q^{1/k_1})^{k_1(k_1-1)/2}}\psi_{k_1}(\tau,m,\epsilon)=\frac{\tau^{\textbf{d}_\textbf{D}+k_1}}{(q^{1/k_1})^{(\textbf{d}_{\textbf{D}}+k_1)(\textbf{d}_{\textbf{D}}+k_1-1)/2}}\textbf{R}_{\textbf{D}}(im)\psi_{k_1}(\tau,m,\epsilon)\nonumber\\
+&\sum_{\ell=1}^{D-1}\left(\sum_{\lambda\in \textbf{I}_{\ell}}\frac{\epsilon^{\boldsymbol{\Delta}_{\lambda,\ell}-\textbf{d}_{\lambda,\ell}}\tau^{\textbf{d}_{\lambda,\ell}+k_1}}{(q^{1/k_1})^{(\textbf{d}_{\lambda,\ell}+k_1)(\textbf{d}_{\lambda,\ell}+k_1-1)/2}}\sigma_q^{\boldsymbol{\delta}_{\ell}-\frac{\textbf{d}_{\lambda,\ell}}{k_1}-1}\frac{1}{(2\pi)^{1/2}}(\textbf{C}_{\lambda,\ell}(m,\epsilon)\ast^{\textbf{R}_{\ell}}\psi_{k_1}(\tau,m,\epsilon))\right)\nonumber\\
&\hspace{10cm}+\frac{\tau^{k_1}}{(q^{1/k_1})^{k_{1}(k_1-1)/2}}\boldsymbol{\psi}_{k_1}(\tau,m,\epsilon),\label{e224b}
\end{align}
where 
$$\psi_{k_1}(\tau,m,\epsilon)=\sum_{n\ge0}F_n(m,\epsilon)\frac{\tau^n}{(q^{1/k_1})^{n(n-1)/2}},\qquad \boldsymbol{\psi}_{k_1}(\tau,m,\epsilon)=\sum_{n\ge0}\textbf{F}_n(m,\epsilon)\frac{\tau^n}{(q^{1/k_1})^{n(n-1)/2}}.$$

We now make the additional assumption that there exists an unbounded sector
$$S_{\textbf{Q},\textbf{R}_{\textbf{D}}}=\left\{z\in\C:|z|\ge r_{\textbf{Q},\textbf{R}_{\textbf{D}}},|\arg(z)-d_{\textbf{Q},\textbf{R}_{\textbf{D}}}|\le \boldsymbol{\eta}_{\textbf{Q},\textbf{R}_{\textbf{D}}}\right\},$$
for some $\boldsymbol{\eta}_{\textbf{Q},\textbf{R}_{\textbf{D}}}>0$ and $d_{\textbf{Q},\textbf{R}_{\textbf{D}}}\in\R$ such that 
$$\frac{\textbf{Q}(im)}{\textbf{R}_{\textbf{D}}(im)}\in S_{\textbf{Q},\textbf{R}_{\textbf{D}}},\quad m\in\R.$$
We consider the polynomial 
$$\textbf{P}_{m}(\tau)=\frac{\textbf{Q}(im)}{(q^{1/k_1})^{k_1(k_1-1)/2}}-\frac{\textbf{R}_{\textbf{D}}(im)}{(q^{1/k_1})^{\frac{(\textbf{d}_{\textbf{D}}+k_1)(\textbf{d}_{\textbf{D}}+k_1-1)}{2}}}\tau^{\textbf{d}_\textbf{D}},$$
and factorize it 
$$\textbf{P}_{m}(\tau)=-\frac{\textbf{R}_{\textbf{D}}(im)}{(q^{1/k_1})^{\frac{(\textbf{d}_{\textbf{D}}+k_1)(\textbf{d}_{\textbf{D}}+k_1-1)}{2}}}\prod_{l=0}^{\textbf{d}_{\textbf{D}}-1}(\tau-\textbf{q}_l(m)),$$
with
\begin{align*}
\textbf{q}_{l}(m)&=\left(\frac{|\textbf{Q}(im)|}{|\textbf{R}_{\textbf{D}}(im)|}(q^{1/k_1})^{\frac{(\textbf{d}_{\textbf{D}}+k_1)(\textbf{d}_{\textbf{D}}+k_1-1)-k_1(k_1-1)}{2}}\right)^{1/\textbf{d}_{\textbf{D}}}\nonumber \\
&\times\exp\left(\sqrt{-1}\left(\frac{1}{\textbf{d}_{\textbf{D}}}\arg\left(\frac{\textbf{Q}(im)}{\textbf{R}_{\textbf{D}}(im)}(q^{1/k_1})^{\frac{(\textbf{d}_{\textbf{D}}+k_1)(\textbf{d}_{\textbf{D}}+k_1-1)-k_1(k_1-1)}{2}}\right)+\frac{2\pi l}{\textbf{d}_{\textbf{D}}}\right)\right),
\end{align*}
for every $0\le l\le \textbf{d}_{\textbf{D}}-1$.

Moreover, we choose the family of unbouded sectors $U_{\mathfrak{d}_{p}}$, $0\le p\le \varsigma$, with vertex at 0, a small closed disc $\overline{D}(0,\rho)$, established in Definition~\ref{def818}, and choose $S_{\textbf{Q},\textbf{R}_{\textbf{D}}}$ satisfying
\begin{enumerate}
\item[1)] There exists $\textbf{M}_1>0$ such that
$$|\tau-\textbf{q}_{l}(m)|\ge \textbf{M}_{1}(1+|\tau|),$$
for every $0\le l\le \textbf{d}_{\textbf{D}}-1$, all $m\in\R$ and  $\tau\in U_{\mathfrak{d}_p}\cup \overline{D}(0,\rho)$, $0\le p\le \varsigma -1$. 
\item[2)] There exists $\textbf{M}_2>0$ and $l_0\in\{0,...,\textbf{d}_{\textbf{D}}-1\}$ such that
$$|\tau-\textbf{q}_{\ell_0}(m)|\ge \textbf{M}_2|\textbf{q}_{l_0}(m)|,$$
for every $m\in\R$ and $\tau\in \overline{D}(0,\rho)\cup U_{\mathfrak{d}_{p}}$, and all $0\le p\le \varsigma-1$.
\end{enumerate}

One can follow analogous steps as in (\ref{e529}) to conclude the existence of a constant $C_{\textbf{P}}>0$ such that
\begin{align}
|\textbf{P}_{m}(\tau)|\ge C_{\textbf{P}}(r_{\textbf{Q},\textbf{R}_{\textbf{D}}})^{1/\textbf{d}_{\textbf{D}}}|\textbf{R}_{\textbf{D}}(im)|(1+|\tau|)^{\textbf{d}_{\textbf{D}}-1},\label{e529b}
\end{align}
for every $\tau\in \overline{D}(0,\rho)\cup U_{\mathfrak{d}_{p}}$, for all $0\le p\le \varsigma-1$, and $m\in\R$.

\begin{prop}
Let $\boldsymbol{\varpi}>0$. Under the hypotheses made at the beginning of Section~\ref{seccion7} on the elements involved in the construction of equation (\ref{e224b}) and (\ref{e1336}), under assumptions 1) and 2) above, and if there exist small enough positive constants $\boldsymbol{\zeta}_{\boldsymbol{\psi}_{k_1}},\boldsymbol{\zeta}_{\lambda,\ell}$ for $1\le \ell\le \textbf{D}-1$ and $\lambda\in \textbf{I}_{\ell}$ such that 
$$\tilde{\textbf{C}}_{\lambda,\ell}\le\boldsymbol{\zeta}_{\lambda,\ell}\qquad \textbf{C}_{\boldsymbol{\psi}_{k_1}}\le\boldsymbol{\zeta}_{\boldsymbol{\psi}_{k_1}},$$
then the equation (\ref{e224b}) admits a unique solution $\psi_{k_1}^{\mathfrak{d}_{p}}(\tau,m,\epsilon)$ in the space $\hbox{Exp}^{\mathfrak{d}_{p}}_{(k_1,\beta,\mu,\alpha)}$, for some $\nu\in\R$. Moreover, $\bigl\|\psi_{k_1}^{\mathfrak{d}_{p}}(\tau,m,\epsilon)\bigr\|_{(k_1,\beta,\mu,\alpha)}\le \boldsymbol{\varpi}$ for every $\epsilon\in D(0,\epsilon_0)$, and this function is holomorphic with respect to $\epsilon$ in $D(0,\epsilon_0)$.
\end{prop}
\begin{proof}
The proof of this result follows analogous steps as those in the demonstration of Proposition~\ref{prop571}, so we omit it.
\end{proof}

Let $F^{\mathfrak{d}_{p}}(T,m,\epsilon)$ be defined in (\ref{e875b}).
%{\color{red}{ One can check that $$F^{\mathfrak{d}_{p}}(T,m,\epsilon)=\mathcal{L}_{q,1/k_1}^{\mathfrak{d}_{p}}(\tau\mapsto\psi_{k_1}^{\mathfrak{d}_{p}}(\tau,m,\epsilon))(T)$$
%for every $0\le p\le \varsigma-1$ and $T$ in an appropriate finite sector with vertex at the origin. We do not enter into details for the proof of this last claim because it follows analogous arguments as those in several results in the present work. Indeed, by means of the properties held by the formal $q-$Borel transform (see Proposition~\ref{prop157}) and by the $q-$Laplace transform (see Proposition~\ref{prop259}), and in view of (\ref{e1458}), one can prove that both functions solve the same equation, whilst a fixed point argument involving appropriate Banach spaces guarantees uniqueness of the solution, so they coincide.
%}}

We put 
$\textbf{c}_{\lambda,\ell}(z,\epsilon)=\mathcal{F}^{-1}(m\mapsto \textbf{C}_{\lambda,\ell}(m,\epsilon)$, which is a holomorphic function defined in $H_{\beta'}\times D(0,\epsilon_0)$, and  $$\textbf{f}(t,z,\epsilon)=\sum_{n\ge0}\mathcal{F}^{-1}(m\mapsto \textbf{F}_n(m,\epsilon))(z)(t\epsilon)^n,$$
which is a holomorphic function defined in $D(0,r)\times H_{\beta'}\times D(0,\epsilon_0)$, for small enough $r>0$.

One can apply $\mathcal{F}^{-1}$ to (\ref{e149b}) to deduce that $f^{\mathfrak{d}_{p}}$  is a solution of the equation

\begin{align}
&\textbf{Q}(\partial_z)\sigma_qf^{\mathfrak{d}_p}(t,z,\epsilon)=(\epsilon t)^{\textbf{d}_{\textbf{D}}}\sigma_{q}^{\frac{\textbf{d}_{\textbf{D}}}{k_1}+1}\textbf{R}_{\textbf{D}}(\partial_{z})f^{\mathfrak{d}_p}(t,z,\epsilon)\nonumber\\
&\hspace{3cm}+\sum_{\ell=1}^{\textbf{D}-1}\left(\sum_{\lambda\in \textbf{I}_{\ell}}t^{\textbf{d}_{\lambda,\ell}}\epsilon^{\boldsymbol{\Delta}_{\lambda,\ell}}\sigma_{q}^{\boldsymbol{\delta}_\ell}\textbf{c}_{\lambda,\ell}(z,\epsilon)\textbf{R}_{\ell}(\partial_z)f^{\mathfrak{d}_p}(t,z,\epsilon)\right)+\sigma_{q}\textbf{f}(t,z,\epsilon).\label{e771h}
\end{align}

\begin{theo}\label{teofin}
We assume the hypotheses of Theorem~\ref{teo872} and those stated in this section hold. We denote
\begin{align*}
P(t,z,\epsilon,\partial_z,\sigma_q)&=Q(\partial_z)\sigma_q-(\epsilon t)^{d_D}\sigma_{q}^{\frac{d_D}{k_2}+1}R_{D}(\partial_{z})-\sum_{\ell=1}^{D-1}\left(\sum_{\lambda\in I_{\ell}}t^{d_{\lambda,\ell}}\epsilon^{\Delta_{\lambda,\ell}}\sigma_{q}^{\delta_\ell}c_{\lambda,\ell}(z,\epsilon)R_{\ell}(\partial_z)\right),\\
\textbf{P}(t,z,\epsilon,\partial_z,\sigma_q)&=\textbf{Q}(\partial_z)\sigma_q-(\epsilon t)^{\textbf{d}_{\textbf{D}}}\sigma_{q}^{\frac{\textbf{d}_{\textbf{D}}}{k_1}+1}\textbf{R}_{\textbf{D}}(\partial_{z})-\sum_{\ell=1}^{\textbf{D}-1}\left(\sum_{\lambda\in \textbf{I}_{\ell}}t^{\textbf{d}_{\lambda,\ell}}\epsilon^{\boldsymbol{\Delta}_{\lambda,\ell}}\sigma_{q}^{\boldsymbol{\delta}_\ell}\textbf{c}_{\lambda,\ell}(z,\epsilon)\textbf{R}_{\ell}(\partial_z)\right).
\end{align*}
Then, the functions $u^{\mathfrak{d}_{p}}(t,z,\epsilon)$ constructed in Theorem~\ref{teo872} solve the problem
\begin{equation}\label{e1450}
\textbf{P}(t,z,\epsilon,\partial_z,\sigma_q)\sigma_q^{-1}P(t,z,\epsilon,\partial_z,\sigma_q)u^{\mathfrak{d}_{p}}(t,z,\epsilon)=\sigma_q\textbf{f}(t,z,\epsilon),
\end{equation}
whose coefficients and forcing term $\textbf{f}$ are analytic functions on $D(0,r_{\mathcal{T}})\times H_{\beta'}\times D(0,\epsilon_0)$. Moreover, the formal power series $\hat{u}(t,z,\epsilon)$, constructed in Theorem~\ref{teo1281} formally solves equation (\ref{e1450}).
\end{theo}
\begin{proof}
For the first part of the proof, one can check that $F^{\mathfrak{d}_{p}}(T,m,\epsilon)$, as defined in (\ref{e875b}) is an analytic solution of equation (\ref{e149b}). This assertion comes from the equality (\ref{e1458}) and the application of the properties of the $q-$Laplace transform, stated in Proposition~\ref{prop259} in the same manner as it has been done throughout the work, so we omit the details at this point. Then, the first part of the result is straight from (\ref{e771h}) and Theorem~\ref{teo872}. In order to prove that $\hat{u}(t,z,\epsilon)$ provides a formal solution of equation (\ref{e1450}) one takes into account that $\hat{u}(t,z,\epsilon)$ formally solves equation (\ref{e771}) and  that $\hat{f}(t,z,\epsilon)$, constructed in Lemma~\ref{lema1271} formally solves equation
$$\textbf{P}(t,z,\epsilon,\partial_z,\sigma_q)\hat{f}(t,z,\epsilon)=\sigma_q\textbf{f}(t,z,\epsilon).$$
\end{proof}

\end{document}